\documentclass[12pt]{amsart}
\usepackage{amscd,amssymb,amsthm,amsmath,amssymb,textcomp,supertabular,longtable,enumerate,rotating,mathrsfs}
\usepackage[matrix,arrow,curve]{xy}
\usepackage{pdflscape}

\newtheorem{theorem}[equation]{Theorem}

\newtheorem{proposition}[equation]{Proposition}
\newtheorem{lemma}[equation]{Lemma}
\newtheorem{corollary}[equation]{Corollary}

\theoremstyle{definition}
\newtheorem{example}[equation]{Example}

\theoremstyle{remark}
\newtheorem{remark}[equation]{Remark}

\makeatletter\@addtoreset{equation}{section} \makeatother

\newcommand{\mumu}{\boldsymbol{\mu}}
\def\SS {\mathfrak{S}}
\def\A {\mathfrak{A}}

\author{Ivan Cheltsov and Constantin Shramov}

\title{Finite collineation groups and birational rigidity}

\address{\emph{Ivan Cheltsov}
\newline
\textnormal{School of Mathematics, The University of Edinburgh,  Edinburgh EH9 3JZ, UK.}
\newline
\textnormal{National Research University Higher School of Economics, Laboratory of Algebraic Geometry, NRU HSE, 6 Usacheva str., Moscow, 117312, Russia.
}
\newline
\textnormal{\texttt{I.Cheltsov@ed.ac.uk}}}

\address{\emph{Constantin Shramov}
\newline
\textnormal{Steklov Mathematical Institute of RAS,
8 Gubkina street, Moscow 119991, Russia.
}
\newline
\textnormal{National Research University Higher School of Economics, Laboratory of Algebraic Geometry, NRU HSE, 6 Usacheva str., Moscow, 117312, Russia.
}
\newline
\textnormal{\texttt{costya.shramov@gmail.com}}}

\begin{document}

\begin{abstract}
We classify finite groups $G$ in $\mathrm{PGL}_{4}(\mathbb{C})$ such that $\mathbb{P}^3$ is $G$-birationally rigid.
\end{abstract}

\maketitle

All varieties are assumed to be projective and defined over the field of complex numbers.

\section{Introduction}
\label{section:intro}

The projective space $\mathbb{P}^3$ is one of the most basic objects in geometry.
Its biregular and birational symmetries have attracted attention of many mathematicians.
Among them was Hans Frederick Blichfeldt, who classified in \cite{Blichfeldt1905}
all finite subgroups of
$$
\mathrm{Aut}\big(\mathbb{P}^3\big)\cong\mathrm{PGL}_4\big(\mathbb{C}\big)
$$
back in 1905.
In 1917, he wrote his \emph{magnum opus} \cite{Blichfeldt1917},
which gave a detailed coverage of classification of finite subgroups in $\mathrm{PGL}_3(\mathbb{C})$ and $\mathrm{PGL}_4(\mathbb{C})$.
This book became a standard reference for the theory of finite collineation groups.
We want to celebrate its centennial by enriching this classical theory with modern technique of birational geometry.

Finite subgroups of $\mathrm{PGL}_4(\mathbb{C})$ are traditionally divided into two major types: transitive and intransitive.
In geometric language, transitive groups do not fix any point in $\mathbb{P}^3$ and do not leave any line invariant.
Intransitive groups are those that are not transitive.
Transitive groups are further subdivided into imprimitive and primitive groups.
Imprimitive groups either leave a union of two skew lines invariant
or have an orbit of length $4$ (or both).
All other finite subgroups in $\mathrm{PGL}_4(\mathbb{C})$ are said to be primitive.

The structure of the group $\mathrm{Bir}(\mathbb{P}^3)$, known as the space Cremona group, is way more complicated than that of $\mathrm{PGL}_4(\mathbb{C})$.
In particular, the classification of all finite subgroups in $\mathrm{Bir}(\mathbb{P}^3)$ seems to be out of reach at the moment.
There are just several sporadic known results in this direction.
These include the classification of all finite \emph{simple} subgroups in $\mathrm{Bir}(\mathbb{P}^3)$ in \cite{Prokhorov2012},
and some boundedness properties in \cite{ProkhorovShramov2016}.
The goal of this paper is to understand how the finite subgroups of $\mathrm{PGL}_4(\mathbb{C})$
described by Blichfeldt behave as subgroups of the larger group $\mathrm{Bir}(\mathbb{P}^3)$.

Given a finite subgroup $G$ in $\mathrm{PGL}_4(\mathbb{C})$,
the most natural birational problem is to describe all $G$-birational maps from $\mathbb{P}^3$
to other three-dimensional $G$-varieties.
Minimal Model Program implies that to solve this problem it is enough to describe
$G$-birational maps from $\mathbb{P}^3$ to \emph{$G$-Mori fibre spaces} (see \cite[Definition~1.1.5]{CheltsovShramov}).
In this paper we aim to describe all finite subgroups $G\subset\mathrm{PGL}_4(\mathbb{C})$
such that the projective space $\mathbb{P}^3$ cannot be $G$-birationally transformed into other $G$-Mori fiber spaces.
In this case, we say that $\mathbb{P}^3$ is $G$-\emph{birationally rigid} (see \cite[Definition~3.1.1]{CheltsovShramov}).
This simply means that the following three conditions are satisfied:
\begin{enumerate}
\item There is no $G$-rational map $\mathbb{P}^3\dasharrow S$ whose general fiber is a rational curve.
\item There is no $G$-rational map $\mathbb{P}^3\dasharrow\mathbb{P}^1$ whose general fiber is a rational surface.
\item There is no $G$-birational map $\mathbb{P}^3\dasharrow X$ such that $X$ is a Fano threefold with terminal singularities,
the $G$-invariant class group of the threefold $X$ is of rank $1$
and $X$ is not $G$-isomorphic to $\mathbb{P}^3$ (equipped with the same action of the group $G$).
\end{enumerate}

The main result of our paper is the following theorem.

\begin{theorem}
\label{theorem:main}
Let $G$ be a finite subgroup in $\mathrm{PGL}_4(\mathbb{C})$.
Then $\mathbb{P}^3$ is $G$-birationally rigid if and only if $G$ is a primitive group
that is not isomorphic to $\mathfrak{A}_5$ or $\mathfrak{S}_5$.
\end{theorem}

It is a three-dimensional analogue of the following recent result of Sakovics.

\begin{theorem}[{\cite{Sakovics}}]
\label{theorem:Sakovich}
Let $G$ be a finite subgroup in $\mathrm{PGL}_3(\mathbb{C})$.
Then $\mathbb{P}^2$ is $G$-birationally rigid if and only if $G$ is a transitive subgroup
that is not isomorphic to $\mathfrak{A}_4$ or $\mathfrak{S}_4$.
\end{theorem}

Let us describe the plan of the proof of Theorem~\ref{theorem:main}.
First, we show that if $G$ is a finite intransitive or imprimitive subgroup of $\mathrm{PGL}_4(\mathbb{C})$,
then $\mathbb{P}^3$ is not $G$-birationally rigid.
This is done in Corollary~\ref{corollary:non-primitive}.
Thus, we are left with the case of finite primitive groups.
Up to conjugation, there are just $30$ such groups.
They are listed in \cite[Chapter~VII]{Blichfeldt1917}.
For the reader's convenience,
in Appendix~\ref{section:diagram}
we present these groups together with
(some) inclusions between them, including those that are
necessary for the proof of Theorem~\ref{theorem:main}.
However, we point out that we will not need the whole
classification for the proof, but only the information about
the few smallest groups.

We prove that $\mathbb{P}^3$ is not $G$-birationally rigid
if $G$ is a primitive subgroup  isomorphic to $\mathfrak{A}_5$ or $\mathfrak{S}_5$,
see Corollary~\ref{corollary:A5-S5}.
This completes the proof of the ``only if'' part of Theorem~\ref{theorem:main}.
Thus, to prove Theorem~\ref{theorem:main}, we have to show that
$\mathbb{P}^3$ is $G$-birationally rigid for the remaining finite primitive subgroups $G$ in~$\mathrm{PGL}_4(\mathbb{C})$.
These can be described as follows (as usual, we denote by $\mumu_n$ the cyclic group of order~$n$).
\begin{itemize}
\item Primitive groups that leave a quadric surface invariant.

\item Primitive subgroups in $\mathrm{PGL}_4(\mathbb{C})$ that contain the imprimitive group $\mumu_2^4$ as a normal subgroup.
Some of them, like the group $\mathfrak{A}_4\times\mathfrak{A}_4$, also leave a quadric surface invariant,
so that this description is not exclusive (see the discussion in \cite[\S125]{Blichfeldt1917}).

\item The Klein simple group $\mathrm{PSL}_2(\mathbf{F}_7)$.

\item The groups $\mathfrak{A}_6$, $\mathfrak{S}_6$, and $\mathfrak{A}_7$.

\item The simple group of order $25920$, which is isomorphic to $\mathrm{PSp}_{4}(\mathbf{F}_3)$.
\end{itemize}

Every primitive subgroup in $\mathrm{PGL}_4(\mathbb{C})$ that leaves a quadric surface invariant
(and is not isomorphic to $\mathfrak{A}_5$ and $\mathfrak{S}_5$)
contains a subgroup isomorphic to $\mathfrak{A}_4\times\mathfrak{A}_4$.
Likewise, every primitive subgroup that contains the imprimitive group $\mumu_2^4$ and
does not leave any quadric surface invariant has a subgroup isomorphic to $\mumu_2^4\rtimes\mumu_5$.
Both of them contain the imprimitive subgroup $\mumu_2^4$ as a normal subgroup,
so that they have similar group-theoretical properties.
In \S\ref{section:rigidity}, we prove that $\mathbb{P}^3$ is $G$-birationally rigid if
$G$ is isomorphic either to $\mathfrak{A}_4\times\mathfrak{A}_4$ or to~$\mumu_2^4\rtimes\mumu_5$,
see Corollaries~\ref{corollary:144} and \ref{corollary:80-160}.
If $G\cong\mathfrak{A}_6$, then $\mathbb{P}^3$ is $G$-birationally rigid by \cite[Theorem~1.24]{ChSh09b}.
If $G\cong\mathrm{PSL}_2(\mathbf{F}_7)$, then $\mathbb{P}^3$ is $G$-birationally rigid by \cite[Theorem~1.9]{ChSh10a}.

Therefore, to prove Theorem~\ref{theorem:main}, it remains to consider primitive subgroups in
$\mathrm{PGL}_4(\mathbb{C})$ that contain one of the following three groups as a proper subgroup:
$\mathfrak{A}_4\times\mathfrak{A}_4$, $\mumu_2^4\rtimes\mumu_5$, or~$\mathfrak{A}_6$.
We warn the reader that in general one cannot automatically conclude that $\mathbb{P}^3$ is \mbox{$G$-birationally} rigid
if it is $H$-birationally rigid for a subgroup $H\subset G$, cf. Remark~\ref{remark:subgroup-rigidity}.
However, in our particular cases, this is indeed true.
We show this in \S\ref{section:rigidity}, see
Corollary~\ref{corollary:144} for groups containing $\mathfrak{A}_4\times\mathfrak{A}_4$,
Corollary~\ref{corollary:80-160} and Lemma~\ref{lemma:rigidity-80} for groups containing~\mbox{$\mumu_2^4\rtimes\mumu_5$},
and Lemma~\ref{lemma:rigidity-A6} for groups containing $\mathfrak{A}_6$.

In addition to Theorem~\ref{theorem:main}, we also provide a classification of finite
subgroups~\mbox{$G\subset\mathrm{PGL}_4(\mathbb{C})$} such that~$\mathbb{P}^3$ is $G$-birationally super-rigid.
Recall from  \cite[Definition~3.1.1]{CheltsovShramov} that~$\mathbb{P}^3$ is said to be $G$-birationally super-rigid
if it is $G$-birationally rigid and
$$
\mathrm{Bir}^{G}\big(\mathbb{P}^3\big)=\mathrm{Aut}^{G}\big(\mathbb{P}^3\big),
$$
where $\mathrm{Bir}^{G}(\mathbb{P}^3)$ and $\mathrm{Aut}^{G}(\mathbb{P}^3)$ are the normalizers of
the group $G$ in $\mathrm{Bir}(\mathbb{P}^3)$ and~\mbox{$\mathrm{Aut}(\mathbb{P}^3)$}, respectively.
We prove the following

\begin{theorem}
\label{theorem:super-main}
Let $G$ be a finite subgroup in $\mathrm{PGL}_4(\mathbb{C})$.
Then $\mathbb{P}^3$ is $G$-birationally super-rigid if and only if $G$ is a primitive group
that is isomorphic neither to $\mathfrak{A}_5$, nor to $\mathfrak{S}_5$, nor to $\mathrm{PSL}_{2}(\mathbf{F}_7)$,
nor to~$\mathfrak{A}_6$, nor to $\mumu_2^4\rtimes\mumu_5$, nor to $\mumu_2^4\rtimes\mathrm{D}_{10}$.
\end{theorem}

Looking at Blichfeldt's classification of finite primitive subgroups in  $\mathrm{PGL}_4(\mathbb{C})$,
we conclude that $\mathbb{P}^3$ is $G$-birationally rigid for $23$ subgroups $G$,
and is $G$-birationally super-rigid for $19$ subgroups $G$.

\smallskip

Let us describe the structure of the paper.
In~\S\ref{section:non-rigidity}, we show that if $G$ is a finite intransitive or imprimitive subgroup of $\mathrm{PGL}_4(\mathbb{C})$,
or $G$ is a primitive subgroup isomorphic to~$\mathfrak{A}_5$ or $\mathfrak{S}_5$, then $\mathbb{P}^3$ is not $G$-birationally rigid.
In~\S\ref{section:Heisenberg}, we present basic facts about the imprimitive group $\mumu_2^4$.
In~\S\ref{section:144}, we study equivariant geometry of $\mathbb{P}^3$ acted on by a primitive group~\mbox{$\mathfrak{A}_4\times\mathfrak{A}_4$}.
Similarly, in~\S\ref{section:80-160-320}, we do the same for the group~$\mumu_2^4\rtimes\mumu_5$.
In particular, we construct a $G$-commutative diagram
$$
\xymatrix{
&X\ar@{->}[ld]_\pi\ar@{->}[dr]^\alpha\ar@{-->}[rr]^\chi && Y\ar@{->}[rd]^\phi\ar@{->}[dl]_\beta &\\
\mathbb{P}^3 && V_8 && \mathbb{P}^3}
$$
Here $G$ is a primitive subgroup $\mumu_2^4\rtimes\mumu_5$ in $\mathrm{PGL}_4(\mathbb{C})$,
the morphisms $\pi$ and $\phi$ are blow ups of $G$-invariant smooth irreducible curves of degree $8$ and genus $5$,
the variety $V_8$ is a complete intersection of three quadrics in~$\mathbb{P}^6$ that has isolated terminal Gorenstein singularities,
the morphisms $\alpha$ and $\beta$ are small birational contractions, and $\chi$ is a composition of flops.
Finally, in~\S\ref{section:rigidity}, we prove our Theorem~\ref{theorem:main} using the results obtained
in~\S\ref{section:144} and \S\ref{section:80-160-320}.
In Appendix~\ref{section:diagram},
we present the finite primitive subgroups
of~\mbox{$\mathrm{PGL}_4(\mathbb{C})$}
from \cite[Chapter~VII]{Blichfeldt1917}
together with necessary inclusions between.

\bigskip

Throughout the paper we denote by $\mumu_n$ the cyclic group of order~$n$.
If $G$ is a group and~$V$ is a variety acted on by $G$,
we say that $V$ is \emph{$G$-irreducible} if $G$ acts
transitively on the set of irreducible components of~$V$.
If $Z$ is a subvariety of $V$,
we will sometimes abuse terminology and refer to the union of the images $g(Z)$, $g\in G$, as the \emph{$G$-orbit} of~$Z$.
By a $G$-commutative diagram we will mean a
commutative diagram of $G$-rational maps,
see~\mbox{\cite[p.~xix]{CheltsovShramov}}.

\bigskip

\textbf{Acknowledgements.}
The authors are grateful to Igor Dolgachev,
David Eklund, Takeru Fukuoka, Alexander Kuznetsov, Dmitrii Pasechnik,
Yuri Prokhorov, and Dmitrijs Sakovics for useful discussions.
Special thanks go to the referee for his careful reading of our paper.
The first draft of this paper was written during Ivan Cheltsov's stay at the Max Planck Institute for Mathematics in 2017.
He would like to thank the institute for the excellent working condition.
Both authors were supported by the
Russian Academic Excellence Project~\mbox{``5-100''}.
Ivan Cheltsov was supported by the Royal Society grant No. IES\slash R1\slash 180205.
Constantin Shramov was supported by
RFBR grants 15-01-02164 and 15-01-02158, by Young Russian Mathematics award,
and by the Foundation for the
Advancement of Theoretical Physics and Mathematics ``BASIS''.

\section{Birational non-rigidity}
\label{section:non-rigidity}

In this section we prove that if $\mathbb{P}^3$ is $G$-birationally rigid,
then $G$ is a primitive group that is not isomorphic to $\mathfrak{A}_5$ or $\mathfrak{S}_5$.
We start with

\begin{proposition}
\label{proposition:Fano-Enriques}
Suppose that $G$ is a transitive subgroup  in $\mathrm{PGL}_4(\mathbb{C})$
such that there exists a $G$-orbit in $\mathbb{P}^3$ of length $4$.
Then there exists a $G$-birational map $\mathbb{P}^3\dasharrow V_{24}$ such that
$V_{24}$ is a toric Fano threefold, one has~\mbox{$-K_{V_{24}}^3=24$},
the singular locus of $V_{24}$ consists of $8$ quotient singularities of type~\mbox{$\frac{1}{2}(1,1,1)$}.
In particular, the singularities of $V_{24}$ are terminal and $\mathbb{Q}$-factorial.
Moreover, one has $\mathrm{Pic}(V_{24})\cong\mathbb{Z}^3$.
Furthermore, if there is no $G$-invariant pair of lines in $\mathbb{P}^3$, then
$\mathrm{Pic}(V_{24})^G\cong\mathbb{Z}$.
\end{proposition}

\begin{proof}
By assumption, there exists a $G$-orbit in $\mathbb{P}^3$ of length $4$.
Denote the points of this $G$-orbit by $P_1$, $P_2$, $P_3$ and $P_4$.
Without loss of generality, we may assume that
$$
P_1=[1:0:0:0],\quad P_2=[0:1:0:0],\quad P_3=[0:0:1:0],\quad P_4=[0:0:0:1].
$$
For every $1\leqslant i<j\leqslant 4$, let $L_{ij}$ be the line in $\mathbb{P}^3$ that passes though $P_i$ and~$P_j$.
Let $\mathcal{M}$ be the linear system consisting of all sextic surfaces in $\mathbb{P}^3$
that are singular along every line~$L_{ij}$.
Then the surfaces in $\mathcal{M}$ are given by
\begin{multline}
\label{equation:Enriques-surface}
a_{0}x^2y^2z^2+a_{1}x^2y^2w^2+a_{2}x^2z^2w^2+a_{3}y^2z^2w^2+\\
+a_4x^3yzw+a_5x^2y^2zw+a_6x^2yz^2w+a_7x^2yzw^2+a_8xy^3zw+a_9xy^2z^2w+\\
+a_{10}xy^2zw^2+a_{11}xyz^3w+a_{12}xyz^2w^2+a_{13}xyzw^3=0
\end{multline}
for $[a_0:\ldots:a_{13}]\in\mathbb{P}^{13}$.
The linear system $\mathcal{M}$ gives a rational map $\psi\colon\mathbb{P}^3\dasharrow\mathbb{P}^{13}$.
Note that $\psi$ is well-defined away from the lines $L_{12}$, $L_{13}$, $L_{14}$, $L_{23}$, $L_{24}$ and $L_{34}$.
Moreover, the image of the map $\psi$ is three-dimensional: one can express the ratios $\frac{x^2}{w^2}$
$\frac{y^2}{w^2}$ and $\frac{z^2}{w^2}$ as rational functions of the monomials $x^2y^2z^2$, $x^2y^2w^2$, $x^2z^2w^2$ and $y^2z^2w^2$.
Denote the image of $\psi$ by $V_{24}$.
By construction, the threefold $V_{24}$ is toric.

Let $\Pi_1$, $\Pi_2$, $\Pi_3$ and $\Pi_4$ be the planes in $\mathbb{P}^3$
that are given by $x=0$,  $y=0$, $z=0$ and~\mbox{$w=0$}, respectively.
Then $\psi$ contracts these planes to $4$ different points.

To resolve the indeterminacy of the rational map $\psi$,
let $\alpha\colon X\to\mathbb{P}^3$ be the blow up of the points $P_1$, $P_2$, $P_3$ and $P_4$.
Denote by $\widetilde{L}_{ij}$ the proper transform of the line $L_{ij}$ on the threefold $X$.
Let $\beta\colon W\to X$ be the blow up of the curves
$\widetilde{L}_{12}$, $\widetilde{L}_{13}$, $\widetilde{L}_{14}$, $\widetilde{L}_{23}$, $\widetilde{L}_{24}$ and $\widetilde{L}_{34}$.
Then there exists a $G$-commutative diagram:
\begin{equation}
\label{equation:toric}
\xymatrix{
&W\ar@{->}[ld]_{\beta}\ar@{-->}[rr]^{\rho}&& U\ar@{->}[dr]^{\gamma}&\\%
X\ar@{->}[d]_{\alpha}&&&&Y\ar@{->}[d]^{\delta}\ar@{->}[lllld]_{\phi}\\%
\mathbb{P}^3\ar@{-->}[rrrr]^{\psi}&&&& V_{24}}
\end{equation}
such that $\rho$ is a composition of $12$ Atiyah flops,
and $\gamma$, $\delta$ and $\phi$ are birational morphisms that we are about to describe.

To describe $\rho$, denote by
$\widetilde{\Pi}_i$, $1\leqslant i\leqslant 4$,
the proper transform on $X$ of the plane $\Pi_i$.
Then each $\widetilde{\Pi}_i$ is a smooth del Pezzo surface of degree $6$.
Denote by $E_i$ the exceptional divisor of the birational morphism $\alpha$ that is mapped to the point $P_i$.
Let $\ell_{ij}$ be the intersection curve $E_i\cap\widetilde{\Pi}_j$,
and let $\overline{\ell}_{ij}$ be its proper transform on the threefold $W$.
This gives us $12$ disjoint curves $\overline{\ell}_{ij}$,
$1\leqslant i\neq j\leqslant 4$.
The map $\rho$ is a composition of Atiyah flops in these curves.

To describe $\gamma$ and $\delta$, denote by $\overline{E}_i$ and $\overline{\Pi}_i$,
$1\leqslant i\leqslant 4$,
the proper transforms on $W$ of the surfaces $E_i$ and $\Pi_i$, respectively.
Then all of them are smooth del Pezzo surfaces of degree $6$.
Moreover, the curves flopped by $\rho$ are pairwise disjoint $(-1)$-curves on these surfaces.
Denote by $\widehat{E}_i$ and $\widehat{\Pi}_i$
the proper transforms on $U$ of the surfaces $\overline{E}_i$ and
$\overline{\Pi}_i$, respectively.
Then the surfaces $\widehat{E}_1$,
$\widehat{E}_2$, $\widehat{E}_3$, $\widehat{E}_4$,
$\widehat{\Pi}_1$, $\widehat{\Pi}_2$, $\widehat{\Pi}_3$ and $\widehat{\Pi}_4$
are pairwise disjoint, all of them are isomorphic to~$\mathbb{P}^2$, and their normal bundles in $U$ are $\mathcal{O}_{\mathbb{P}^2}(-2)$.
The map~$\gamma$ is the contraction of the surfaces
$\widehat{E}_1$, $\widehat{E}_2$, $\widehat{E}_3$ and $\widehat{E}_4$
to $4$ quotient singular points of type~\mbox{$\frac{1}{2}(1,1,1)$}.
Similarly, the map $\delta$ is the contraction of the surfaces
$\gamma(\widehat{\Pi}_1)$, $\gamma(\widehat{\Pi}_2)$, $\gamma(\widehat{\Pi}_3)$ and $\gamma(\widehat{\Pi}_4)$.
Thus, the threefold $V_{24}$ has exactly $8$ singular points,
and each of them is a quotient singularity of type $\frac{1}{2}(1,1,1)$.
In particular, the singularities of $V_{24}$ are terminal and $\mathbb{Q}$-factorial.

To describe $\phi$, denote by $\widehat{F}_{12}$,
$\widehat{F}_{13}$, $\widehat{F}_{14}$, $\widehat{F}_{23}$, $\widehat{F}_{24}$ and $\widehat{F}_{34}$
the proper transforms on $Y$ of the $\beta$-exceptional surfaces that are mapped to the curves
$\widetilde{L}_{12}$, $\widetilde{L}_{13}$, $\widetilde{L}_{14}$, $\widetilde{L}_{23}$, $\widetilde{L}_{24}$ and $\widetilde{L}_{34}$, respectively.
Then $\phi$ is the contraction of these six surfaces to the lines
$L_{12}$, $L_{13}$, $L_{14}$, $L_{23}$, $L_{24}$ and $L_{34}$, respectively.
By \cite[Theorem~4.9]{KoMo92}, the morphism $\phi$ is the \emph{symbolic blow up} of the reducible curve
$L_{12}+L_{13}+L_{14}+L_{23}+L_{24}+L_{34}$, see also \cite[\S6.1]{ProkhorovReid}.

By construction, the threefold $V_{24}$ is toric Fano threefold, one has~\mbox{$-K_{V_{24}}^3=24$}
and~\mbox{$\mathrm{Pic}(V_{24})\cong\mathbb{Z}^3$}.
If there is no $G$-invariant pair of lines in $\mathbb{P}^3$,
then the group $G$ permutes the six lines $L_{ij}$ transitively.
This gives~\mbox{$\mathrm{Pic}(V_{24})^G\cong\mathbb{Z}$}.
\end{proof}

In the notation of \cite{grdb}, the threefold $V_{24}$ in Proposition~\ref{proposition:Fano-Enriques} is the terminal toric Fano threefold~\textnumero{47}.
One can show that a general hyperplane section of the threefold~\mbox{$V_{24}\subset\mathbb{P}^{13}$} is a smooth Enriques surface.
Of course, this is well-known: the equation \eqref{equation:Enriques-surface} was found by Enriques to construct what is now known as an Enriques surface.
The variety~$V_{24}$ is one of the Fano--Enriques threefolds classified by Bayle in~\cite{Bayle}.
To be precise, the threefold~$V_{24}$ is a quotient of $\mathbb{P}^1\times\mathbb{P}^1\times\mathbb{P}^1$
by an involution $\tau$ that acts as
$$
\Big(\big[x_1:y_1\big],\big[x_2:y_2\big],\big[x_3:y_3\big]\Big)\mapsto\Big(\big[y_1:x_1\big],\big[y_2:x_2\big],\big[y_3:x_3\big]\Big).
$$

\begin{corollary}
\label{corollary:non-primitive}
Suppose that $G$ is a finite intransitive or imprimitive subgroup
in~\mbox{$\mathrm{PGL}_4(\mathbb{C})$}.
Then $\mathbb{P}^3$ is not $G$-birationally rigid.
\end{corollary}

\begin{proof}
If $G$ fixes a point in $\mathbb{P}^3$,
then the linear projection from this point $\mathbb{P}^3\dasharrow\mathbb{P}^2$ is a $G$-rational dominant map
whose general fiber is $\mathbb{P}^1$.
Likewise, if there is a $G$-invariant line in $\mathbb{P}^3$,
then the linear projection from this line $\mathbb{P}^3\dasharrow\mathbb{P}^1$ is a $G$-rational dominant map
whose general fiber is $\mathbb{P}^2$.
Thus, if $G$ is intransitive, then $\mathbb{P}^3$ is not $G$-birationally rigid.

To complete the proof, we may assume that $G$ is imprimitive.
If there is a $G$-invariant union of two skew lines in $\mathbb{P}^3$,
then the blow up of $\mathbb{P}^3$ along them is a $G$-equivariant $\mathbb{P}^1$-bundle over $\mathbb{P}^1\times\mathbb{P}^1$.
Likewise, if there exists a $G$-orbit in $\mathbb{P}^3$ that consists of $4$ points, but there is no
$G$-invariant pair of skew lines,
then $\mathbb{P}^3$ is $G$-birational to a toric Fano threefold~$V_{24}$ with terminal $\mathbb{Q}$-factorial
singularities such that the $G$-invariant class group of $V_{24}$ is of rank~$1$ by Proposition~\ref{proposition:Fano-Enriques}.
Therefore, by definition, $\mathbb{P}^3$ is not $G$-birationally rigid in this case as well.
\end{proof}

Thus, we are left with the case of primitive
subgroups in $\mathrm{PGL}_{4}(\mathbb{C})$ isomorphic to $\mathfrak{A}_5$ or $\mathfrak{S}_5$.
Up to conjugation, there are two primitive subgroups in $\mathrm{PGL}_{4}(\mathbb{C})$ isomorphic to $\mathfrak{A}_5$.
One of them leaves a quadric surface invariant.
Its action on $\mathbb{P}^3$ comes from the irreducible four-dimensional representation of the icosahedral group.
Another one preserves a twisted cubic curve, so that
its action on $\mathbb{P}^3$ comes from an irreducible four-dimensional representation of the binary icosahedral group.
Similarly, there are two primitive subgroups in~\mbox{$\mathrm{PGL}_{4}(\mathbb{C})$} isomorphic to $\mathfrak{S}_5$.
One of them preserves a quadric surface, and its action on~$\mathbb{P}^3$ comes from an irreducible four-dimensional representation of the group~$\mathfrak{S}_5$.
Another one leaves invariant a pair of disjoint twisted cubic curves.
Its action on $\mathbb{P}^3$ comes from an irreducible four-dimensional representation of a central extension of the group~$\mathfrak{S}_5$.
In all these cases, the projective space $\mathbb{P}^3$ is not $G$-birationally rigid.

\begin{example}[{\cite[2.13.1]{Takeuchi}, \cite[\S2]{Prokhorov2010}, \cite[Example~1.2.8]{CheltsovShramov}}]
\label{example:P3-non-rigid-A5-V22}
Let $G$ be a primitive finite subgroup $\mathfrak{A}_5$ in $\mathrm{PGL}_{4}(\mathbb{C})$ such that there exists a $G$-invariant quadric surface in $\mathbb{P}^3$.
Then there is a unique $G$-invariant smooth cubic surface in $\mathbb{P}^3$,
which is known as the \emph{Clebsch cubic surface}.
Denote this cubic surface by $S$.
Then $S$ contains a $G$-invariant irreducible smooth rational sextic curve.
Denote it by $\mathcal{C}$.
The surface $S$ also contains a $G$-invariant curve that consists of six disjoint lines,
which are bi-tangents of the curve $\mathcal{C}$.
Denote this curve by $\mathcal{L}_{6}$.
Let $\pi\colon X\to\mathbb{P}^3$ be the blow-up of the rational sextic curve $\mathcal{C}$.
Then there exists a $G$-commutative diagram (that is, a commutative diagram of $G$-rational maps,
see~\mbox{\cite[p.~xix]{CheltsovShramov}} for terminology)
$$
\xymatrix{
&X\ar@{-->}[rr]^{\chi}\ar@{->}[ld]_{\pi}\ar@{->}[rd]^{\alpha}&& Y\ar@{->}[ld]_{\beta}\ar@{->}[rd]^{\phi}&\\%
\mathbb{P}^3&&V_{14}&&V_{22}}
$$ %
Here $V_{14}$ is a Fano threefold with six isolated ordinary double points such that  $-K_{V_{14}}^3=14$,
and $V_{22}$ is a smooth Fano threefold such that $-K_{V_{22}}^3=22$ and $\mathrm{Aut}(V_{22})\cong\mathrm{PSL}_2(\mathbb{C})$,
which is known as the Mukai--Umemura threefold (see~\cite{MukaiUmemura}).
The morphism $\alpha$ is the contraction of the proper transforms of the irreducible components of the curve $\mathcal{L}_6$,
the rational map $\chi$ is a composition of Atiyah flops in these $6$ curves,
the morphism $\beta$ is a flopping contraction,
and $\phi$ is a contraction of the proper transform of the cubic surface $S$ to a smooth point of the threefold $V_{22}$.
\end{example}

\begin{example}[{\cite[Application~1]{SzWi90}}]
\label{example:P3-non-rigid-A5-P1-bundle}
Let $G$ be a primitive finite subgroup $\mathfrak{A}_5$
in~\mbox{$\mathrm{PGL}_{4}(\mathbb{C})$}
such that $\mathbb{P}^3$ contains a $G$-invariant twisted cubic curve  $\mathcal{C}$.
Let $\pi\colon X\to\mathbb{P}^3$ be the blow-up of the curve $\mathcal{C}$.
Then there exists a $G$-commutative diagram
$$
\xymatrix{
&X\ar@{->}[ld]_{\pi}\ar@{->}[rd]^{\phi}&\\%
\mathbb{P}^3\ar@{-->}[rr]^{\rho}&&\mathbb{P}^2}
$$ %
Here $\phi$ is a $\mathbb{P}^1$-bundle whose fibers are proper transforms of secants of the curve~$\mathcal{C}$,
and the map $\rho$ is given by the linear system of quadric surfaces that contain~$\mathcal{C}$.
\end{example}

\begin{example}[{\cite[Proposition~4.7]{Prokhorov2010}}]
\label{example:P3-non-rigid-S5-P1-bundle-dP5}
Let $G$ be a primitive finite subgroup $\mathfrak{S}_5$
in~\mbox{$\mathrm{PGL}_{4}(\mathbb{C})$} such that
there exists a $G$-invariant quadric surface in $\mathbb{P}^3$.
Then $\mathbb{P}^3$  contains a $G$-orbit~$\Sigma_5$ of length $5$.
Let $\pi\colon X\to\mathbb{P}^3$ be the blow up of this orbit.
Then there exists a $G$-commutative diagram
$$
\xymatrix{
&X\ar@{-->}[rr]^{\chi}\ar@{->}[ld]_{\pi}\ar@{->}[rd]^{\alpha}&& Y\ar@{->}[ld]_{\beta}\ar@{->}[rd]^{\phi}&\\%
\mathbb{P}^3&&V_3&& S}
$$
Here $\alpha$ is the contraction of the proper transforms of the $10$ lines in $\mathbb{P}^3$
passing through pairs of points in $\Sigma_5$, the rational map $\chi$ is a composition of Atiyah flops
in these $10$ curves, the morphism $\beta$ is a flopping contraction, the morphism $\phi$ is a $\mathbb{P}^1$-bundle,
the variety $V_3$ is the Segre cubic hypersurface in $\mathbb{P}^4$,
and $S$ is the smooth del Pezzo surface of degree~$5$.
\end{example}

\begin{example}
\label{example:P3-non-rigid-S5-cubic}
Similarly to Example~\ref{example:P3-non-rigid-S5-P1-bundle-dP5},
let $G$ be a primitive finite subgroup $\mathfrak{S}_5$ in~\mbox{$\mathrm{PGL}_{4}(\mathbb{C})$} such that
there exists a $G$-invariant quadric surface in $\mathbb{P}^3$.
Then such quadric surface is smooth and unique. Denote it by $Q$.
There is a unique $G$-invariant smooth cubic surface in $\mathbb{P}^3$,
which is the Clebsch cubic surface.
The complete intersection of this cubic surface and the quadric surface $Q$ is a $G$-invariant irreducible smooth sextic curve of genus $4$,
which is known as the \emph{Bring curve}.
Let $\pi\colon X\to\mathbb{P}^3$ be the blow up of this curve.
Then there exists a $G$-commutative diagram
$$
\xymatrix{
&X\ar@{->}[ld]_{\pi}\ar@{->}[rd]^{\phi}&\\%
\mathbb{P}^3&&V_{3}\ar@{-->}[ll]^{\rho}}
$$ %
Here $V_3$ is a cubic hypersurface in $\mathbb{P}^4$ that has one isolated ordinary double point,
the morphism $\phi$ is a contraction of the proper transform of the quadric $Q$ to the singular point of the cubic $V_3$,
and $\rho$ is a linear projection from this point.
\end{example}

The case of a primitive subgroup $\mathfrak{S}_5$ in $\mathrm{PGL}_4(\mathbb{C})$
that does not leave invariant any quadric surfaces is a little bit more tricky.

\begin{proposition}
\label{proposition:S5-link}
Suppose that $G$ is a primitive subgroup in $\mathrm{PGL}_4(\mathbb{C})$ such that $G\cong\mathfrak{S}_5$,
and $G$ leaves invariant a pair of disjoint twisted cubic curves. Denote them by $C_1$ and $C_2$.
Let $\pi\colon X\to\mathbb{P}^3$ be the blow up of the curves $C_1$ and $C_2$.
Then there is a $G$-commutative diagram
$$
\xymatrix{
&&X\ar@{->}[lld]_\pi\ar@{->}[drr]^\alpha\ar@{-->}[rrrr]^\chi &&&& Y\ar@{->}[rrd]^\phi\ar@{->}[dll]_\beta &&\\
\mathbb{P}^3 &&&& V_{12} &&&&\mathbb{P}^1\times\mathbb{P}^1}
$$
Here $V_{12}$ is a divisor of bi-degree $(2,2)$ in $\mathbb{P}^2\times\mathbb{P}^2$ with ten isolated terminal Gorenstein singularities,
$\alpha$ is a small birational morphism that contracts proper transforms of ten common secants of the curves $C_1$ and $C_2$,
the map $\chi$ is a composition of ten flops, the morphism $\beta$ is a flopping contraction, the morphism $\phi$ is a $\mathbb{P}^1$-bundle,
and the action of the group $G$ on $\mathbb{P}^1\times\mathbb{P}^1$ is faithful.
\end{proposition}

\begin{proof}
Denote by $\Gamma$ the subgroup in $G$ isomorphic to $\mathfrak{A}_5$.
Let $\pi_1\colon X_1\to\mathbb{P}^3$ be the blow up of the curve $C_1$,
and let $\pi_2\colon X_2\to\mathbb{P}^3$ be the blow up of the curve $C_2$.
Then $X_1$ and~$X_2$ are isomorphic Fano threefolds.
Moreover, it follows from Example~\ref{example:P3-non-rigid-A5-P1-bundle} that there is a $\Gamma$-commutative diagram
$$
\xymatrix{
&&X_1\ar@{->}[lld]_{\phi_1}\ar@{->}[drr]^{\pi_1} &&X\ar@{->}[d]_{\pi}\ar@{->}[ll]_{\omega_2}\ar@{->}[rr]^{\omega_1}&& X_2\ar@{->}[rrd]^{\phi_2}\ar@{->}[dll]_{\pi_2} &&\\
\mathbb{P}^2 &&&& \mathbb{P}^3\ar@{-->}[llll]^{\theta_1}\ar@{-->}[rrrr]_{\theta_2}&&&& \mathbb{P}^2}
$$
Here $\phi_1$ and $\phi_2$ are $\mathbb{P}^1$-bundles,
$\theta_1$ and $\theta_2$ are rational maps that are given by the linear systems of quadrics passing through $C_1$ and $C_2$, respectively,
and $\omega_1$ and $\omega_2$ are blow ups of proper transforms of the curves $C_1$ and $C_2$, respectively.
Note that the action of $\Gamma$ on~$\mathbb{P}^2$ is faithful.

We claim that there are exactly $10$ lines in $\mathbb{P}^3$
that are common secants of both twisted cubic curves $C_1$ and $C_2$ (cf. \cite[\S2]{Wakeford} or \cite[Remark~2]{Dolgachev2016}).
Indeed, denote by $E_1^1$ the $\pi_1$-exceptional surface.
Then the morphism $\phi_1$ is given by the linear system~\mbox{$|\pi_1^*(\mathcal{O}_{\mathbb{P}^2}(2))-E_1^1|$},
and its fibers are proper transforms of the secant lines of the curve~$C_1$.
Denote by $\overline{C}_2$ the proper transform of the curve $C_2$ on the threefold $X_1$.
Since~$\overline{C}_2$ is $\Gamma$-invariant,
the curve $\phi_1(\overline{C}_2)$ is also $\Gamma$-invariant.
By \cite[Lemma~5.3.1]{CheltsovShramov}, there are no $\Gamma$-invariant irreducible curves in~$\mathbb{P}^2$ of degree $1$ and $3$. Thus we see that either~$\phi_1(\overline{C}_2)$ is an irreducible conic, or~$\phi_1(\overline{C}_2)$ is an irreducible curve of degree~$6$.
The former case is clearly impossible.
Indeed, if~$\phi_1(\overline{C}_2)$ is a conic, then the induced
morphism~\mbox{$\overline{C}_2\to\phi_1(\overline{C}_2)$} is a triple cover,
so that the curve~$C_2$ has infinitely many $3$-secant lines, which are also secants of the curve $C_1$.
However, twisted cubic curves in $\mathbb{P}^2$ do not have $3$-secant lines at all, because they are cut out by quadrics.
Therefore, we see that~$\phi_1(\overline{C}_2)$ is an irreducible curve of degree $6$, so that the induced morphism~\mbox{$\overline{C}_2\to\phi_1(\overline{C}_2)$} is birational.
By~\mbox{\cite[Lemma~5.3.1]{CheltsovShramov}}, the sextic curve $\phi_1(\overline{C}_2)$ is contained in the pencil that consists of all $\Gamma$-invariant curves of degree $6$ in $\mathbb{P}^2$.
The curve $\phi_1(\overline{C}_2)$ must be singular, since the curve~$C_2$ is rational.
All singular curves in this pencil are described in \cite[Remark~6.1.5]{CheltsovShramov}.
Using this description, we deduce that $\phi_1(\overline{C}_2)$ has $10$ double points (cf.~\mbox{\cite[Remark~5]{Dolgachev2016}}),
and these singular points form one $\Gamma$-orbit in~$\mathbb{P}^2$.
Note that the singular points of $\phi_1(\overline{C}_2)$ are ordinary double points.
Indeed, since the arithmetic genus of the curve $\phi_1(\overline{C}_2)$ equals $10$, its singular points are either
ordinary double points or ordinary cusps.
In the latter case, the curve~$\overline{C}_2$ would have a $\Gamma$-orbit of length $10$, which is impossible
by~\mbox{\cite[Lemma~5.1.5]{CheltsovShramov}}.
Thus, every singular point of the curve $\phi_1(\overline{C}_2)$ is an ordinary double point.
The fibers of $\phi_1$ over the singular points of the curve $\phi_1(\overline{C}_2)$ are exactly the proper transforms of the common secants of both curves $C_1$ and $C_2$.
Vice versa, if $\ell$ is a common secant of the curves $C_1$ and $C_2$,
then $\theta_1(\ell)$ must be a singular point of the curve~$\phi_1(\overline{C}_2)$.
Hence, we conclude that there are exactly $10$ lines in $\mathbb{P}^3$ that are secants of both~$C_1$ and~$C_2$,
and the group $\Gamma$ acts transitively on them.

Denote by $E_1^1$ the $\pi_1$-exceptional surface,
and denote by $E_2^2$ the $\pi_2$-exceptional surface.
Then
$$
-K_X\sim \omega_2^*\Big(\pi_1^*\big(\mathcal{O}_{\mathbb{P}^2}(2)\big)-E_1^1\Big)+\omega_1^*\Big(\pi_2^*\big(\mathcal{O}_{\mathbb{P}^2}(2)\big)-E_2^2\Big).
$$
Moreover, both divisors $\pi_1^*(\mathcal{O}_{\mathbb{P}^2}(2))-E_1^1$ and $\pi_2^*(\mathcal{O}_{\mathbb{P}^2}(2))-E_2^2$ are nef.
In particular, the divisor $-K_X$ is nef and big, since $-K_X^3=12$.

Denote the $10$  common secants of both $C_1$ and $C_2$ by $L_{1},\ldots,L_{10}$,
and denote their proper transforms on $X$ by $\widetilde{L}_{1},\ldots,\widetilde{L}_{10}$, respectively.
Then~$-K_X$ intersects these curves trivially.
Moreover, the curves $\widetilde{L}_{1},\ldots,\widetilde{L}_{10}$
are the only irreducible curves in $X$ that have trivial intersection with $-K_X$.
Indeed, let $C$ be an irreducible curve in the threefold~$X$ such that $-K_X\cdot C=0$. Then
$$
\omega_2(C)\cdot\Big(\pi_1^*\big(\mathcal{O}_{\mathbb{P}^2}(2)\big)-E_1^1\Big)=\omega_1(C)\cdot\Big(\pi_2^*\big(\mathcal{O}_{\mathbb{P}^2}(2)\big)-E_2^2\Big)=0,
$$
and $\pi(C)$ is an irreducible curve.
Since $\phi_1$ is given by the linear system $|\pi_1^*(\mathcal{O}_{\mathbb{P}^2}(2))-E_1^1|$,
we see that $\omega_2(C)$ is contracted by $\phi_1$ to a point.
Similarly, we see that $\omega_1(C)$ is contracted by $\phi_2$ to a point.
Hence, the curve $\pi(C)$ is a common secant of both curves $C_1$ and $C_2$,
which implies that $C$ is one of the curves $\widetilde{L}_{1},\ldots,\widetilde{L}_{10}$.

Consider the morphism
$$
\alpha=\pi_1\circ\omega_2=\pi_2\circ\omega_1.
$$
We already know that the image of $\alpha$ is three-dimensional.
Now we claim that $\alpha$
is birational. Indeed, choose a general fiber of $\alpha$, and suppose that it
contains at least two different points, say $A$ and $A^\prime$.
Then $\omega_1(A)=\omega_1(A^\prime)$, so that $\pi(A)$ and $\pi(A^\prime)$ lie on
a secant of the twisted cubic~$C_2$. Similarly,
we have $\omega_2(A)=\omega_2(A^\prime)$, so that $\pi(A)$ and $\pi(A^\prime)$ lie on
a secant of~$C_1$. On the other hand, since the fiber of $\alpha$ was chosen to be general, we can assume that
the points $\pi(A)$ and $\pi(A^\prime)$ are different points of~$\mathbb{P}^3$.
Hence they lie on one of the lines $L_{1},\ldots,L_{10}$, which gives a contradiction.

We have constructed the morphism $\alpha\colon X\to V_{12}$,
where $V_{12}$ is a divisor of bi-degree~\mbox{$(2,2)$}
in~\mbox{$\mathbb{P}^2\times\mathbb{P}^2$}.
The variety $V_{12}$ is an anticanonical model of the threefold~$X$, and the morphism~$\alpha$ is given by the linear system $|-K_X|$.
The threefold $V_{12}$ has $10$ isolated terminal Gorenstein singular points,
which are the images of the curves~\mbox{$\widetilde{L}_{1},\ldots,\widetilde{L}_{10}$}.
Thus, there exists a $G$-commutative diagram
$$
\xymatrix{
&&X\ar@{->}[lld]_\pi\ar@{->}[drr]^\alpha\ar@{-->}[rrrr]^\chi &&&& Y\ar@{->}[rrd]^\phi\ar@{->}[dll]_\beta &&\\
\mathbb{P}^3\ar@{-->}[rrrr]_\psi &&&& V_{12} &&&& Z}
$$
Here $\psi$ is a birational map that is given by the linear system of quartic surfaces in~$\mathbb{P}^3$ that contain both curves~$C_1$ and~$C_2$,
and~$\chi$ is a composition of flops in the curves~\mbox{$\widetilde{L}_{1},\ldots,\widetilde{L}_{10}$}.
The map $\beta$ is a flopping contraction, and the morphism $\phi$ is a $G$-extremal contraction such that the divisor $-K_Y$ is $\phi$-ample.
Now arguing as in  \cite{Takeuchi,CutroneMarshburn} and using  \cite[Theorem~3.3]{Mori} together with \cite[Theorem~4]{Cutkosky},
we see that $\phi$ is a $\mathbb{P}^1$-bundle, and $Z\cong\mathbb{P}^1\times\mathbb{P}^1$.
Finally,  using the fact that our $\mathbb{P}^3$ contains exactly two $\Gamma$-orbits of length $12$ by \cite[Lemma~3.2]{CheltsovShramov2015},
we conclude that  the group $G$ acts faithfully on $Z$.
\end{proof}

In fact, one can show that the variety $V_{12}$ in Proposition~\ref{proposition:S5-link} has isolated ordinary double points,
and $\chi$ is a composition of Atiyah flops.
Moreover, one can show that a general fiber of the rational map $\phi\circ\chi\circ\pi^{-1}$
is a twisted cubic curve~$Z$ such that both intersections $Z\cap C_1$ and $Z\cap C_2$ consist of $5$ points.

\begin{corollary}
\label{corollary:A5-S5}
Let $G$ be a subgroup in $\mathrm{PGL}_4(\mathbb{C})$ such that $G\cong\mathfrak{A}_5$ or $G\cong\mathfrak{S}_5$.
Then~$\mathbb{P}^3$ is not $G$-birationally rigid.
\end{corollary}

\begin{proof}
If $G$ is not primitive, the assertion follows from Corollary~\ref{corollary:non-primitive}.
If $G$ is primitive, the assertion follows from Examples~\ref{example:P3-non-rigid-A5-V22},
\ref{example:P3-non-rigid-A5-P1-bundle}, \ref{example:P3-non-rigid-S5-P1-bundle-dP5} and Proposition~\ref{proposition:S5-link}.
Alternatively, instead of using Example~\ref{example:P3-non-rigid-S5-P1-bundle-dP5},
we can use Example~\ref{example:P3-non-rigid-S5-cubic}.
\end{proof}

\section{The Heisenberg group}
\label{section:Heisenberg}

Let $\varphi\colon\mathrm{SL}_4(\mathbb{C})\to\mathrm{PGL}_4(\mathbb{C})$ be the natural projection.
For every group $G\subset\mathrm{SL}_4(\mathbb{C})$, let us denote by $\overline{G}$ the group $\varphi(G)$.
Similarly, for every element $g\in\mathrm{SL}_4(\mathbb{C})$, we denote by $\overline{g}$ its image $\varphi(g)$.
Let
\begin{multline}
\label{equation:S1-S2-T1-T2}
S_1=\left(
  \begin{array}{cccc}
    0 & 0 & 1 & 0 \\
    0 & 0 & 0 & 1 \\
    1 & 0 & 0 & 0 \\
    0 & 1 & 0 & 0
  \end{array}
\right), \quad S_2=\left(
  \begin{array}{cccc}
    0 & 1 & 0 & 0\\
    1 & 0 & 0 & 0\\
    0 & 0 & 0 & 1\\
    0 & 0 & 1 & 0
 \end{array}
\right), \\
T_1=\left(
  \begin{array}{cccc}
    1 & 0 & 0 & 0\\
    0 & 1 & 0 & 0\\
    0 & 0 & -1& 0\\
    0 & 0 & 0 & -1
 \end{array}
\right), \quad
T_2=\left(
  \begin{array}{cccc}
    1 & 0 & 0& 0\\
    0 & -1& 0& 0\\
    0 & 0 & 1& 0\\
    0 & 0 & 0& -1
 \end{array}
\right),
\end{multline}
and denote by $\mathrm{H}$ the subgroup in $\mathrm{SL}_{4}(\mathbb{C})$ that is generated by the matrices $S_1$, $S_2$, $T_1$, $T_2$.
Then the center of $\mathrm{H}$ coincides with the commutator subgroup of
$\mathrm{H}$
and consists of $\pm I_4$, where $I_4$ is the identity matrix.
Since $|\mathrm{H}|=32$, we have $\overline{\mathrm{H}}\cong\mumu_2^4$.
We say that $\mathrm{H}$ is the Heisenberg group.
The group $\overline{\mathrm{H}}$ is a transitive imprimitive subgroup of $\mathrm{PGL}_4(\mathbb{C})$.

\begin{lemma}
\label{lemma:Heisenberg-subgroup}
The following assertions hold.
\begin{itemize}
\item[(i)] For every non-trivial element $g\in\overline{\mathrm{H}}$,
the locus of fixed points of $g$ consists of two skew lines $L_g$ and $L_g^\prime$ in $\mathbb{P}^3$.

\item[(ii)] For two distinct non-trivial elements $g$ and $h$ in $\overline{\mathrm{H}}$, one has
$$
\Big\{L_g,L_g^\prime\Big\}\cap\Big\{L_h,L_h^\prime\Big\}=\varnothing.
$$
\item[(iii)] No subgroup of order $8$ in $\overline{\mathrm{H}}$ has a fixed point in $\mathbb{P}^3$.
\end{itemize}
\end{lemma}

\begin{proof}
Assertions (i) and (ii) are easy to check by direct computations (cf. \cite[\S2.2]{Eklund}).
Assertion (iii) follows from the transitivity of the subgroup $\overline{\mathrm{H}}$.
\end{proof}

\begin{lemma}
\label{lemma:Heisenberg-elliptic}
Let $E$ be a rational or an elliptic curve. Then $\overline{\mathrm{H}}$ cannot act faithfully on~$E$.
\end{lemma}

\begin{proof}
If $E$ is rational, the assertion follows from the classification of finite subgroups of the group $\mathrm{PGL}_2(\mathbb{C})$.
Thus, we suppose that $E$ is a (smooth) elliptic curve with a faithful action of $\overline{\mathrm{H}}$.
Then there is an exact sequence of groups
$$
1\to E(\mathbb{C})\to\mathrm{Aut}(E)\to\mumu_r\to 1,
$$
where $E(\mathbb{C})$ is the group of points of $E$ acting on $E$ by translations, and $r\leqslant 6$.
We see that $\overline{\mathrm{H}}\cap E(\mathbb{C})$ contains a subgroup isomorphic to $\mumu_2^3$.
This is impossible, because there are just four $2$-torsion points in $E(\mathbb{C})$.
\end{proof}

Denote by $\mathbf{N}$ the normalizer of the group $\mathrm{H}$ in $\mathrm{SL}_{4}(\mathbb{C})$, and let
\begin{equation}
\label{equation:S-T}
S=\frac{1+i}{\sqrt{}2}\left(
  \begin{array}{cccc}
    i & 0 & 0& 0\\
    0 & i& 0& 0\\
    0 & 0 & 1& 0\\
    0 & 0 & 0& 1
 \end{array}
\right), \quad T=\frac{1+i}{2}\left(
  \begin{array}{cccc}
    -i& 0& 0& i\\
    0 & 1& 1& 0\\
    1 & 0& 0& 1\\
    0 & -i& i& 0
 \end{array}
\right),
\end{equation}
where $i=\sqrt{-1}$. One can check that both $S$ and $T$ are contained in $\mathbf{N}$, and $S^4=T^5=-I_4$.
Let
$$
A=\frac{1+i}{2}\left(
  \begin{array}{cccc}
    1 & 0 & 0& -1\\
    0 & 1& 1& 0\\
    0 & i & -i& 0\\
    i & 0 & 0& -i
 \end{array}
\right), \quad B=\frac{1+i}{2}\left(
  \begin{array}{cccc}
    -1 & i & 0& 0\\
    1 & i& 1& 0\\
    0 & 0 & i& -1\\
    0 & 0 & -i& -1
 \end{array}
\right).
$$
Then $A$ and $B$ are contained in $\mathbf{N}$, because
$A=T^{4}\cdot S\cdot T\cdot S$
and
$$
B=S\cdot T^{4}\cdot S\cdot T\cdot S\cdot T^2\cdot S\cdot T^{3}\cdot S\cdot T\cdot S\cdot T^3\cdot S\cdot T^{3}\cdot S.
$$
Similarly, let
$$
R=\frac{1}{\sqrt{2}}\left(
  \begin{array}{cccc}
    1 & i & 0 & 0\\
    i & 1 & 0 & 0\\
    0 & 0 & i & 1\\
    0 & 0 & -1 & -i
 \end{array}
\right).
$$
Then $R$ is also contained in $\mathbf{N}$ (see \cite[\S124]{Blichfeldt1917}).

It follows from \cite[\S124]{Blichfeldt1917} that  $S$ and $T$ together with $\mathrm{H}$ and $\pm iI_4$ generate the group~$\mathbf{N}$.
Let~$\mathbb{H}=\langle\mathrm{H},\pm iI_4\rangle$. Then $\overline{\mathbb{H}}=\overline{\mathrm{H}}$,
and it follows from \cite[\S124]{Blichfeldt1917} that there is an exact sequence
\begin{equation}
\label{equation:exact-sequence}
\xymatrix{
1\ar@{->}[rr]&&\mathbb{H}\ar@{->}[rr]&&\mathbf{N}\ar@{->}[rr]&&\mathfrak{S}_6\ar@{->}[rr]&& 1.}
\end{equation}

Let $\mathcal{U}_4$ be  the vector space of all $\mathbb{H}$-invariant homogeneous polynomials of degree $4$ in~$\mathbb{C}[x,y,z,w]$.
Recall from \cite[\S2.3]{Eklund} that $\mathcal{U}_4$ is generated by the polynomials
$$
x^4+y^4+z^4+w^4, \quad x^2y^2+z^2w^2, \quad  x^2z^2+y^2w^2, \quad  x^2w^2+y^2z^2, \quad xyzw.
$$
This vector space is $\mathbf{N}$-invariant. Write
\begin{multline}
\label{equation:Heisenberg-quartics}
t_0=\frac{x^4+y^4+z^4+w^4}{3}-2\big(x^2y^2+z^2w^2\big)-2\big(x^2z^2+y^2w^2\big)-2\big(x^2w^2+y^2z^2\big),\\
t_1=\frac{x^4+y^4+z^4+w^4}{3}-2\big(x^2y^2+z^2w^2\big)+2\big(x^2z^2+y^2w^2\big)+2\big(x^2w^2+y^2z^2\big),\\
t_2=\frac{x^4+y^4+z^4+w^4}{3}+2\big(x^2y^2+z^2w^2\big)-2\big(x^2z^2+y^2w^2\big)+2\big(x^2w^2+y^2z^2\big),\\
t_3=\frac{x^4+y^4+z^4+w^4}{3}+2\big(x^2y^2+z^2w^2\big)+2\big(x^2z^2+y^2w^2\big)-2\big(x^2w^2+y^2z^2\big),\\
t_4=-\frac{2}{3}\big(x^4+y^4+z^4+w^4\big)+8xyzw,\\
t_5=-\frac{2}{3}\big(x^4+y^4+z^4+w^4\big)-8xyzw.
\end{multline}
Note that $t_0+t_1+t_2+t_3+t_4+t_5=0$, and $t_0$, $t_1$, $t_2$, $t_3$, $t_3$, $t_4$, $t_5$ generate $\mathcal{U}_4$.

\begin{remark}
\label{remark:S6-quartics}
The representation of the group $\mathbf{N}/\mathbb{H}\cong\mathfrak{S}_6$ in the vector space  $\mathcal{U}_4$
is the standard five-dimensional representation twisted by the sign character.
Indeed, the matrix~$T$ acts on the polynomials $t_0,\ldots,t_5$ as the cycle~\mbox{$(t_0\,t_4\,t_2\,t_5\,t_1)$},
and the matrix~\mbox{$-S$} acts on  $t_0,\ldots,t_5$ as the permutation~\mbox{$(t_0\,t_1)(t_2\,t_3)(t_4\,t_5)$}.
Observe also that the matrix $A$ acts on them as the permutation~\mbox{$(t_0\,t_5\,t_2)(t_1\,t_3\,t_4)$}.
Similarly, the matrix $B$ acts as the permutation~\mbox{$(t_0\,t_2\,t_5)(t_1\,t_3\,t_4)$}.
\end{remark}

Let $G_{80}$ be the subgroup in $\mathrm{SL}_{4}(\mathbb{C})$ that is generated by $\mathbb{H}$ and the matrix $T$,
and let~$\overline{G}_{80}$ be its image in $\mathrm{PGL}_{4}(\mathbb{C})$.
Then
$$
\overline{G}_{80}\cong\mumu_2^4\rtimes\mumu_5
$$
by Remark~\ref{remark:S6-quartics}.
Similarly, let $G_{320}$ be the subgroup in $\mathrm{SL}_{4}(\mathbb{C})$ that is generated by $\mathbb{H}$ and the matrices $T$ and $R$,
and let $\overline{G}_{320}$ be its image in $\mathrm{PGL}_{4}(\mathbb{C})$.
Then there is a non-split exact sequence of groups
$$
1\longrightarrow\overline{\mathrm{H}}\longrightarrow\overline{G}_{320}\longrightarrow\mumu_5\rtimes\mumu_4\longrightarrow 1,
$$
see \cite[\S124]{Blichfeldt1917}.
Moreover, it follows from \eqref{equation:exact-sequence} that there exists an exact sequence of groups
$$
1\longrightarrow\mathbb{H}\longrightarrow G_{320}\longrightarrow\mumu_5\rtimes\mumu_4\longrightarrow 1.
$$
Let $G_{160}$ be the subgroup in $\mathrm{SL}_{4}(\mathbb{C})$ that is generated by $\mathbb{H}$ and the matrices $T$ and $R^2$,
and let $\overline{G}_{160}$ be its image in $\mathrm{PGL}_{4}(\mathbb{C})$. Then
$$
\overline{G}_{160}\cong\mumu_2^4\rtimes\big(\mumu_5\rtimes\mumu_2\big)\cong\mumu_2^4\rtimes\mathrm{D}_{10},
$$
where $\mathrm{D}_{10}$ is the dihedral group of order $10$,
see \cite[\S124]{Blichfeldt1917}. There exists an exact sequence of groups
$$
1\longrightarrow\mathbb{H}\longrightarrow G_{160}\longrightarrow\mathrm{D}_{10}\longrightarrow 1.
$$
Furthermore, the group $\overline{G}_{80}$ is a normal subgroup in both $\overline{G}_{160}$ and $\overline{G}_{320}$,
and the group~$\overline{G}_{160}$ is a normal subgroup in $\overline{G}_{320}$.
Let $G_{144}$ be the subgroup in $\mathrm{SL}_{4}(\mathbb{C})$ that is generated by the subgroup $\mathbb{H}$ together with the matrices $A$ and $B$,
and let $\overline{G}_{144}$ be its image in~\mbox{$\mathrm{PGL}_{4}(\mathbb{C})$}. Then $G_{144}$ is a finite subgroup of order  $576$,
and there exists an exact sequence of groups
$$
1\longrightarrow\mathbb{H}\longrightarrow G_{144}\longrightarrow\mumu_3\times\mumu_3\longrightarrow 1.
$$
Moreover, one has
$$
\overline{G}_{144}\cong\mumu_2^4\rtimes\big(\mumu_3\times\mumu_3\big)\cong\mathfrak{A}_4\times\mathfrak{A}_4
$$
by Remark~\ref{remark:S6-quartics}.
The groups $\overline{G}_{80}$, $\overline{G}_{160}$, $\overline{G}_{320}$, $\overline{G}_{144}$ are primitive subgroups in  $\mathrm{PGL}_4(\mathbb{C})$.

It is well known that $\mathbb{P}^3$ contains exactly ten $\overline{\mathrm{H}}$-invariant quadrics (see \cite[\S2.4]{Eklund}).
The groups $\overline{G}_{80}$, $\overline{G}_{160}$, $\overline{G}_{320}$ and $\overline{G}_{144}$ naturally act on them,
because $\overline{\mathrm{H}}$ is a normal subgroup of these four groups.
For instance, the group $\overline{G}_{80}$ splits them into two $\overline{G}$-orbits.
One of them consists of the five quadrics
\begin{multline}
\label{equation:quadrics-1-2-3-4-5}
x^2+y^2+z^2+w^2=0, \quad xw+yz=0, \quad xz+yw=0,\\
x^2+y^2-z^2-w^2=0, \quad x^2-y^2-z^2+w^2=0.
\end{multline}
Let us denote them by $\mathcal{Q}_1$, $\mathcal{Q}_2$, $\mathcal{Q}_3$, $\mathcal{Q}_4$ and $\mathcal{Q}_5$, respectively.
The second $\overline{G}_{80}$-orbit consists of the five quadrics
\begin{equation}
\label{equation:quadrics-6-7-8-9-10}
x^2-y^2+z^2-w^2=0, \quad xy+zw=0, \quad xy-zw=0, \quad xz-yw=0, \quad xw-yz=0.
\end{equation}
We will denote them by $\mathcal{Q}_6$, $\mathcal{Q}_7$, $\mathcal{Q}_8$, $\mathcal{Q}_9$ and $\mathcal{Q}_{10}$, respectively.
Observe that all these ten quadrics are smooth.
Similarly, the only  $\overline{G}_{144}$-invariant quadric in $\mathbb{P}^3$ is the quadric $\mathcal{Q}_3$, i.e. the quadric given by $xz+yw=0$.
In fact, we can say more.

\begin{lemma}
\label{lemma:80-144-quadrics}
There are no $\overline{G}_{80}$-invariant surfaces in $\mathbb{P}^3$ of degree $1$, $2$ and $3$.
Similarly, there are no $\overline{G}_{144}$-invariant surfaces in $\mathbb{P}^3$ of degree $1$ and $3$.
Finally, there is unique  $\overline{G}_{144}$-invariant quadric surface in $\mathbb{P}^3$.
\end{lemma}

\begin{proof}
This follows from explicit computations of symmetric powers of corresponding four-dimensional representations of the groups $G_{80}$ and $G_{144}$.
We used \cite{GAP} to perform these computations.
\end{proof}

Recall from Lemma~\ref{lemma:Heisenberg-subgroup}(i) that for every non-trivial element $g\in\overline{\mathrm{H}}$, there exist two skew lines $L_g$ and $L_g^\prime$ in $\mathbb{P}^3$ that are pointwise fixed by $g$.
The stabilizer of each of these lines in $\overline{\mathrm{H}}$ is $\mumu_2^3$,
and the kernel of its action on any of these lines is a subgroup $\mumu_2$ generated by $g$.
In total, there are $30$ lines in $\mathbb{P}^3$ that are pointwise fixed by elements in $\overline{\mathrm{H}}$.
Denote the set of these lines by $\mathbb{L}_{30}$. Let us describe these lines explicitly.

The group $\overline{\mathrm{H}}\cong\mumu_2^4$ acts on $\mathcal{Q}_1\cong\mathbb{P}^1\times\mathbb{P}^1$ naturally.
This means the following.
The group $\mumu_2^2$ acts on $\mathbb{P}^1$ in a standard way:
in appropriate homogeneous coordinates one of its generators acts as $[u:v]\mapsto[v:u]$,
and the other by $[u:v]\mapsto[-u:v]$.
The product action of the group $\mumu_2^2\times\mumu_2^2$ on $\mathbb{P}^1\times\mathbb{P}^1$
gives the action of the group $\overline{\mathrm{H}}$ on $\mathcal{Q}_1$.

Since every subgroup $\mumu_2$ in $\mumu_2^2$ fixes exactly two points in $\mathbb{P}^1$,
and there are three such subgroups in $\mumu_2^2$, we see that there are exactly six pairs of skew lines in~$\mathcal{Q}_1$
that are $\overline{\mathrm{H}}$-invariant. The smooth quadric surface $\mathcal{Q}_1$ contains two one-parameter
families of lines.
Three of the above pairs of skew lines lie in one family of lines on $\mathcal{Q}_1$,
and the remaining three pairs lie in another family of lines on $\mathcal{Q}_1$.
Let $\ell$ be the line in $\mathbb{P}^3$ that passes through the points $[0:i:1:0]$ and $[1:0:0:-i]$,
and let $\check{\ell}$ be the line in $\mathbb{P}^3$ that passes through the points $[0:-i:1:0]$ and $[1:0:0:i]$.
Then $\ell$ and $\check{\ell}$ are skew lines contained in $\mathcal{Q}_1$.
Moreover, both of them are pointwise fixed by $\overline{T}_1\circ\overline{S}_2\circ\overline{S}_1$ (see \eqref{equation:S1-S2-T1-T2}),
so that $\ell$ and $\check{\ell}$ are two lines in $\mathbb{L}_{30}$.
Applying powers of $\overline{T}$ to them, we obtain $8$ more lines in~$\mathbb{L}_{30}$.
Here $\overline{T}$ is an element of the group $\overline{G}$ of order $5$ defined in \eqref{equation:S-T}.
To be precise, let
\begin{multline}
\label{equation:ell-1-2-3-4-5}
\ell_1=\ell, \quad \ell_2=\overline{T}(\ell_1), \quad \ell_3=\overline{T}(\ell_2), \quad \ell_4=\overline{T}(\ell_3), \quad \ell_5=\overline{T}(\ell_4), \\
\check{\ell}_1=\check{\ell}, \quad \check{\ell}_2=\overline{T}(\check{\ell}_1), \quad \check{\ell}_3=\overline{T}(\check{\ell}_2), \quad \check{\ell}_4=\overline{T}(\check{\ell}_3), \quad \check{\ell}_5=\overline{T}(\check{\ell}_4).
\end{multline}
Then  the lines in \eqref{equation:ell-1-2-3-4-5} are contained in $\mathbb{L}_{30}$.
Similarly, let $\ell^\prime$ be the line in $\mathbb{P}^3$ that passes through the points $[1:0:-i:0]$ and $[0:1:0:-i]$,
let $\check{\ell}^\prime$ be the line in $\mathbb{P}^3$ that passes through the points $[1:0:i:0]$ and $[0:1:0:i]$,
and let
\begin{multline}
\label{equation:ell-1-2-3-4-5-prime}
\ell^\prime_1=\ell^\prime, \quad \ell^\prime_2=\overline{T}(\ell^\prime_1), \quad \ell^\prime_3=\overline{T}(\ell^\prime_2), \quad \ell^\prime_4=\overline{T}(\ell^\prime_3), \quad \ell^\prime_5=\overline{T}(\ell^\prime_4),\\
\check{\ell}^\prime_1=\check{\ell}^\prime, \quad \check{\ell}^\prime_2=\overline{T}(\check{\ell}^\prime_1), \quad \check{\ell}^\prime_3=\overline{T}(\check{\ell}^\prime_2), \quad \check{\ell}^\prime_4=\overline{T}(\check{\ell}^\prime_3), \quad \check{\ell}^\prime_5=\overline{T}(\check{\ell}^\prime_4).
\end{multline}
Then the lines in \eqref{equation:ell-1-2-3-4-5-prime} are contained in $\mathbb{L}_{30}$.
Finally, let $\ell^{\prime\prime}$ be the line in $\mathbb{P}^3$ that passes through the points $[1:i:0:0]$ and $[0:0:1:-i]$,
let $\check{\ell}^{\prime\prime}$ be the line in $\mathbb{P}^3$ that passes through the points $[1:-i:0:0]$ and $[0:0:1:i]$,
and let
\begin{multline}
\label{equation:ell-1-2-3-4-5-prime-prime}
\ell^{\prime\prime}_1=\ell^{\prime\prime}, \quad \ell^{\prime\prime}_2=\overline{T}(\ell^{\prime\prime}_1), \quad \ell^{\prime\prime}_3=\overline{T}(\ell^{\prime\prime}_2), \quad \ell^{\prime\prime}_4=\overline{T}(\ell^{\prime\prime}_3), \quad \ell^{\prime\prime}_5=\overline{T}(\ell^{\prime\prime}_4),\\
\check{\ell}^{\prime\prime}_1=\check{\ell}^{\prime\prime}, \quad \check{\ell}^{\prime\prime}_2=\overline{T}(\check{\ell}^{\prime\prime}_1), \quad \check{\ell}^{\prime\prime}_3=\overline{T}(\check{\ell}^{\prime\prime}_2), \quad \check{\ell}^{\prime\prime}_4=\overline{T}(\check{\ell}^{\prime\prime}_3), \quad \check{\ell}^{\prime\prime}_5=\overline{T}(\check{\ell}^{\prime\prime}_4).
\end{multline}
Then the lines in \eqref{equation:ell-1-2-3-4-5-prime-prime} are the remaining $10$ lines in $\mathbb{L}_{30}$.
By construction, the lines $\ell$, $\check{\ell}$, $\ell^\prime$, $\check{\ell}^\prime$, $\ell^{\prime\prime}$ and $\check{\ell}^{\prime\prime}$
are contained in the quadric surface $\mathcal{Q}_1$.

\begin{lemma}
\label{lemma:Heseinebrg-lines-quadrics}
Each line among the $30$ lines in $\mathbb{L}_{30}$ is contained in exactly $4$ of the quadric surfaces
$\mathcal{Q}_1,\ldots,\mathcal{Q}_{10}$.
Each of these quadrics contains exactly $12$ such lines.
For every quadric, these $12$ lines contained in it are described by Table~\ref{table:quadric-lines}.
Every two quadrics among
$\mathcal{Q}_1,\ldots,\mathcal{Q}_{10}$
intersect in a quadruple of lines in $\mathbb{L}_{30}$.
\end{lemma}

\begin{proof}
Direct computations.
The last assertion can be also deduced from Lemma~\ref{lemma:Heisenberg-elliptic}.
\end{proof}

\begin{table}\renewcommand\arraystretch{1.42}
\caption{Ten $\overline{\mathrm{H}}$-invariant quadrics and thirty lines in them.\label{table:quadric-lines}}
\begin{tabular}{|c||c|c|c|c|c|c|c|c|c|c|}
\hline
$\quad$ $\enskip$  & $\enskip$ $\mathcal{Q}_1$ $\enskip$  & $\enskip$ $\mathcal{Q}_2$ $\enskip$ & $\enskip$ $\mathcal{Q}_3$ $\enskip$  & $\enskip$ $\mathcal{Q}_4$ $\enskip$ & $\enskip$ $\mathcal{Q}_5$ $\enskip$ & $\enskip$ $\mathcal{Q}_6$ $\enskip$ & $\enskip$ $\mathcal{Q}_7$ $\enskip$ & $\enskip$ $\mathcal{Q}_8$ $\enskip$ & $\enskip$ $\mathcal{Q}_9$ $\enskip$ & $\enskip$ $\mathcal{Q}_{10}$ $\enskip$\\
\hline\hline
$\ell_1$                        & $+$  & $-$  & $-$  & $-$  & $+$  & $-$  & $+$  & $-$  & $+$  & $-$ \\
\hline
$\ell_2$                        & $-$  & $-$  & $-$  & $+$  & $+$  & $-$  & $-$  & $+$  & $-$  & $+$ \\
\hline
$\ell_3$                        & $-$  & $-$  & $+$  & $+$  & $-$  & $+$  & $+$  & $-$  & $-$  & $-$\\
\hline
$\ell_4$                        & $-$  & $+$  & $+$  & $-$  & $-$  & $-$  & $-$  & $-$  & $+$  & $+$ \\
\hline
$\ell_5$                        & $+$  & $+$  & $-$  & $-$  & $-$  & $+$  & $-$  & $+$  & $-$  & $-$ \\
\hline
$\check{\ell}_1$                & $+$  & $-$  & $-$  & $-$  & $+$  & $-$  & $+$  & $-$  & $+$  & $-$ \\
\hline
$\check{\ell}_2$                & $-$  & $-$  & $-$  & $+$  & $+$  & $-$  & $-$  & $+$  & $-$  & $+$ \\
\hline
$\check{\ell}_3$                & $-$  & $-$  & $+$  & $+$  & $-$  & $+$  & $+$  & $-$  & $-$  &  $-$\\
\hline
$\check{\ell}_4$                & $-$  & $+$  & $+$  & $-$  & $-$  & $-$  & $-$  & $-$  & $+$  & $+$ \\
\hline
$\check{\ell}_5$                & $+$  & $+$  & $-$  & $-$  & $-$  & $+$  & $-$  & $+$  & $-$  & $-$ \\
\hline
$\ell_1^\prime$                 & $+$  & $-$  & $-$  & $-$  & $-$  & $+$  & $+$  & $-$  & $-$  & $+$ \\
\hline
$\ell_2^\prime$                 & $-$  & $-$  & $-$  & $-$  & $+$  & $+$  & $-$  & $-$  & $+$  & $+$ \\
\hline
$\ell_3^\prime$                 & $-$  & $-$  & $-$  & $+$  & $-$  & $+$  & $-$  & $+$  & $+$  & $-$ \\
\hline
$\ell_4^\prime$                 & $-$  & $-$  & $+$  & $-$  & $-$  & $-$  & $+$  & $+$  & $+$  & $-$ \\
\hline
$\ell_5^\prime$                 & $-$  & $+$  & $-$  & $-$  & $-$  & $-$  & $+$  & $+$  & $-$  & $+$ \\
\hline
$\check{\ell}_1^\prime$         & $+$  & $-$  & $-$  & $-$  & $-$  & $+$  & $+$  & $-$  & $-$  & $+$ \\
\hline
$\check{\ell}_2^\prime$         & $-$  & $-$  & $-$  & $-$  & $+$  & $+$  & $-$  & $-$  & $+$  & $+$ \\
\hline
$\check{\ell}_3^\prime$         & $-$  & $-$  & $-$  & $+$  & $-$  & $+$  & $-$  & $+$  & $+$  &  $-$\\
\hline
$\check{\ell}_4^\prime$         & $-$  & $-$  & $+$  & $-$  & $-$  & $-$  & $+$  & $+$  & $+$  &  $-$\\
\hline
$\check{\ell}_5^\prime$         & $-$  & $+$  & $-$  & $-$  & $-$  & $-$  & $+$  & $+$  & $-$  & $+$\\
\hline
$\ell_1^{\prime\prime}$         & $+$  & $+$  & $-$  & $+$  & $-$  & $-$  & $-$  & $-$  & $+$  &  $-$\\
\hline
$\ell_2^{\prime\prime}$         & $+$  & $-$  & $+$  & $-$  & $+$  & $-$  & $-$  & $+$  & $-$  & $-$ \\
\hline
$\ell_3^{\prime\prime}$         & $-$  & $+$  & $-$  & $+$  & $+$  & $-$  & $+$  & $-$  & $-$  &  $-$\\
\hline
$\ell_4^{\prime\prime}$         & $+$  & $-$  & $+$  & $+$  & $-$  & $-$  & $-$  & $-$  & $-$  & $+$ \\
\hline
$\ell_5^{\prime\prime}$         & $-$  & $+$  & $+$  & $-$  & $+$  & $+$  & $-$  & $-$  & $-$  & $-$ \\
\hline
$\check{\ell}_1^{\prime\prime}$ & $+$  & $+$  & $-$  & $+$  & $-$  & $-$  & $-$  & $-$  & $+$  & $-$ \\
\hline
$\check{\ell}_2^{\prime\prime}$ & $+$  & $-$  & $+$  & $-$  & $+$  & $-$  & $-$  & $+$  & $-$  & $-$ \\
\hline
$\check{\ell}_3^{\prime\prime}$ & $-$  & $+$  & $-$  & $+$  & $+$  & $-$  & $+$  & $-$  & $-$  & $-$ \\
\hline
$\check{\ell}_4^{\prime\prime}$ & $+$  & $-$  & $+$  & $+$  & $-$  & $-$  & $-$  & $-$  & $-$  & $+$ \\
\hline
$\check{\ell}_5^{\prime\prime}$ & $-$  & $+$  & $+$  & $-$  & $+$  & $+$  & $-$  & $-$  & $-$  & $-$ \\
\hline
\end{tabular}
\end{table}

Let us describe the intersection points of the $30$ lines in $\mathbb{L}_{30}$.
Namely, let
$\Sigma_{20}$ be the subset in $\mathbb{P}^3$ that consists of the $20$ points
\begin{multline}
\label{equation:Sigma-20}
[i:0:0:1], [0:i:1:0], [1:0:1:0], [0:-1:0:1], [1:i:-i:1],\\
[1:-i:i:1], [0:0:-1:1], [0:0:1:1], [-1:i:-i:1], [1:i:i:1],\\
[-i:0:0:1], [0:-i:1:0], [1:0:1:0], [0:1:0:1], [-1:i:i:1],\\
[1:-i:-i:1], [-1:1:0:0], [1:1:0:0], [-1:-i:i:1], [1:-i:-i:1].
\end{multline}
Similarly, let
$\Sigma_{20}^\prime$ be the subset in $\mathbb{P}^3$ that consists of the $20$ points
\begin{multline}
\label{equation:Sigma-20-prime}
[-i:i:-1:1], [i:0:1:0], [i:i:1:1], [0:i:0:1], [-1:1:-1:1],\\
[i:-i:-1:1], [1:-1:-1:1], [-1:0:0:1], [-1:-1:1:1], [0:-1:1:0],\\
[1:0:0:0], [1:0:0:1], [0:0:0:1], [-i:0:1:0], [0:0:1:0],\\
[-i:-i:1:1], [0:-i:0:1], [1:1:1:1], [0:1:1:0], [0:1:0:0].
\end{multline}
Finally, let $\Sigma_{20}^{\prime\prime}$ be the subset in $\mathbb{P}^3$ that consists of the $20$ points
\begin{multline}
\label{equation:Sigma-20-prime-prime}
[i:-1:i:1], [0:0:i:1], [-i:1:i:1], [-i:1:0:0], [-i:-1:i:1],\\
[-i:i:1:1], [-i:1:-i:1], [i:-i:-1:1], [1:-1:1:1], [i:-1:-i:1],\\
[1:-1:-1:1], [i:1:0:0], [i:1:i:1], [i:-i:1:1], [-1:1:1:1],\\
[i:1:-i:1], [0:0:-i:1], [i:-1:-i:1], [i:i:-1:1], [1:1:-1:1].
\end{multline}
Explicit computations show that $\Sigma_{20}$, $\Sigma_{20}^\prime$ and $\Sigma_{20}^{\prime\prime}$ are $\overline{G}_{80}$-orbits.
This also follows from Lemma~\ref{lemma:80-quartic-surfaces}.

\begin{lemma}
\label{lemma:Heseinebrg-lines-points}
The set $\Sigma_{20}\cup\Sigma_{20}^\prime\cup\Sigma_{20}^{\prime\prime}$ contains all intersection points of the lines from~$\mathbb{L}_{30}$.
Moreover, for every point $P\in\Sigma_{20}\cup\Sigma_{20}^\prime\cup\Sigma_{20}^{\prime\prime}$
there are exactly three lines from $\mathbb{L}_{30}$ passing through $P$.
\end{lemma}

\begin{proof}
Direct computations.
\end{proof}

Now we are going to describe the $\overline{G}_{80}$-invariant quartic surfaces
and the $\overline{G}_{144}$-invariant quartic surfaces.

\begin{lemma}
\label{lemma:80-144-quartic-surfaces}
There are exactly five $\overline{G}_{80}$-invariant irreducible quartic surfaces in $\mathbb{P}^3$.
Similarly, there are exactly four $\overline{G}_{144}$-invariant irreducible quartic surfaces in $\mathbb{P}^3$.
\end{lemma}

\begin{proof}
It follows from Remark~\ref{remark:S6-quartics} that
$\mathcal{U}_4$ splits as a sum of distinct one-dimensional representations of $G_{80}$ and $G_{144}$.
Now the assertion follows from Lemma~\ref{lemma:80-144-quadrics}.
\end{proof}

Up to scaling, all homogeneous semi-invariants of the group~$G_{144}$ of degree~$4$ in~\mbox{$\mathbb{C}[x,y,z,w]$}  are
\begin{equation}
\label{equation:144-p0-p1-p2-p3-p4}
\left\{\aligned%
&p_0=-\frac{1}{4}\Big(t_0+t_2+t_5\Big),\\
&p_1=-\Big(\xi_3t_0+t_2+\xi_3^2t_5\Big),\\
&p_2=(\xi_3+1)\Big(t_0+\xi_3^2t_5+\xi_3t_2\Big),\\
&p_3=(\xi_3+1)\Big(t_1+\xi_3^2t_4+\xi_3t_3\Big),\\
&p_4=-\Big(\xi_3t_1+t_3+\xi_3^2t_4\Big).\\
\endaligned
\right.
\end{equation}
Here $\xi_3$ is a primitive cubic root of unity,
and $t_0,\ldots,t_5$ are the polynomials defined in~\eqref{equation:Heisenberg-quartics}.
Observe that
$$
p_0=(wy+xz)^2
$$
is the only homogeneous invariant of degree $4$ of the group $G_{144}$.
The remaining four polynomials $p_1$, $p_2$, $p_3$ and $p_4$ are semi-invariants,
and they are irreducible.
Thus, the four irreducible $\overline{G}_{144}$-invariant quartic surfaces in $\mathbb{P}^3$ are given by $p_1=0$, $p_2=0$, $p_3=0$ and~\mbox{$p_4=0$}.
We will study properties of these surfaces in~\S\ref{section:144}.

The eigenvectors of the matrix $T$ on the vector space of all
$\mathbb{H}$-invariant homogeneous quartic polynomials in $\mathbb{C}[x,y,z,w]$ are the polynomials
\begin{equation}
\label{equation:80-q0-q1-q2-q3-q4}
\left\{\aligned%
&q_0=t_0+t_1+t_2+t_3+t_4,\\
&q_1=\xi_5^4t_4+\xi_5^3t_2+\xi_5^2t_5+\xi_5t_1+t_0,\\
&q_2=\xi_5^4t_5+\xi_5^3t_4+\xi_5^2t_1+\xi_5t_2+t_0,\\
&q_3=\xi_5^4t_2+\xi_5^3t_1+\xi_5^2t_4+\xi_5t_5+t_0,\\
&q_4=\xi_5^4t_1+\xi_5^3t_5+\xi_5^2t_2+\xi_5t_4+t_0.\\
\endaligned
\right.
\end{equation}
Here $\xi_5$ is a primitive fifth root of unity.
Note that $q_i$ is the eigenvector of $T$ that corresponds to the eigenvalue $\xi_5^i$.
The five $\overline{G}_{80}$-invariant quartic surfaces in $\mathbb{P}^3$ are the quartic surfaces $q_0=0$, $q_1=0$, $q_2=0$, $q_3=0$ and $q_4=0$.
We will study properties of these surfaces in~\S\ref{section:80-160-320}.
The polynomial
$$
q_0=\frac{1}{2}\Big(w^4+y^4+z^4+x^4\Big)+3x^2y^2+3z^2w^2+3x^2z^2+3y^2w^2-3x^2w^2-3y^2z^2.
$$
is the only $G_{80}$-invariant homogeneous quartic polynomial in $\mathbb{C}[x,y,z,w]$.

\begin{lemma}
\label{lemma:80-sextic-surfaces}
There are exactly four $\overline{G}_{80}$-invariant surfaces in $\mathbb{P}^3$ of degree $6$.
Moreover, they are irreducible and reduced.
\end{lemma}

\begin{proof}
One can find all irreducible representations of the group $G_{80}$.
They can be described as follows:
$10$ one-dimensional representations,
$10$ four-dimensional representations,
and~$6$ five-dimensional representations.
Similarly, one can compute the sixth symmetric power of every four-dimensional representation.
These computations imply that  the sixth symmetric power of every four-dimensional representations
contains exactly $4$ different one-dimensional subrepresentations.
Geometrically, this means that there are exactly four $\overline{G}_{80}$-invariant surfaces in $\mathbb{P}^3$ of degree $6$.
Since $\mathbb{P}^3$ does not contain $\overline{G}_{80}$-invariant surfaces of degree $1$, $2$ and $3$ by Lemma~\ref{lemma:80-144-quadrics},
each $\overline{G}_{80}$-invariant surface of degree $6$ must be  irreducible and reduced.
We used \cite{GAP} to perform these computations.
\end{proof}

One can find explicit equations of the $\overline{G}_{80}$-invariant sextic surfaces in $\mathbb{P}^3$.
However, we decided not to do this to keep the paper shorter.

\section{The group of order $144$}
\label{section:144}

Let us use notation and assumptions of~\S\ref{section:Heisenberg}.
In this section, we present basic facts about $\overline{G}_{144}$-orbits,
$\overline{G}_{144}$-invariant curves and $\overline{G}_{144}$-invariant surfaces in $\mathbb{P}^3$.

Let us start with studying basic group-theoretic properties of the group $\overline{G}_{144}$.
Recall that $\overline{G}_{144}$ is isomorphic to
$$
\mathfrak{A}_4\times\mathfrak{A}_4\cong \mumu_2^4\rtimes\big(\mumu_3\times\mumu_3\big).
$$
Note that the isomorphism $\overline{G}_{144}\cong\mathfrak{A}_4\times\mathfrak{A}_4$
is uniquely defined up to permutations of the factors and automorphisms of the group $\mathfrak{A}_4$.
We will fix such an isomorphism until the end of this section.
Denote by $\mathfrak{A}_4^{\Delta}$ the \emph{diagonal} subgroup in $\overline{G}_{144}$,
i.e. the subgroup that consists of the elements $(g,g)$, where $g\in\mathfrak{A}_4$.
Similarly, denote by $\mathfrak{A}_4^{\nabla}$ the \emph{twisted diagonal} subgroup in $\overline{G}_{144}$,
i.e. the subgroup that consists of the elements $(g,\sigma(g))$,
where~\mbox{$g\in\mathfrak{A}_4$}, and $\sigma$ is an outer automorphism of $\mathfrak{A}_4$ that is given by conjugation with a fixed odd element of $\mathfrak{S}_4$.
By a \emph{basic} subgroup in $\overline{G}_{144}$, we will mean any subgroup contained in one of the factors.
Similarly, we say that a subgroup in $\overline{G}_{144}$ is of \emph{product type} if it is a product of two basic subgroups.

\begin{lemma}
\label{lemma:144-subgroups}
Let $\Gamma$ be a subgroup in $\overline{G}_{144}$. Then the following assertions hold.
\begin{itemize}
\item[(o)] If $\Gamma$ is cyclic, then $|\Gamma|\in\{1,2,3,6\}$.

\item[(i)] The group $\overline{\mathrm{H}}$ contains all $2$-subgroups of the group $\overline{G}_{144}$.

\item[(ii)] If $|\Gamma|$ is divisible by $16$, then $\Gamma$ contains the subgroup $\overline{\mathrm{H}}$.

\item[(iii)] If $|\Gamma|=9$, then $\Gamma\cong\mumu_3^2$ and $\Gamma$ is of product type.

\item[(iv)] One has $|\Gamma|\ne 18$.

\item[(v)] If $|\Gamma|=6$, then $\Gamma\cong\mumu_6\cong\mumu_2\times\mumu_3$, and $\Gamma$ is of product type.
Moreover, there are just two such subgroups up to conjugation.
Furthermore, every subgroup in $\overline{G}_{144}$ of order $12$ that contains $\Gamma$ is isomorphic to $\mumu_2^2\times\mumu_3$ and is of product type.

\item[(vi)] If $|\Gamma|=12$ and $\Gamma$ is not conjugate to $\mathfrak{A}_4^{\Delta}$ and $\mathfrak{A}_4^{\nabla}$,
then the group $\Gamma$ contains a basic subgroup $\mumu_2^2$.

\item[(vii)] If $\Gamma\ne\overline{G}_{144}$, then $|\Gamma|\leqslant 48$.

\item[(viii)] If $|\Gamma|>16$, then there is a surjective homomorphism $\Gamma\to\mathfrak{A}_4$.
\end{itemize}
\end{lemma}

\begin{proof}
Assertion (o) is obvious.
The subgroup $\overline{\mathrm{H}}$ is the normal Sylow $2$-subgroup of the group  $\overline{G}_{144}$.
This implies assertion (i).
Assertion (ii) follows from (i) and Sylow theorems.
Assertion (iii) also follows from Sylow theorems.

Let $\pi_1\colon\overline{G}_{144}\to\mathfrak{A}_4$ and $\pi_2\colon\overline{G}_{144}\to\mathfrak{A}_4$ be the two projections to the factors.
If $|\Gamma|=18$, then the subgroups $\pi_1(\Gamma)$ and $\pi_2(\Gamma)$ are proper subgroups,
so that at least one of them has order greater than $4$, which is impossible.
This proves assertion (iv).

Similarly, if $|\Gamma|=6$,
then, up to permutation, the groups $\pi_1(\Gamma)$ and $\pi_2(\Gamma)$ are isomorphic to $\mumu_2$ and $\mumu_3$, respectively.
In this case, we have $\Gamma\cong\mumu_2\times\mumu_3\cong\mumu_6$, and the conjugacy class of $\Gamma$ is determined by the choice of $\pi_i$ such that $\pi_i(\Gamma)\cong\mumu_2$.
Let $\Upsilon$ be a subgroup in $\overline{G}_{144}$ of order $12$ that contains $\Gamma$.
We may assume that  $\pi_1(\Gamma)\cong\mumu_2$ and $\pi_2(\Gamma)\cong\mumu_3$.
Then $\pi_2(\Upsilon)\cong\mumu_3$, because $\mumu_3$ is a maximal proper subgroup in $\mathfrak{A}_4$.
This easily implies assertion (v).

To prove assertion (vi), we suppose that $|\Gamma|=12$.
If $\pi_1(\Gamma)=\mathfrak{A}_4$ and $\pi_2(\Gamma)=\mathfrak{A}_4$,
then~$\Gamma$ is conjugate to either $\mathfrak{A}_4^{\Delta}$ or $\mathfrak{A}_4^{\nabla}$.
Thus, we may assume that $\pi_1(\Gamma)\ne\mathfrak{A}_4$.
Then the kernel of $\pi_1\vert_{\Gamma}$ contains a basic subgroup $\mumu_2^2$.

To prove assertion (vii), we have to show that the index of $\Gamma$ in $\overline{G}_{144}$ is not $2$.
Suppose that it is. Then $\Gamma$ is normal. Thus, the intersection of $\Gamma$ with each factor in
$\overline{G}_{144}\cong\mathfrak{A}_4\times\mathfrak{A}_4$ is either $\mathfrak{A}_4$ or a subgroup of index $2$ in $\mathfrak{A}_4$.
Since $\mathfrak{A}_4$ does not contain subgroups of index~$2$, we see that $\Gamma$ contains both factors,
so it is the whole group $\overline{G}_{144}$, which is absurd.

To prove assertion (viii), we suppose that $|\Gamma|>16$.
If  $\pi_1(\Gamma)=\mathfrak{A}_4$ or $\pi_2(\Gamma)=\mathfrak{A}_4$, then we are done.
Otherwise, both $\pi_1(\Gamma)$ and $\pi_2(\Gamma)$ are proper subgroups of $\mathfrak{A}_4$,
so that their orders are at most $4$.
This is impossible, because $|\Gamma|>16$.
\end{proof}

To study $\overline{G}_{144}$-invariant curves in $\mathbb{P}^3$, we need the following result,
which is a simple consequence of the Riemann--Hurwitz formula.

\begin{lemma}
\label{lemma:144-sporadic-genera}
Let $C$ be a smooth irreducible curve of genus $g\leqslant 13$
with a faithful action of the group $\overline{G}_{144}$.
Then $g\in\{8,13\}$, and the curve $C$ is not hyperelliptic.
Furthermore, if~\mbox{$\Omega\subset C$} is a $\overline{G}_{144}$-orbit, then $|\Omega|\in\{24,48,72,144\}$.
Finally, the possible numbers $a_i$ of $\overline{G}_{144}$-orbits of length $i\in\{24,48,72\}$ in $C$
are contained in the following table:
\begin{center}\renewcommand\arraystretch{1.1}
\begin{tabular}{|c||c|c|c|}
\hline
$g$ &  $8$ & $13$ & $13$ \\
\hline\hline
$a_{24}$ & $0$ & $1$ & $2$ \\
\hline
$a_{48}$ & $0$ & $2$ & $0$\\
\hline
$a_{72}$ & $3$ & $0$ & $1$\\
\hline
\end{tabular}
\end{center}
\end{lemma}

\begin{proof}
The assertion about the lengths of $\overline{G}_{144}$-orbits follows from Lemma~\ref{lemma:144-subgroups}(o),
since the stabilizers in $\overline{G}_{144}$ of points in $C$ are cyclic (see \cite[Lemma~5.1.4]{CheltsovShramov}).
By Lemma~\ref{lemma:Heisenberg-elliptic}, the group $\overline{G}_{144}$ cannot act faithfully on $\mathbb{P}^1$ and on a smooth elliptic curve,
so that $g\geqslant 2$.

Suppose that the curve $C$ is hyperelliptic. Since the group $\overline{G}_{144}$
does not contain normal subgroups of order $2$, it does not contain the hyperelliptic involution of $C$,
and we obtain a faithful action of the group $\overline{G}_{144}$ on $\mathbb{P}^1$.
The latter is impossible by Lemma~\ref{lemma:Heisenberg-elliptic}.

Let $\widehat{C}=C\slash\overline{G}_{144}$. Then $\widehat{C}$ is a smooth curve.
Let~$\hat{g}$ be the genus of the curve $\widehat{C}$.
Then the Riemann--Hurwitz formula gives
$$
2g-2=144\big(2\hat{g}-2\big)+72a_{72}+96a_{48}+120a_{24}.
$$
Since $a_k\geqslant 0$ and $g\leqslant 13$, one has $\hat{g}=0$, so that
$$
2g-2=-288+72a_{72}+96a_{48}+120a_{24},
$$
Going through the possible values of $g$, and solving this equation
case by case we obtain the required result.
\end{proof}

Denote by $\mathcal{Q}$ the smooth quadric in $\mathbb{P}^3$ that is given by
$$
wy+xz=0,
$$
which is the quadric $\mathcal{Q}_3$ in the notation of~\S\ref{section:Heisenberg}.
Then $\mathcal{Q}$ is $\overline{G}_{144}$-invariant.
We will see in Lemmas~\ref{lemma:144-orbits} and~\ref{lemma:144-reducible-curves}
that this quadric contains most of $\overline{G}_{144}$-orbits of small length in $\mathbb{P}^3$
and most of $\overline{G}_{144}$-irreducible curves in~$\mathbb{P}^3$ that have small degree.

Observe that the action of the group $\overline{G}_{144}\cong\mathfrak{A}_4\times\mathfrak{A}_4$ on the surface $\mathcal{Q}\cong\mathbb{P}^1\times\mathbb{P}^1$
is just the product action given by the natural action of the group $\mathfrak{A}_4$ on $\mathbb{P}^1$ induced by a two-dimensional irreducible representation of the binary tetrahedral group~$2.\mathfrak{A}_4$.
Denote this representation by $\mathbb{V}_2$.
Then there is a $\overline{G}_{144}$-equivariant identification
\begin{equation}
\label{equation:product}
\mathbb{P}^3\cong\mathbb{P}\Big(\mathbb{W}_4\Big),
\end{equation}
where $\mathbb{W}_4$ is the four-dimensional representation of the group $2.\mathfrak{A}_4\times2.\mathfrak{A}_4$ in
the vector space of $2\times 2$-matrices such that the first (respectively, the second) factor
$$
2.\mathfrak{A}_4\subset\mathrm{GL}(\mathbb{V}_2)\cong\mathrm{GL}_2(\mathbb{C})
$$
acts by left (respectively, right) multiplications.
Then the points of the quadric $\mathcal{Q}$ correspond to $2\times 2$-matrices of rank $1$ in $\mathbb{W}_4$.

Since $\mathbb{P}^1$ contains two $\mathfrak{A}_4$-orbits of length $4$ and one orbit of length $6$,
this gives four $\overline{G}_{144}$-irreducible curves in $\mathcal{Q}$ that are unions of $4$ disjoint lines,
and two $\overline{G}_{144}$-irreducible curves in $\mathcal{Q}$ that are unions of $6$ disjoint lines.
Denote the former four curves by $\mathcal{L}_4^1$, $\mathcal{L}_4^2$, $\mathcal{L}_4^3$ and~$\mathcal{L}_4^4$,
and denote the latter two curves by $\mathcal{L}_6^1$ and $\mathcal{L}_6^2$.
Without loss of generality, we may assume that
$$
\mathcal{L}_4^1\cap\mathcal{L}_4^2=\mathcal{L}_4^1\cap\mathcal{L}_6^1=\mathcal{L}_4^3\cap\mathcal{L}_4^4=\mathcal{L}_4^3\cap\mathcal{L}_6^2=\varnothing.
$$
In other words, on $\mathcal{Q}\cong\mathbb{P}^1\times\mathbb{P}^1$, the curves
$\mathcal{L}_4^1$ and $\mathcal{L}_4^2$ have bi-degree $(4,0)$,
while the curves $\mathcal{L}_4^3$ and $\mathcal{L}_4^4$ have bi-degree $(0,4)$.
Similarly, the curve $\mathcal{L}_6^1$ has bi-degree $(6,0)$,
and the  curve $\mathcal{L}_6^2$ has bi-degree $(0,6)$.
Let
$$
\Sigma_{16}^1=\mathcal{L}_4^1\cap\mathcal{L}_4^3, \quad \Sigma_{16}^2=\mathcal{L}_4^1\cap\mathcal{L}_4^4,
\quad \Sigma_{16}^3=\mathcal{L}_4^2\cap\mathcal{L}_4^3, \quad \Sigma_{16}^4=\mathcal{L}_4^2\cap\mathcal{L}_4^4.
$$
Then $\Sigma_{16}^1$, $\Sigma_{16}^2$, $\Sigma_{16}^3$ and $\Sigma_{16}^4$ are $\overline{G}_{144}$-orbits of length $16$.
Similarly, let
$$
\Sigma_{24}^1=\mathcal{L}_4^1\cap\mathcal{L}_6^2, \quad  \Sigma_{24}^2=\mathcal{L}_4^2\cap\mathcal{L}_6^2, \quad
\Sigma_{24}^3=\mathcal{L}_4^3\cap\mathcal{L}_6^1, \quad \Sigma_{24}^4=\mathcal{L}_4^4\cap\mathcal{L}_6^1.
$$
Then the subsets $\Sigma_{24}^1$, $\Sigma_{24}^2$, $\Sigma_{24}^3$, $\Sigma_{24}^4$ are $\overline{G}_{144}$-orbits of length $24$.
Finally, let $\Sigma_{36}=\mathcal{L}_6^1\cap\mathcal{L}_6^2$.
Then the subset $\Sigma_{36}$ is a $\overline{G}_{144}$-orbit of length $36$.
We summarize the intersections of the curves
$\mathcal{L}_4^1$, $\mathcal{L}_4^2$, $\mathcal{L}_4^3$, $\mathcal{L}_4^4$, $\mathcal{L}_6^1$ and $\mathcal{L}_6^2$
in the following table:
\begin{center}\renewcommand\arraystretch{1.8}
\begin{tabular}{|c||c|c|c|c|c|c|}
  \hline
   $\qquad$& $\mathcal{L}_4^1$ & $\mathcal{L}_4^2$ & $\mathcal{L}_4^3$ & $\mathcal{L}_4^4$ & $\mathcal{L}_6^1$ & $\mathcal{L}_6^2$ \\
  \hline
  \hline
  $\mathcal{L}_4^1$ & $\mathcal{L}_4^1$ & $\varnothing$ & $\Sigma_{16}^1$ & $\Sigma_{16}^2$ & $\varnothing$ & $\Sigma_{24}^1$ \\
  \hline
  $\mathcal{L}_4^2$ & $\varnothing$ & $\mathcal{L}_4^2$ & $\Sigma_{16}^3$ & $\Sigma_{16}^4$ & $\varnothing$ & $\Sigma_{24}^2$\\
  \hline
  $\mathcal{L}_4^3$ & $\Sigma_{16}^1$ & $\Sigma_{16}^3$ & $\mathcal{L}_4^3$ & $\varnothing$ & $\Sigma_{24}^3$ & $\varnothing$ \\
  \hline
  $\mathcal{L}_4^4$ & $\Sigma_{16}^2$ & $\Sigma_{16}^4$ & $\varnothing$ & $\mathcal{L}_4^4$ & $\Sigma_{24}^4$ & $\varnothing$ \\
  \hline
  $\mathcal{L}_6^1$ & $\varnothing$ & $\varnothing$ & $\Sigma_{24}^3$ & $\Sigma_{24}^4$ & $\mathcal{L}_6^1$ & $\Sigma_{36}$\\
  \hline
  $\mathcal{L}_6^2$ & $\Sigma_{24}^1$ & $\Sigma_{24}^2$ & $\varnothing$ & $\varnothing$  &$\Sigma_{36}$ & $\mathcal{L}_6^2$\\
  \hline
\end{tabular}
\end{center}

\begin{lemma}
\label{lemma:144-Q-orbits}
Let $\Sigma$ be a $\overline{G}_{144}$-orbit in $\mathcal{Q}$ such that $|\Sigma|\leqslant 36$.
Then $\Sigma$ is one of the $\overline{G}_{144}$-orbits  $\Sigma_{16}^1$, $\Sigma_{16}^2$, $\Sigma_{16}^3$, $\Sigma_{16}^4$,
$\Sigma_{24}^1$, $\Sigma_{24}^2$, $\Sigma_{24}^3$, $\Sigma_{24}^4$,  $\Sigma_{36}$.
\end{lemma}

\begin{proof}
This follows from the fact that $\mathbb{P}^1$ contains exactly two $\mathfrak{A}_4$-orbits of length $4$,
exactly one orbit of length $6$,
and the remaining $\mathfrak{A}_4$-orbits in $\mathbb{P}^1$ are of length $12$.
\end{proof}

By construction, the smooth quadric $\mathcal{Q}$ contains the curves
$\mathcal{L}_4^1$, $\mathcal{L}_4^2$, $\mathcal{L}_4^3$, $\mathcal{L}_4^4$, $\mathcal{L}_6^1$ and $\mathcal{L}_6^2$.
The same construction gives infinitely many $\overline{G}_{144}$-irreducible curves in $\mathcal{Q}$ that are unions of $12$ disjoint lines.
Moreover, we have

\begin{lemma}
\label{lemma:144-Q-curves}
Let $C$ be a $\overline{G}_{144}$-irreducible curve in $\mathcal{Q}$ such that $\mathrm{deg}(C)\leqslant 23$.
Then either~$C$ is a disjoint union of $12$ lines, or $C$ is one of the curves $\mathcal{L}_4^1$, $\mathcal{L}_4^2$, $\mathcal{L}_4^3$, $\mathcal{L}_4^4$, $\mathcal{L}_6^1$, $\mathcal{L}_6^2$.
\end{lemma}

\begin{proof}
The curve $C$ is a curve of bi-degree $(m,n)$ on $\mathcal{Q}\cong\mathbb{P}^1\times\mathbb{P}^1$ such that $m+n\leqslant 23$.
The required assertion can be checked by a direct computation using the identification~\eqref{equation:product},
however, we prefer to provide a more geometric proof.

Without loss of generality, we may assume that $m\leqslant n$.
Then $m\leqslant 11$.
Suppose that $C$ is none of the curves $\mathcal{L}_4^1$, $\mathcal{L}_4^2$, $\mathcal{L}_4^3$, $\mathcal{L}_4^4$, $\mathcal{L}_6^1$, $\mathcal{L}_6^2$.
Suppose also that $C$ is not a disjoint union of $12$ lines.
Then $m\geqslant 1$. Let us seek for a contradiction.

Let $P$ be a point in $C$ that is not contained in the curves  $\mathcal{L}_4^1$, $\mathcal{L}_4^2$, $\mathcal{L}_4^3$, $\mathcal{L}_4^4$, $\mathcal{L}_6^1$, $\mathcal{L}_6^2$.
Let~$L$ be the line in $\mathcal{Q}$ that is a curve of bi-degree $(0,1)$ that passes through $P$.
Then $L$ is not an irreducible component of the curve $C$, since $C$ is $\overline{G}_{144}$-irreducible.
Let $\Gamma$ be the subgroup $\mathfrak{A}_4\times\mathrm{Id}$ in $\overline{G}_{144}$ (so that $\Gamma$ is one of the two basic subgroups~$\mathfrak{A}_4$). Then $L$ and $C$ are $\Gamma$-invariant,
and $\Gamma$ acts faithfully on the line~$L$.
Hence, the intersection $L\cap C$ is a union of $\Gamma$-orbits of length $L\cdot C=m\leqslant 11$.
On the other hand, all such orbits are contained in the union
$\mathcal{L}_4^1\cup\mathcal{L}_4^2\cup\mathcal{L}_4^3\cup\mathcal{L}_4^4\cup\mathcal{L}_6^1\cup\mathcal{L}_6^2$
by Lemma~\ref{lemma:144-Q-orbits},
which gives us a contradiction.
\end{proof}

\begin{corollary}
\label{corollary:144-curves-very-small-degree}
Let $C$ be a $\overline{G}_{144}$-invariant curve in $\mathbb{P}^3$ such that $\mathrm{deg}(C)<8$.
Then $C$ is one of the curves $\mathcal{L}_4^1$, $\mathcal{L}_4^2$, $\mathcal{L}_4^3$, $\mathcal{L}_4^4$, $\mathcal{L}_6^1$, $\mathcal{L}_6^2$.
\end{corollary}

\begin{proof}
If $C$ is not contained in the quadric $\mathcal{Q}$, then
$$
|C\cap\mathcal{Q}|\leqslant C\cdot\mathcal{Q}=2\mathrm{deg}(C)\leqslant 14,
$$
which is impossible by Lemma~\ref{lemma:144-Q-orbits}.
\end{proof}

\begin{corollary}
\label{corollary:144-Q-L4-L6-L12-mult}
Let $C$ be a $\overline{G}_{144}$-irreducible curve in $\mathcal{Q}$,
let $\mathcal{D}$ be a (non-empty) mobile $\overline{G}_{144}$-invariant linear system on $\mathbb{P}^3$,
let $D$ be a general surface in $\mathcal{D}$,
and let $n$ be a positive integer such that $\mathcal{D}\sim\mathcal{O}_{\mathbb{P}^3}(n)$.
Write
$$
D\big\vert_{\mathcal{Q}}=mC+\Delta,
$$
where $m$ is a non-negative integer,
and $\Delta$ is an effective divisor on $\mathcal{Q}$ whose support does not contain irreducible components of the curve $C$.
Then $m\leqslant\frac{n}{4}$.
\end{corollary}

\begin{proof}
Suppose that $m>\frac{n}{4}$.
The curve $C$ is a divisor of bi-degree $(a,b)$ on $\mathcal{Q}\cong\mathbb{P}^1\times\mathbb{P}^1$.
But $mC+\Delta$ is a divisor of bi-degree $(n,n)$, so that $a<4$ and $b<4$, which is impossible by Lemma~\ref{lemma:144-Q-curves}.
\end{proof}

\begin{corollary}
\label{corollary:144-L4-L6-L12-mult}
Let $C$ be a $\overline{G}_{144}$-irreducible curve in $\mathcal{Q}$,
let $\mathcal{D}$ be a (non-empty) mobile $\overline{G}_{144}$-invariant linear system on $\mathbb{P}^3$,
and let $n$ be a positive integer such that $\mathcal{D}\sim\mathcal{O}_{\mathbb{P}^3}(n)$.
Then $\mathrm{mult}_{C}(\mathcal{D})\leqslant\frac{n}{4}$.
\end{corollary}

Denote by $\Sigma_{12}$ the subset in $\mathbb{P}^3$ that consists of the $12$ points
\begin{multline*}
[0:1:0:1], [0:-1:0:1], [1:i:i:1], [1:-i:-i:1], [1:-i:i:-1], [1:i:-i:-1],\\
[1:-1:i:-i], [1:-1:-i:i], [1:1:i:i], [1:1:-i:-i], [1:0:1:0], [1:0:-1:0].
\end{multline*}
Similarly, denote by $\Sigma_{12}^\prime$ the subset in $\mathbb{P}^3$ that consists of the $12$ points
\begin{multline*}
[1:1:1:1], [1:-1:-1:1], [1:i:1:-i], [1:-i:1:i], [1:i:-1:i], [1:-i:-1:-i],\\
[1:0:i:0], [1:0:-i:0], [0:i:0:1], [0:-i:0:1], [1:1:-1:-1], [1:-1:1:-1].
\end{multline*}
Explicit computations show that $\Sigma_{12}$ and $\Sigma_{12}^\prime$ are $\overline{G}_{144}$-orbits.
This also follows from Lemma~\ref{lemma:144-quartics-singular} below.

\begin{lemma}
\label{lemma:144-orbit-12}
The subsets $\Sigma_{12}$ and $\Sigma_{12}^\prime$ are the only $\overline{G}_{144}$-orbits in $\mathbb{P}^3$ of length $12$.
\end{lemma}

\begin{proof}
Let $\Upsilon$ be a basic subgroup $\mumu_2^2\subset\overline{G}_{144}$.
Recall that the action of the group $\Upsilon$ on~\mbox{$\mathcal{Q}\cong\mathbb{P}^1\times\mathbb{P}^1$} is a product action,
with trivial group acting on one of the factors.
Let $\widetilde{\Upsilon}$ be the preimage of $\Upsilon$ via the canonical projection $2.\mathfrak{A}_4\to\mathfrak{A}_4$.
Then the restriction of the two-dimensional representation $\mathbb{V}_2$ to $\widetilde{\Upsilon}$
is an irreducible two-dimensional  representations of the group $\widetilde{\Upsilon}$.
The restriction of the four-dimensional representation $\mathbb{W}_4$ of the group $2.\mathfrak{A}_4\times2.\mathfrak{A}_4$
to the group  $\widetilde{\Upsilon}$ splits as sum of two isomorphic irreducible two-dimensional representations.
Thus, $\Upsilon$ does not have fixed points in $\mathbb{P}^3$.

Let $\Sigma$ be a $\overline{G}_{144}$-orbit in $\mathbb{P}^3$ of length $12$,
and let $\Gamma$ be the stabilizer in $\overline{G}_{144}$ of a point in $\Sigma$.
Then $|\Gamma|=12$. We have seen that $\Gamma$ does not contain $\Upsilon$.
Thus, it follows from Lemma~\ref{lemma:144-subgroups}(vi) that $\Gamma$ is conjugate to either $\mathfrak{A}_4^\Delta$ or $\mathfrak{A}_4^\nabla$.
Using the identification \eqref{equation:product}, we see that each of the latter groups
has a unique fixed point in $\mathbb{P}^3$.
Since we already know two  $\overline{G}_{144}$-orbits in $\mathbb{P}^3$ of length $12$,
namely $\Sigma_{12}$ and $\Sigma_{12}^\prime$, the assertion follows.
\end{proof}

Now we are ready to describe $\overline{G}_{144}$-orbits in $\mathbb{P}^3$ of small length.

\begin{lemma}
\label{lemma:144-orbits}
Let $\Sigma$ be a $\overline{G}_{144}$-orbit in $\mathbb{P}^3$ such that $|\Sigma|\leqslant 35$.
Then $\Sigma$ is one of the $\overline{G}_{144}$-orbits $\Sigma_{12}$, $\Sigma_{12}^\prime$, $\Sigma_{16}^1$,
$\Sigma_{16}^2$, $\Sigma_{16}^3$, $\Sigma_{16}^4$, $\Sigma_{24}^1$, $\Sigma_{24}^2$, $\Sigma_{24}^3$, $\Sigma_{24}^4$.
\end{lemma}

\begin{proof}
One has $|\Sigma|>4$, because the group $\overline{G}_{144}$ is primitive.
This implies that
$$
\big|\Sigma\big|\in\big\{6,8,9,12,16,18,24\big\}.
$$

Let $\Gamma$ be a stabilizer in $\overline{G}_{144}$ of a point in $\Sigma$.
If $|\Sigma|\in\{6,9,18\}$, then $|\Gamma|$ is divisible by $8$,
so that $\Gamma$ contains a subgroup $\mumu_2^3\subset\overline{\mathrm{H}}$
by Sylow theorem and Lemma~\ref{lemma:144-subgroups}(i).
This is impossible by Lemma~\ref{lemma:Heisenberg-subgroup}(iii).
If $|\Sigma|=8$, then $|\Gamma|=18$,
which is impossible by Lemma~\ref{lemma:144-subgroups}(iv).
If $|\Sigma|=12$, then the required assertion follows from Lemma~\ref{lemma:144-orbit-12}.

Suppose that $|\Sigma|=16$. Then $|\Gamma|=9$, so that $\Gamma\cong\mumu_3^2$ by Lemma~\ref{lemma:144-subgroups}(iii).
Recall that the action of the group $\Gamma\cong\mumu_3\times\mumu_3$ on $\mathcal{Q}\cong\mathbb{P}^1\times\mathbb{P}^1$ is a product action.
Let $\widetilde{\mumu}_3$ be the preimage of $\mumu_3$ via the canonical projection $2.\mathfrak{A}_4\to\mathfrak{A}_4$.
Then the restriction of the two-dimensional representation $\mathbb{V}_2$ to $\widetilde{\mumu}_3$
splits as a sum of two distinct one-dimensional representations of the group $\widetilde{\mumu}_3$.
The restriction of the four-dimensional representation $\mathbb{W}_4$ to the group
$\widetilde{\mumu}_3\times\widetilde{\mumu}_3$ splits as sum of four pairwise non-isomorphic one-dimensional subrepresentations,
which are generated by $2\times 2$-matrices of rank $1$.
Since matrices of rank $1$ correspond to the points of the quadric $\mathcal{Q}$,
we see that $\Sigma\subset\mathcal{Q}$, so that
$\Sigma$ is one of the $\overline{G}_{144}$-orbits $\Sigma_{16}^1$, $\Sigma_{16}^2$, $\Sigma_{16}^3$, $\Sigma_{16}^4$ by Lemma~\ref{lemma:144-Q-orbits}.

Suppose that $|\Sigma|=24$. Then $\Gamma\cong\mumu_6\cong\mumu_2\times\mumu_3$ by Lemma~\ref{lemma:144-subgroups}(v).
Now arguing as in the previous case, we see that $\Sigma$ is one of the $\overline{G}_{144}$-orbits $\Sigma_{24}^1$, $\Sigma_{24}^2$, $\Sigma_{24}^3$, $\Sigma_{24}^4$.
\end{proof}

Let $p_1$, $p_2$, $p_3$ and $p_4$ be the quartic polynomials  in \eqref{equation:144-p0-p1-p2-p3-p4}.
Denote by $S_1$, $S_2$, $S_3$, $S_4$ the quartic surfaces $\mathbb{P}^3$ that are given by $p_1=0$, $p_2=0$, $p_3=0$ and $p_4=0$, respectively.
Then $S_1$, $S_2$, $S_3$, $S_4$ are all $\overline{G}_{144}$-invariant irreducible quartic surfaces in $\mathbb{P}^3$ by Lemma~\ref{lemma:80-144-quartic-surfaces}.
Moreover, explicit computations imply that
\begin{equation}
\label{equation:S1-S2-S3-S4-12-12}
\Sigma_{12}\in S_1\cap S_2, \quad \Sigma_{12}^\prime\in S_3\cap S_4, \quad \Sigma_{12}\not\subset S_3\cup S_4, \quad \Sigma_{12}^\prime\not\subset S_1\cup S_2.
\end{equation}

In fact, we can say more.

\begin{lemma}
\label{lemma:144-quartics-singular}
One has $\mathrm{Sing}(S_1)=\mathrm{Sing}(S_2)=\Sigma_{12}$ and $\mathrm{Sing}(S_3)=\mathrm{Sing}(S_4)=\Sigma_{12}^\prime$.
Moreover, all singular points of the surfaces $S_1$, $S_2$, $S_3$ and $S_4$ are ordinary double points.
\end{lemma}

\begin{proof}
It is enough to show the required assertions for the surface $S_1$, since the proof is identical in the remaining cases.
We already know that $S_1$ is irreducible.
Taking a general plane section of the surface $S_1$,
we see that either $S_1$ has isolated singularities, or~\mbox{$\mathrm{Sing}(S_1)$} contains a $\overline{G}_{144}$-invariant curve of degree at most $3$.
The latter is impossible by Corollary~\ref{corollary:144-curves-very-small-degree}.
Thus, the surface $S_1$ has at most isolated singularities.

To see that $S_1$ is singular at every point of the $\overline{G}_{144}$-orbit $\Sigma_{12}$,
one can simply take partial derivatives of the polynomial $p_1$, and plug the point $[0:1:0:1]$ into them.
Alternatively, one can use the fact that the stabilizer in $\overline{G}_{144}$ of every point in  $\Sigma_{12}$ is isomorphic to $\mathfrak{A}_4$.
Since the group $\mathfrak{A}_4$ does not have faithful two-dimensional representations, the surface $S_1$ must be singular at every point of $\Sigma_{12}$ by \cite[Lemma~4.4.1]{CheltsovShramov}.

Thus, we see that $S_1$ has isolated singularities and $\Sigma_{12}\subset\mathrm{Sing}(S_1)$.
Moreover, the surface $S_1$ cannot have more that two non-Du Val singular points by~\mbox{\cite[Theorem~1]{Umezu81}}.
Thus, the surface $S_1$ has at most Du Val singularities by Lemma~\ref{lemma:144-orbits}.

Let $f\colon\widetilde{S}_1\to S_1$ be the minimal resolution of singularities of the surface $S_1$.
Then $\widetilde{S}_1$ is a smooth~$K3$ surface. Then
$$
\mathrm{rk}\,\mathrm{Pic}\big(\widetilde{S}_1\big)\geqslant\mathrm{rk}\,\mathrm{Pic}\big(S_1\big)+\big|\mathrm{Sing}(S_1)\big|=\mathrm{rk}\,\mathrm{Pic}\big(S_1\big)+12+\big|\mathrm{Sing}(S_1)\setminus\Sigma_{12}\big|.
$$
On the other hand, the rank of the Picard group of a smooth~$K3$ surface is at most $20$,
so that $\mathrm{Sing}(S_1)=\Sigma_{12}$, since $S_1$ does not contain $\overline{G}_{144}$-orbits of length less than $12$ by Lemma~\ref{lemma:144-orbits}.
Moreover, the same arguments imply that every singular point of the surface $S_1$ is an ordinary double point,
since otherwise we would have
$$
20\geqslant\mathrm{rk}\,\mathrm{Pic}\big(\widetilde{S}_1\big)\geqslant\mathrm{rk}\,\mathrm{Pic}\big(S_1\big)+2\big|\mathrm{Sing}(S_1)\big|=\mathrm{rk}\,\mathrm{Pic}\big(S_1\big)+24\geqslant 25,
$$
which is absurd.
\end{proof}

\begin{corollary}
\label{corollary:144-cubics-12-points}
Let $\Sigma$ be one of the $\overline{G}_{144}$-orbits $\Sigma_{12}$ or $\Sigma_{12}^\prime$
and let $\mathcal{M}$ be a linear system that consists of cubic surfaces in $\mathbb{P}^3$ containing $\Sigma$.
Then $\Sigma$ is the base locus of the linear system $\mathcal{M}$.
\end{corollary}

\begin{proof}
By Lemma~\ref{lemma:144-quartics-singular}, we have $\Sigma=\mathrm{Sing}(S_i)$ for some $i\in\{1,2,3,4\}$.
Recall that the surface $S_i$ is given by the equation $p_i(x,y,z,w)=0$, where
$p_i$ is the quartic polynomial in \eqref{equation:144-p0-p1-p2-p3-p4}.
Thus, the linear system $\mathcal{M}$ contains the linear system that consists of cubic surfaces
$$
\lambda_0\frac{\partial p_i}{\partial x}+\lambda_1\frac{\partial p_i}{\partial y}+\lambda_2\frac{\partial p_i}{\partial z}+\lambda_3\frac{\partial p_i}{\partial w}=0,
$$
where $[\lambda_0:\lambda_1:\lambda_2:\lambda_3]\in\mathbb{P}^3$. The base locus of the latter system is $\mathrm{Sing}(S_i)=\Sigma$.
\end{proof}

\begin{remark}
\label{remark:144-quartics}
Let $H$ be the plane section of the quadric $\mathcal{Q}$.
Then there is an exact sequences of $G$-representations
$$
0\longrightarrow H^0\Big(\mathcal{O}_{\mathbb{P}^3}\big(2\big)\Big)
\longrightarrow H^0\Big(\mathcal{O}_{\mathbb{P}^3}\big(4\big)\Big)
\longrightarrow H^0\Big(\mathcal{O}_{\mathcal{Q}}\big(4H\big)\Big)\longrightarrow 0.
$$
On the other hand, the divisors $\mathcal{L}_4^1+\mathcal{L}_4^3$, $\mathcal{L}_4^1+\mathcal{L}_4^4$,
$\mathcal{L}_4^2+\mathcal{L}_4^3$, $\mathcal{L}_4^2+\mathcal{L}_4^4$ are contained in the linear system $|\mathcal{O}_{\mathcal{Q}}(4H)|$.
Thus, each intersection among
$S_1\cap\mathcal{Q}$, $S_2\cap\mathcal{Q}$, $S_3\cap\mathcal{Q}$ and $S_4\cap\mathcal{Q}$ is one of the unions
$\mathcal{L}_4^1\cup\mathcal{L}_4^3$, $\mathcal{L}_4^1\cup\mathcal{L}_4^4$, $\mathcal{L}_4^2\cup\mathcal{L}_4^3$ and $\mathcal{L}_4^2\cup\mathcal{L}_4^4$.
\end{remark}

Using this remark, we see that for every $\overline{G}_{144}$-orbit among $\Sigma_{16}^1$, $\Sigma_{16}^2$, $\Sigma_{16}^3$ and $\Sigma_{16}^4$,
there exists exactly one quartic surface among $S_1$, $S_2$, $S_3$ and $S_4$ that does not contain it.
This implies

\begin{lemma}
\label{lemma:144-curves-very-small-degree}
Let $C$ be a $\overline{G}_{144}$-invariant curve in $\mathbb{P}^3$ of degree less than $12$.
Then $C\subset\mathcal{Q}$.
\end{lemma}

\begin{proof}
We may assume that $C$ is $\overline{G}_{144}$-irreducible.
Suppose that $C$ is not contained in the quadric $\mathcal{Q}$.
Let us seek for a contradiction.
Observe that
$$
C\cdot\mathcal{Q}=2\cdot\mathrm{deg}(C)\leqslant 22.
$$
Thus, it follows from Lemma~\ref{lemma:144-Q-orbits}
that $\mathrm{deg}(C)=8$, and the intersection $C\cap\mathcal{Q}$ is one of the orbits $\Sigma_{16}^1$, $\Sigma_{16}^2$, $\Sigma_{16}^3$ or~$\Sigma_{16}^4$.
Then the curve $C$ does not contains other $\overline{G}_{144}$-orbits of length $16$ in~$\mathbb{P}^3$, because all of them lie in $\mathcal{Q}$ by Lemma~\ref{lemma:144-orbits}.
Let $S_k$ be the surface among $S_1$, $S_2$, $S_3$ and~$S_4$ that does not contain $C\cap\mathcal{Q}$.
Then $S_k$ does not contain the curve $C$.
On the other hand, we have
$$
C\cdot S_k=4\cdot\mathrm{deg}(C)=32.
$$
Hence, the intersection $C\cap S_k$ is one of the $\overline{G}_{144}$-orbits $\Sigma_{16}^1$, $\Sigma_{16}^2$, $\Sigma_{16}^3$, $\Sigma_{16}^4$ by Lemma~\ref{lemma:144-orbits}.
This is impossible, since $S_k$ does not contain the $\overline{G}_{144}$-orbit $C\cap\mathcal{Q}$,
and $C$ does not contain other $\overline{G}_{144}$-orbits of length $16$.
\end{proof}

\begin{lemma}
\label{lemma:144-reducible-curves}
Let $C$ be a reducible $\overline{G}_{144}$-irreducible curve in $\mathbb{P}^3$ of degree less than $16$.
Then $C\subset\mathcal{Q}$.
\end{lemma}

\begin{proof}
Suppose that $C$ is not contained in the quadric $\mathcal{Q}$.
Then $\mathrm{deg}(C)\geqslant 12$ by Lemma~\ref{lemma:144-curves-very-small-degree}.
Moreover, since
$$
30\geqslant 2\cdot\mathrm{deg}(C)=C\cdot\mathcal{Q},
$$
we see that $\mathrm{deg}(C)=12$ by Lemma~\ref{lemma:144-Q-orbits}.

Let $C=C_1+\ldots+C_r$, where $C_i$ is an irreducible curve of degree~$d$.
Denote by $\Gamma$ the stabilizer in $\overline{G}_{144}$ of the curve~$C_1$.
Since $rd=\mathrm{deg}(C)=12$ and $\overline{G}_{144}$ does not have subgroups of index~$2$ by Lemma~\ref{lemma:144-subgroups}(vii),
we see that either $r=3$ and $d=4$, or~\mbox{$r=4$} and~\mbox{$d=3$}, or $r=6$ and $d=2$, or $r=12$ and $d=1$.
If $d\ne 1$, then $C_1$ is not contained in a plane, because there are no $\overline{G}_{144}$-orbits in $\mathbb{P}^3$ of length at most $6$ by Lemma~\ref{lemma:144-orbits}.
In particular, the curve $C_1$ is not a conic, so that $r\ne 6$.
We also conclude that $\Gamma$ acts faithfully on $C_1$ and on its normalization in the case when $d\ne 1$.
Thus, the curve $C$ is either a union of $3$ irreducible quartic curves,
or a union of $4$ cubic curves, or a union of~$12$ lines.
Let us deal with each of these cases one by one.

Suppose that $r=3$ and $d=4$.
Then $|\Gamma|=48$, so that $\Gamma$ contains the subgroup $\overline{\mathrm{H}}$ by Lemma~\ref{lemma:144-subgroups}(ii).
On the other hand, the curve $C_1$ must be either rational or elliptic, because $d=4$ and $C_1$ is not contained in a plane.
This is impossible by Lemma~\ref{lemma:Heisenberg-elliptic}.

Suppose that $r=4$ and $d=3$. Then $C_1$ is a twisted cubic curve.
One has $|\Gamma|=36$, so that  $\Gamma$ has a surjective homomorphism to $\mathfrak{A}_4$ by Lemma~\ref{lemma:144-subgroups}(viii).
This is impossible by the classification of finite subgroups in  $\mathrm{PGL}_2(\mathbb{C})$.

We see that $r=12$ and $d=1$, so that $C_1$ is a line and $|\Gamma|=12$.
Since~\mbox{$C\cdot\mathcal{Q}=24$}, we conclude from Lemma~\ref{lemma:144-Q-orbits}
that $C\cap\mathcal{Q}$ is one of the $\overline{G}_{144}$-orbits~$\Sigma_{24}^1$, $\Sigma_{24}^2$, $\Sigma_{24}^3$ or~$\Sigma_{24}^4$.
Without loss of generality, we may assume that $C\cap\mathcal{Q}=\Sigma_{24}^1$.
We see that every irreducible component of the curve $C$ contains exactly two points in $\Sigma_{24}^1$,
because every line $C_i$ can intersect $\mathcal{Q}$ by at most two points.
Write $C_1\cap\Sigma_{24}^1=\{P,Q\}$, and denote the stabilizer of the point $P$ in $\overline{G}_{144}$ by $\Gamma_P$.
Then $|\Gamma_P|=6$, so that $\Gamma_P\cong\mumu_6$ by Lemma~\ref{lemma:144-subgroups}(v).
On the other hand, the stabilizer of this point in $\Gamma$ has index at most $2$.
This means that this stabilizer is the group $\Gamma_P$, so that $\Gamma_P\subset\Gamma$.
Now using Lemma~\ref{lemma:144-subgroups}(v) again, we see that~\mbox{$\Gamma\cong\mumu_2^2\times\mumu_3$} and $\Gamma$ is of product type.
Pick an element $\gamma\in\Gamma$ such that $\gamma\not\in\Gamma_P$ and the order of $\gamma$ is two.
Recall that
$$
\Sigma_{24}^1=\mathcal{L}_4^1\cap\mathcal{L}_6^2.
$$
Then $\gamma$ preserves the irreducible component of $\mathcal{L}_4^1$ that contains the point $P$.
Denote it by $L$. This means that $Q=\gamma(P)\in L$, so that the lines $L$ and $C_1$ coincide,
which is impossible, since $C_1$ is not contained in $\mathcal{Q}$ by assumption.
\end{proof}

Let us describe  the intersections $S_i\cap S_j$, $1\leqslant i<j\leqslant 4$.

\begin{lemma}
\label{lemma:144-S1-S2-S4-S3}
Both $S_1\cap S_2$ and $S_3\cap S_4$ are $\overline{G}_{144}$-irreducible curves that are unions of $16$ lines.
\end{lemma}

\begin{proof}
It
is enough to show that $S_1\cap S_2$ is a $\overline{G}_{144}$-irreducible curve that is a union of $16$ lines,
since the proof is identical in the remaining case.
Write
$$
S_2\big\vert_{S_1}=\sum_{i=1}^{r}n_iZ_i,
$$
where each $Z_i$ is a $\overline{G}_{144}$-irreducible curve, each $n_i$ is a positive integer, and $r\geqslant 1$.
We may assume that $\mathrm{deg}(Z_i)\leqslant\mathrm{deg}(Z_j)$ for $i\leqslant j$.
Then
$$
16=\sum_{i=1}^{r}n_i\mathrm{deg}(Z_i).
$$
Note that the intersection $S_1\cap S_2$ does not contain the curve $\mathcal{L}_6^1$ and the curve $\mathcal{L}_6^2$,
simply because $S_1$ and $S_2$ do not contain them.
Indeed, the restrictions $S_1\vert_{\mathcal{Q}}$ and $S_2\vert_{\mathcal{Q}}$ are divisors of bi-degree $(4,4)$ on the quadric $\mathcal{Q}\cong\mathbb{P}^1\times\mathbb{P}^1$, while both curves $\mathcal{L}_6^1$ and $\mathcal{L}_6^2$
are divisors on $\mathcal{Q}$ of bi-degree $(6,0)$ and $(0,6)$, respectively.
Moreover, it follows from Remark~\ref{remark:144-quartics} that the intersection $S_1\cap S_2$
cannot contain two curves among $\mathcal{L}_4^1$, $\mathcal{L}_4^2$, $\mathcal{L}_4^3$ and~$\mathcal{L}_4^4$.
Thus, it follows from Lemmas~\ref{lemma:144-Q-curves} and \ref{lemma:144-curves-very-small-degree}
that one of the following cases holds:
\begin{itemize}
\item either $r=1$ and $n_1=1$,

\item or $r=2$, $n_1=n_2=1$, $\mathrm{deg}(Z_2)=12$, and $Z_1$ is one of the curves $\mathcal{L}_4^1$, $\mathcal{L}_4^2$, $\mathcal{L}_4^3$, $\mathcal{L}_4^4$.
\end{itemize}

Let $f\colon\widetilde{S}_1\to S_1$ be the minimal resolution of singularities of the surface $S_1$.
Then the action of the group $\overline{G}_{144}$ lifts to the surface $\widetilde{S}_1$.
Denote by $E_1,\ldots,E_{12}$ the exceptional curves of the birational morphism $f$.
Let $E=E_1+\ldots+E_{12}$, and let $H$ be a plane section of the surface $S_1$.
Denote by $\widetilde{Z}_i$ the proper transform of the curve $Z_i$ on the surface $\widetilde{S}_1$.
Then
$$
\sum_{i=1}^{r}Z_i\sim f^*(4H)-mE
$$
for some none-negative integer $m$.
Moreover, we have $m\geqslant 2$, since $S_2$ is singular at every point of $\Sigma_{12}$.
On the other hand, since  $\widetilde{S}_1$ is a smooth $K3$ surface, we have
$$
\widetilde{Z}_i^2\geqslant -2\cdot\mathrm{deg}(Z_i),
$$
because the curve $\widetilde{Z}_i$ cannot consist of more than $\mathrm{deg}(Z_i)$ irreducible components.

If $r=2$, then $\widetilde{Z}_1^2=-8$ and $\widetilde{Z}_1\cdot E=0$,
since $Z_1$ is a disjoint union of $4$ lines that does not contain $\Sigma_{12}$.
In this case, we have
$$
-24=-2\cdot\mathrm{deg}\big(Z_2\big)\leqslant\widetilde{Z}_2^2=\Big(f^*(4H)-mE-\widetilde{Z}_1\Big)^2=64-24m^2-40\leqslant -72,
$$
which is absurd.

Thus, we see that $r=1$, so that $S_1\cap S_2=Z_1$ is a $\overline{G}_{144}$-irreducible curve of degree $16$.
This gives
$$
-32=-2\cdot\mathrm{deg}\big(Z_1\big)\leqslant\widetilde{Z}_1^2=\Big(f^*(4H)-mE\Big)^2=64-24m^2\leqslant -32,
$$
so that $\widetilde{Z}_1^2=-2\cdot\mathrm{deg}(Z_1)$.
This easily implies that $Z_1$ is a union of~\mbox{$\mathrm{deg}(Z_1)=16$} lines,
cf.~the proof of \mbox{\cite[Lemma~10.1.1]{CheltsovShramov}}.
\end{proof}

As we already mentioned in Remark~\ref{remark:144-quartics}, each of the intersections
$S_1\cap\mathcal{Q}$, $S_2\cap\mathcal{Q}$, $S_3\cap\mathcal{Q}$ and $S_4\cap\mathcal{Q}$
is one of the unions $\mathcal{L}_4^1\cup\mathcal{L}_4^3$, $\mathcal{L}_4^1\cup\mathcal{L}_4^4$, $\mathcal{L}_4^2\cup\mathcal{L}_4^3$ or $\mathcal{L}_4^2\cup\mathcal{L}_4^4$.
On the other hand, we just proved that the intersections $S_1\cap S_2$ and $S_3\cap S_4$ are $\overline{G}_{144}$-irreducible curves.
Thus, without loss of generality, we may assume that
\begin{equation}
\label{equation:S1-S2-S3-S4-L4-1-2-3-4}
S_1\cap\mathcal{Q}=\mathcal{L}_4^2\cup\mathcal{L}_4^4, \quad S_2\cap\mathcal{Q}=\mathcal{L}_4^1\cup\mathcal{L}_4^3, \quad
S_3\cap\mathcal{Q}=\mathcal{L}_4^1\cup\mathcal{L}_4^4, \quad S_4\cap\mathcal{Q}=\mathcal{L}_4^2\cup\mathcal{L}_4^3.
\end{equation}
Using this, we can determine which $\overline{G}_{144}$-orbit $\Sigma_{16}^1$, $\Sigma_{16}^2$, $\Sigma_{16}^3$, $\Sigma_{16}^4$
is contained in which surface  $S_1$, $S_2$, $S_3$ and $S_4$.
This is summarized in the following table.
\begin{center}\renewcommand\arraystretch{1.8}
\begin{tabular}{|c||c|c|c|c|}
  \hline
   $\qquad$& $S_1$ & $S_2$ & $S_3$ & $S_4$ \\
  \hline
  \hline
  $\Sigma_{16}^1$ & $-$ & $+$ & $+$ & $+$\\
  \hline
  $\Sigma_{16}^2$ & $+$ & $+$ & $+$ & $-$\\
  \hline
  $\Sigma_{16}^3$ & $+$ & $+$ & $-$ & $+$\\
  \hline
  $\Sigma_{16}^4$ & $+$ & $-$ & $+$ & $+$\\
  \hline
\end{tabular}
\end{center}

Observe that the intersection $S_1\cap S_3$ contains the curve $\mathcal{L}_4^4$,
the intersection~\mbox{$S_1\cap S_4$} contains the curve $\mathcal{L}_4^2$,
the intersection $S_2\cap S_3$ contains the curve $\mathcal{L}_4^1$,
and the intersection~\mbox{$S_2\cap S_4$} contains the curve $\mathcal{L}_4^3$.
This gives

\begin{lemma}
\label{lemma:144-C12}
One has
$$
S_1\cap S_3=\mathcal{L}_4^4\cup C_{12}^4, \quad S_1\cap S_4=\mathcal{L}_4^2\cup C_{12}^2,
\quad S_2\cap S_3=\mathcal{L}_4^1\cup C_{12}^1, \quad S_2\cap S_4=\mathcal{L}_4^3\cup C_{12}^3,
$$
where $C_{12}^1$, $C_{12}^2$, $C_{12}^3$ and $C_{12}^4$ are distinct $\overline{G}_{144}$-invariant irreducible smooth curves in $\mathbb{P}^3$ of degree $12$ and genus $13$.
\end{lemma}

\begin{proof}
It is enough to prove that $S_1\cap S_3=\mathcal{L}_4^4\cup C_{12}^4$,
where $C_{12}^4$ is a $\overline{G}_{144}$-invariant irreducible smooth curve of degree $12$ and genus $13$.
We see from \eqref{equation:S1-S2-S3-S4-L4-1-2-3-4} that $S_1$ and $S_3$ contain the curve $\mathcal{L}_4^4$.
Thus, we write
$$
S_3\big\vert_{S_1}=m\mathcal{L}_4^4+\sum_{i=1}^{r}n_iZ_i,
$$
where $m$ is a positive integer, each $Z_i$ is a $\overline{G}_{144}$-irreducible curve that
is different from $\mathcal{L}_4^4$, each $n_i$ is a positive integer, and $r\geqslant 0$.
Then
$$
16=4m+\sum_{i=1}^{r}n_i\mathrm{deg}(Z_i).
$$
Note that $r$ is actually positive, since $S_3\vert_{S_1}$ is an ample divisor.
If $m\geqslant 2$ or $r\geqslant 2$, then $\mathrm{deg}(Z_i)<12$ for each curve $Z_i$,
so that $Z_i$ is contained in $\mathcal{Q}$ by Lemma~\ref{lemma:144-curves-very-small-degree},
which contradicts \eqref{equation:S1-S2-S3-S4-L4-1-2-3-4}.
Thus, we see that $m=1$ and $r=1$.
Similarly, we obtain $n_1=1$.

For simplicity, denote $Z=Z_1$. Let $H$ be a plane section of the surface $S_1$.
Then
$$
S_3\big\vert_{S_1}=\mathcal{L}_4^4+Z\sim 4H.
$$
Note that the curve $Z$ is contained in the smooth locus of $S_1$, because $S_1\cap S_3$ does not contain $\Sigma_{12}$ by \eqref{equation:S1-S2-S3-S4-12-12}.
Thus, we have
$$
-8+\mathcal{L}_4^4\cdot Z=\mathcal{L}_4^4\cdot\Big(\mathcal{L}_4^4+Z\Big)=4H\cdot\mathcal{L}_4^4=16,
$$
so that $\mathcal{L}_4^4\cdot Z=24$  on the surface $S_1$. This gives
$$
Z^2+24=Z^2+\mathcal{L}_4^4\cdot Z=Z\cdot\Big(\mathcal{L}_4^4+Z\Big)=4H\cdot Z=48,
$$
so that $Z^2=24$.

If $Z$ is reducible, then $Z$ is contained in the quadric $\mathcal{Q}$ by Lemma~\ref{lemma:144-reducible-curves},
which is impossible, because $S_1\cap\mathcal{Q}=\mathcal{L}_4^2\cup\mathcal{L}_4^4$ by \eqref{equation:S1-S2-S3-S4-L4-1-2-3-4}.
Thus, we see that $Z$ is irreducible.
Then it has arithmetic genus $13$.
If $Z$ is singular, then it has at least $12$ singular points by Lemma~\ref{lemma:144-orbits}.
In this case, the genus of its normalization is at most $1$, which contradicts Lemma~\ref{lemma:144-sporadic-genera}.
Therefore, we see that the curve $Z$ is smooth.
\end{proof}

\begin{corollary}
\label{corollary:144-C12-sextics}
Let $C$ be one of the curves $C_{12}^1$, $C_{12}^2$, $C_{12}^3$, $C_{12}^4$,
and let $\mathcal{M}$ be the linear system of surfaces in $\mathbb{P}^3$ of degree $6$ that contains $C$.
Then the base locus of $\mathcal{M}$ consists of the curve $C$.
\end{corollary}

\begin{proof}
Without loss of generality, we may assume that $C=C_{12}^4$.
Then $C$ is contained in~$S_1$ and
$$
S_3\big\vert_{S_1}=\mathcal{L}_4^4+C.
$$
by Lemma~\ref{lemma:144-C12}.
Note that both curves $\mathcal{L}_4^4$ and $C$ are contained in the smooth locus of the surface $S_1$ by \eqref{equation:S1-S2-S3-S4-12-12} and Lemma~\ref{lemma:144-quartics-singular}.
Let $H$ be a plane section of $S_1$, and let $D=2H+\mathcal{L}_4^4$.
Then $D^2=24$ and $D$ is nef, since $D\cdot\mathcal{L}_4^4=0$.
Using the Riemann--Roch formula and Nadel vanishing theorem (see \cite[Theorem~9.4.8]{La04}), we see that $h^0(\mathcal{O}_{S_1}(D))=14$.
This implies that the linear system $|D|$ does not have fixed curves.
Indeed, if it does, then its fixed part must be the curve $\mathcal{L}_4^4$,
so that
$$
14=h^0\Big(\mathcal{O}_{S_1}\big(D\big)\Big)=h^0\Big(\mathcal{O}_{S_1}\big(2H\big)\Big)=10,
$$
which is absurd.
Thus $|D|$ does not have base points by \cite[Corollary~3.2]{SAINTDONAT}.
Therefore, the linear system $|6H-C|$ is base point free.
Since $S_1$ is projectively normal, we immediately obtain the required assertions.
\end{proof}

\begin{corollary}
\label{corollary:144-C12-mult}
Let $\mathcal{D}$ be a (non-empty) mobile $\overline{G}_{144}$-invariant linear system on $\mathbb{P}^3$,
let~$n$ be a positive integer such that $\mathcal{D}\sim\mathcal{O}_{\mathbb{P}^3}(n)$,
and let $C$ be one of the curves $C_{12}^1$, $C_{12}^2$, $C_{12}^3$ or~$C_{12}^4$.
Then $\mathrm{mult}_{C}(\mathcal{D})\leqslant\frac{n}{4}$.
\end{corollary}

\begin{proof}
Without loss of generality, we may assume that $C=C_{12}^4$.
Then $C$ is contained in the surface~$S_1$ by Lemma~\ref{lemma:144-C12}, and $S_3\vert_{S_1}=\mathcal{L}_4^4+C$.
Let $D$ be a general surface in $\mathcal{D}$. Then
$$
D\vert_{S_1}=aC+b\mathcal{L}_4^4+\Delta,
$$
where $a$ and $b$ are non-negative integers such that $a\geqslant\mathrm{mult}_{C}(\mathcal{D})$,
and $\Delta$ is an effective divisor on $S_1$ whose support contains neither $C$ nor any irreducible component of the curve $\mathcal{L}_4^4$.
Then
$$
aC+b\mathcal{L}_4^4+\Delta\sim_{\mathbb{Q}}\frac{n}{4}\Big(\mathcal{L}_4^4+C\Big).
$$
If $a>\frac{n}{4}$, then $b<\frac{n}{4}$, since we have
$$
\Big(\frac{n}{4}-b\Big)\mathcal{L}_4^4\sim_{\mathbb{Q}}\Big(a-\frac{n}{4}\Big)C+\Delta.
$$
Therefore, if $a>\frac{n}{4}$, then
$$
0>-8\Big(\frac{n}{4}-b\Big)=\Big(\frac{n}{4}-b\Big)\mathcal{L}_4^4\cdot\mathcal{L}_4^4=\Big(a-\frac{n}{4}\Big)\mathcal{L}_4^4\cdot C+\mathcal{L}_4^4\cdot\Delta\geqslant 0,
$$
which is absurd. This shows that $\mathrm{mult}_{C}(\mathcal{D})\leqslant a\leqslant\frac{n}{4}$ as required.
\end{proof}

Now we are ready to prove

\begin{theorem}
\label{theorem:144-curves}
Let $C$ be $\overline{G}_{144}$-irreducible curve in $\mathbb{P}^3$ such that $\mathrm{deg}(C)\leqslant 15$.
Then one of the following cases holds:
\begin{itemize}
\item $\mathrm{deg}(C)=4$, and $C$ is one of the curves $\mathcal{L}_4^1$, $\mathcal{L}_4^2$, $\mathcal{L}_4^3$, $\mathcal{L}_4^4$;

\item $\mathrm{deg}(C)=6$, and $C$ is one of the curves $\mathcal{L}_6^1$ or $\mathcal{L}_6^2$;

\item $\mathrm{deg}(C)=12$, and $C$ is one of the smooth irreducible curves $C_{12}^1$, $C_{12}^2$, $C_{12}^3$, $C_{12}^4$;

\item $\mathrm{deg}(C)=12$, and $C$ is a union of $12$ disjoint lines contained in the quadric $\mathcal{Q}$.
\end{itemize}
\end{theorem}

\begin{proof}
By Lemma~\ref{lemma:144-Q-curves}, we may assume that $C$ is not contained in the quadric $\mathcal{Q}$.
Then~\mbox{$\mathrm{deg}(C)\geqslant 12$} by Lemma~\ref{lemma:144-curves-very-small-degree},
and $C$ is irreducible by Lemma~\ref{lemma:144-reducible-curves}.
Moreover, since
$$
30\geqslant 2\cdot\mathrm{deg}(C)=C\cdot\mathcal{Q},
$$
we see that $\mathrm{deg}(C)=12$ by Lemma~\ref{lemma:144-Q-orbits}.
Let us show that $C$ is one of the smooth irreducible curves $C_{12}^1$, $C_{12}^2$, $C_{12}^3$ or~$C_{12}^4$.

Since $C\cdot\mathcal{Q}=24$, it follows from Lemma~\ref{lemma:144-Q-orbits} that
$C\cap\mathcal{Q}$ is a $\overline{G}_{144}$-orbit of length $24$,
and $C$ intersects $\mathcal{Q}$ transversally at the points of $C\cap\mathcal{Q}$.
In particular, the curve $C$ is smooth at these points.
Recall from Lemma~\ref{lemma:144-orbits} that
all $\overline{G}_{144}$-orbits of length $24$ in $\mathbb{P}^3$ and
all $\overline{G}_{144}$-orbits of length $16$ in $\mathbb{P}^3$ are contained in $\mathcal{Q}$.
Thus, the curve $C$ does not contain  $\overline{G}_{144}$-orbits of length $24$ that are different from $C\cap\mathcal{Q}$,
and it does not contain  $\overline{G}_{144}$-orbits of length $16$.
Without loss of generality, we may assume that
$$
C\cap\mathcal{Q}=\Sigma_{24}^1=\mathcal{L}_4^1\cap\mathcal{L}_6^2
$$
Recall that $\Sigma_{24}^1$ is contained in $S_2\cap S_3$, because $S_2\cap\mathcal{Q}=\mathcal{L}_4^1\cup\mathcal{L}_4^3$
and $S_3\cap\mathcal{Q}=\mathcal{L}_4^1\cup\mathcal{L}_4^4$, see \eqref{equation:S1-S2-S3-S4-L4-1-2-3-4}.
Since $S_2\cdot C=S_3\cdot C=48$,
it follows from Lemma~\ref{lemma:144-orbits}, \eqref{equation:S1-S2-S3-S4-12-12} and Lemma~\ref{lemma:144-quartics-singular}
that one of the following cases holds:
\begin{itemize}
\item the curve $C$ is contained in one of the surfaces $S_2$ and $S_3$,
\item both $S_2$ and $S_3$ are tangent to $C$ at the points of $\Sigma_{24}^1$, and $S_2\cap C=S_3\cap C=\Sigma_{24}^1$,
\item $S_2\cap C=\Sigma_{24}^1\cup\Sigma_{12}$ or $S_3\cap C=\Sigma_{24}^1\cup\Sigma_{12}^\prime$.
\end{itemize}

Suppose that $S_2\cap C=\Sigma_{24}^1\cup\Sigma_{12}$.
Since $C$ is irreducible, it must be singular at every point of $\Sigma_{12}$ by Lemma~\ref{lemma:144-sporadic-genera}.
Then
$$
48=S_2\cdot C\geqslant |\Sigma_{24}^1|\mathrm{mult}_{\Sigma_{24}^1}(S_2)+|\Sigma_{12}|\mathrm{mult}_{\Sigma_{12}}(S_2)\mathrm{mult}_{\Sigma_{12}}(C)\geqslant 72,
$$
which is absurd. Thus, we have $S_2\cap C\ne\Sigma_{24}^1\cup\Sigma_{12}$.
Similarly, we see that the intersection~\mbox{$S_3\cap C$} is not the union $\Sigma_{24}^1\cup\Sigma_{12}^\prime$.

Suppose that the surfaces $S_2$ and $S_3$ are tangent to the curve $C$ at the points of $\Sigma_{24}^1$, and
$$
S_2\cap C=S_3\cap C=\Sigma_{24}^1.
$$
Let $\mathcal{P}$ be the pencil generated by~$S_2$~and~$S_3$. Then~$\mathcal{P}$ is $\overline{G}_{144}$-invariant.
Moreover, every surface in this pencil either contains $C$ or is tangent to the curve $C$ at the points of~$\Sigma_{24}^1$
(here we include the case when the surface is singular at some points of~$\Sigma_{24}^1$).
This implies that there exists a surface $S$ in $\mathcal{P}$ that contains $C$.
Indeed, let $P$ be a point in $C$ that is not contained in~$\Sigma_{24}^1$,
and let $S$ be a surface in $\mathcal{P}$ that passes through~$P$.
If $C$ is not contained in $S$, then we obtain contradictory inequalities
$$
48=S\cdot C\geqslant\sum_{Q\in\Sigma_{24}^1}\mathrm{mult}_{Q}\Big(S\cdot C\Big)+\mathrm{mult}_{P}\Big(S\cdot C\Big)\geqslant 2|\Sigma_{24}^1|+1=49,
$$
because $S$ is tangent to $C$ at the points of $\Sigma_{24}^1$. Thus, we see that the curve $C$ is contained in $S$.
Moreover, if $S$ is not $\overline{G}_{144}$-invariant, then there exists $\bar{g}\in\overline{G}_{144}$
such that $\bar{g}(S)\ne S$, so that $\bar{g}(S)\in\mathcal{P}$,
and the surfaces $S$ and $\bar{g}(S)$ generate the whole pencil $\mathcal{P}$.
This implies that $C$ is contained in $S_2\cap S_3$, which is not the case by assumption.

We are left with the case when $C$ is contained in one of the surfaces $S_2$ and~$S_3$.
Without loss of generality, we may assume that $C\subset S_2$.
Let us show that the curve $C$ is also contained in $S_3$.
This would imply that $C=C_{12}^1$ by Lemma~\ref{lemma:144-C12}.

Let $f\colon\widetilde{S}_2\to S_2$ be the minimal resolution of singularities of the surface $S_2$.
Then the action of the group $\overline{G}_{144}$ lifts to the surface $\widetilde{S}_2$.
Denote by $E_1,\ldots,E_{12}$ the $f$-exceptional curves.
Let $E=E_1+\ldots+E_{12}$, and let $H$ be a plane section of the surface $S_2$.
Denote by~$\widetilde{C}$ the proper transform of the curve $C$ on the surface $\widetilde{S}_2$.
Recall from Lemma~\ref{lemma:144-C12} that~\mbox{$S_2\cap S_3=\mathcal{L}_4^1\cup C_{12}^1$}
and~\mbox{$S_2\cap S_4=\mathcal{L}_4^3\cup C_{12}^3$}.
Denote by $\widetilde{C}_{12}^1$ and~$\widetilde{C}_{12}^3$ the proper transforms on the surface~$\widetilde{S}_{2}$ of the curves
$C_{12}^1$ and $C_{12}^3$, respectively.
Similarly, denote by $\widetilde{\mathcal{L}}_4^1$ and $\widetilde{\mathcal{L}}_4^3$
the proper transforms on the surface $\widetilde{S}_{2}$ of the curves
$\mathcal{L}_4^1$ and $\mathcal{L}_4^3$, respectively.
Since
$$
\mathrm{Sing}(S_2)=\Sigma_{12}\not\subset S_3\cup S_4
$$
by \eqref{equation:S1-S2-S3-S4-12-12} and Lemma~\ref{lemma:144-quartics-singular}, we have
$$
\widetilde{\mathcal{L}}_4^1+\widetilde{C}_{12}^1\sim \widetilde{\mathcal{L}}_4^3+\widetilde{C}_{12}^3\sim f^*(4H).
$$
Moreover, we have $\widetilde{\mathcal{L}}_4^1+\widetilde{\mathcal{L}}_4^3\sim f^*(2H)$,
so that $\widetilde{C}\cdot(\widetilde{\mathcal{L}}_4^1+\widetilde{\mathcal{L}}_4^3)=24$.
Using Lemma~\ref{lemma:144-Q-orbits} and keeping in mind that the curve $C$ contains
the $\overline{G}_{144}$-orbit $\Sigma_{24}^1=\mathcal{L}_4^1\cap\mathcal{L}_6^2$,
we conclude that $\widetilde{C}\cdot\widetilde{\mathcal{L}}_4^1=24$ and $\widetilde{C}\cdot\widetilde{\mathcal{L}}_4^3=0$.
Thus
$$
48=\widetilde{C}\cdot\Big(\widetilde{\mathcal{L}}_4^1+\widetilde{C}_{12}^1\Big)=24+\widetilde{C}\cdot\widetilde{C}_{12}^1,
$$
so that $\widetilde{C}\cdot\widetilde{C}_{12}^1=24$. Therefore, it follows from Hodge index theorem that
$$
\left|\begin{matrix} %
\widetilde{C}^2& \widetilde{C}\cdot\widetilde{C}_{12}^1\cr%
\widetilde{C}\cdot\widetilde{C}_{12}^1& \widetilde{C}_{12}^1\cdot\widetilde{C}_{12}^1\cr%
\end{matrix}\right|=
\left|\begin{matrix} %
\widetilde{C}^2 & 24\cr%
24& 24\cr%
\end{matrix}\right|\leqslant 0.
$$
This implies that either $\widetilde{C}^2\leqslant 22$ or $\widetilde{C}^2=24$.
Therefore, the arithmetic genus of the curve $\widetilde{C}$ is at most
$$
\frac{\widetilde{C}^2}{2}+1\leqslant 13.
$$
If $\widetilde{C}$ is singular, then it has at least $12$ singular points by Lemma~\ref{lemma:144-orbits}.
In this case, the genus of its normalization is at most $1$, which contradicts Lemma~\ref{lemma:144-sporadic-genera}.
Thus, we see that the curve $\widetilde{C}$ is a smooth curve of genus $\frac{\widetilde{C}^2}{2}+1$. Therefore, it follows from Lemma~\ref{lemma:144-sporadic-genera} that~\mbox{$\widetilde{C}^2=24$},
because $\widetilde{C}$ contains $\overline{G}_{144}$-orbits of length $24$,
for instance, the $\overline{G}_{144}$-orbit  $\widetilde{C}\cap\widetilde{\mathcal{L}}_4^1$.
Thus, it follows from Hodge index theorem that the curves $\widetilde{C}$ and $\widetilde{C}_{12}^1$ are numerically equivalent on~$\widetilde{S}_2$,
so that~\mbox{$\widetilde{C}\sim\widetilde{C}_{12}^1$}.
Hence $\widetilde{C}\cdot E=0$, so that $C$ does not contain the singular locus of the surface $S_2$,
which implies that
$$
C+\mathcal{L}_4^1\sim 4H.
$$
Since $S_2$ is projectively normal, we see that $C=C_{12}^1$ by Lemma~\ref{lemma:144-C12}.
\end{proof}

\section{The group of order $80$}
\label{section:80-160-320}

Let us use notation and assumptions of~\S\ref{section:Heisenberg}.
In this section, we present basic facts about $\overline{G}_{80}$-orbits,
$\overline{G}_{80}$-invariant curves and $\overline{G}_{80}$-invariant surfaces in $\mathbb{P}^3$.
Let us start with the following very easy group-theoretic result.

\begin{lemma}
\label{lemma:80-subgroups}
Let $\Gamma$ be a proper non-trivial subgroup of the group $\overline{G}_{80}$.
Then $\Gamma$ is isomorphic to one of the following groups: $\mumu_2$, $\mumu_2^2$, $\mumu_2^3$, $\mumu_2^4$, or $\mumu_5$.
Moreover, a subgroup of $\overline{G}_{80}$ isomorphic to $\mumu_2^4$ is unique,
and a subgroup isomorphic to $\mumu_5$ is unique up to conjugation.
Furthermore, the only non-trivial proper normal subgroup of $\overline{G}_{80}$ is $\overline{\mathrm{H}}\cong\mumu_2^4$.
\end{lemma}

\begin{proof}
Sylow theorems imply that the order of every non-trivial element of $\overline{G}_{80}$
equals either $2$ or $5$. If every non-trivial element of $\Gamma$ has order~$2$,
then $\Gamma$ is contained in the Sylow $2$-subgroup of $\overline{G}_{80}$, which is exactly $\overline{\mathrm{H}}$.
If every non-trivial element of $\Gamma$ has order~$5$, then $\Gamma\cong\mumu_5$.
The uniqueness assertions for subgroups isomorphic to $\mumu_2^4$ and $\mumu_5$ also follow from Sylow theorems.

Suppose that $\Gamma$ contains both elements of order $2$ and elements of order $5$.
Then it is isomorphic to $\Gamma^\prime\rtimes\mumu_5$, where $\Gamma^\prime$ is a non-trivial proper subgroup of the group $\overline{\mathrm{H}}$
that is invariant under the action of
$$
\overline{G}_{80}/\overline{\mathrm{H}}\cong\mumu_5\subset\mathrm{GL}_4\big(\mathbf{F}_2\big)
$$
on $\overline{\mathrm{H}}$.
However, the group $\mumu_5$ is not a subgroup of $\mathrm{GL}_k(\mathbf{F}_2)$ for $k\leqslant 3$. This implies that
the action of $\mumu_5$ on $\mathbf{F}_2^4$ is irreducible, so that $\overline{\mathrm{H}}$ has no non-trivial
proper $\mumu_5$-invariant subgroups. Therefore, we see that $\Gamma=\overline{G}_{80}$.

Now we suppose that $\Gamma$ is a non-trivial normal subgroup of $\overline{G}_{80}$.
Then $\Gamma$ is not isomorphic to $\mumu_5$.
Also we know from the above arguments that it cannot be a proper subgroup of $\mumu_2^4$.
Thus, if $\Gamma\not\cong\mumu_2^4$, then $\Gamma$ must contain both elements of order $2$ and elements of order $5$.
This implies that $\Gamma=\overline{G}_{80}$.
\end{proof}

\begin{corollary}
\label{corollary:G80}
Let $\overline{G}$ be a finite subgroup in $\mathrm{PGL}_{4}(\mathbb{C})$ that is isomorphic to~$\overline{G}_{80}$.
Then $\overline{G}$ is conjugate to $\overline{G}_{80}$.
\end{corollary}

\begin{proof}
It follows from \cite[Chapter~VII]{Blichfeldt1917} that $\overline{G}$ is conjugate to $\overline{G}_{80}$ provided that $\overline{G}$ is primitive.
Suppose $\overline{G}$ is not primitive. Let us seek for a contradiction.
To start with, suppose that $\overline{G}$ is transitive.
Then either there exists a $\overline{G}$-orbit of length $4$,
or there exists a $\overline{G}$-invariant pair of skew lines.
In the former case, we have a non-trivial homomorphism~\mbox{$\overline{G}\to\mathfrak{S}_4$}.
In the latter case, we have a non-trivial homomorphism~\mbox{$\overline{G}\to\mathfrak{S}_2$}.
None of this is possible by Lemma~\ref{lemma:80-subgroups}.

Thus, we see that $\overline{G}$ is intransitive.
Then either there exists a $\overline{G}$-invariant point $P$, or there exist a $\overline{G}$-invariant line $L$.
In the former case, the group $\overline{G}$ acts faithfully on the tangent space $T_P(\mathbb{P}^3)\cong\mathbb{C}^3$ by \cite[Lemma~4.4.1]{CheltsovShramov},
which is impossible, because $\mathrm{GL}_3(\mathbb{C})$ does not contain subgroups isomorphic to $\mumu_2^4$.
In the latter case, we obtain a homomorphism
$$
\overline{G}\to\mathrm{Aut}\big(L\big)\cong\mathrm{PGL}_2\big(\mathbb{C}\big).
$$
We conclude from Lemma~\ref{lemma:80-subgroups} and the classification of finite subgroups in $\mathrm{PGL}_2(\mathbb{C})$
that the kernel of this homomorphism contains the subgroup in $\overline{G}$ isomorphic to $\mumu_2^4$,
so that this subgroup must fix every point in $L$, which is again impossible by \cite[Lemma~4.4.1]{CheltsovShramov}.
\end{proof}

To study $\overline{G}_{80}$-invariant curves in $\mathbb{P}^3$, we need the following result,
which follows from the Riemann--Hurwitz formula.

\begin{lemma}
\label{lemma:80-sporadic-genera}
Let $C$ be a smooth irreducible curve of genus $g\leqslant 19$
with a faithful action of the group $\overline{G}_{80}$.
Then $g\in\{5,13,17\}$, and $C$ is not hyperelliptic.
Furthermore, if~\mbox{$\Omega\subset C$} is a $\overline{G}_{80}$-orbit, then $|\Omega|\in\{16,40,80\}$.
Finally, the possible numbers $a_i$ of $\overline{G}_{80}$-orbits of length $i\in\{16,40\}$ in $C$
are contained in the following table:
\begin{center}\renewcommand\arraystretch{1.1}
\begin{tabular}{|c||c|c|c|}
\hline
$g$ &  $5$ & $13$ & $17$ \\
\hline\hline
$a_{16}$ & $2$ & $1$ & $3$ \\
\hline
$a_{40}$ & $1$ & $3$ & $0$\\
\hline
\end{tabular}
\end{center}
\end{lemma}

\begin{proof}
The assertion about the lengths of $\overline{G}_{80}$-orbits follows from Lemma~\ref{lemma:80-subgroups},
since the stabilizers in $\overline{G}_{80}$ of points in $C$ must be cyclic (see \cite[Lemma~5.1.4]{CheltsovShramov}).
Moreover, the group~$\overline{G}_{80}$ cannot act faithfully on $\mathbb{P}^1$ and on a smooth elliptic curve by Lemma~\ref{lemma:Heisenberg-elliptic}.
Hence, one has~\mbox{$g\geqslant 2$}.

Suppose that the curve $C$ is hyperelliptic.
Since the group $\overline{G}_{80}$ does not contain normal subgroups of order $2$ by Lemma~\ref{lemma:80-subgroups},
it does not contain the hyperelliptic involution of $C$, and we obtain
a faithful action of the group $\overline{G}_{80}$ on $\mathbb{P}^1$.
We already proved that this is impossible by  Lemma~\ref{lemma:Heisenberg-elliptic}.

Let $\widehat{C}=C\slash\overline{G}_{80}$. Then $\widehat{C}$ is a smooth curve.
Let~$\hat{g}$ be the genus of the curve~$\widehat{C}$. Then the Riemann--Hurwitz formula gives
$$
2g-2=80\big(2\hat{g}-2\big)+40a_{40}+64a_{16}.
$$
Since $a_k\geqslant 0$ and $g\leqslant 19$, one has $\hat{g}=0$, so that
$$
2g-2=-160+40a_{40}+64a_{16},
$$
Going through the possible values of $g$, and solving this equation
case by case we obtain the required result.
\end{proof}

\begin{lemma}
\label{lemma:80-Sigma-16}
There are exactly four $\overline{G}_{80}$-orbits of length $16$ in $\mathbb{P}^3$.
\end{lemma}

\begin{proof}
A subgroup $\mumu_5\subset\overline{G}_{80}$ has a fixed point in $\mathbb{P}^3$,
so that the length of its $\overline{G}_{80}$-orbit is~$16$ by Lemma~\ref{lemma:80-subgroups}.
This shows that there is at least one  $\overline{G}_{80}$-orbit of length $16$ in $\mathbb{P}^3$.

Now let $\Gamma$ be a stabilizer of a point in a $\overline{G}_{80}$-orbit of length $16$.
Then $\Gamma\cong\mumu_5$.
By Lemma~\ref{lemma:80-subgroups}, the subgroup $\Gamma$ is conjugate to the subgroup generated by the image $T$ in~\mbox{$\mathrm{PGL}_4(\mathbb{C})$}.
On the other hand, the eigenvalues of the matrix $T$ are four distinct primitive fifth roots of unity.
Since the normalizer of $\Gamma$ in $\overline{G}_{80}$ coincides with $\Gamma$ by Lemma~\ref{lemma:80-subgroups},
we conclude that there are exactly four $\overline{G}_{80}$-orbits of length $16$ in $\mathbb{P}^3$.
\end{proof}

Let us denote the $\overline{G}_{80}$-orbits of length $16$ in $\mathbb{P}^3$ by $\Sigma_{16}^1$, $\Sigma_{16}^2$, $\Sigma_{16}^3$ and $\Sigma_{16}^4$.
Computing the null-spaces of the matrices $T-\xi_5^iI_4$, $1\leqslant i\leqslant 4$, one can find these orbits explicitly.
In particular, we may assume that
$$
\big[(-1+i)\xi_5^3+(-1+i)\xi_5^2-\xi_5-1:1+(1+i)\xi_5^2+\xi_5:1+(1+i)\xi_5^3+\xi_5:\xi_5-1\big]\in\Sigma_{16}^1,
$$
$$
\big[(1+i)\xi_5^3+(1-i)\xi_5^2+i\xi_5+1:-2\xi_5^3+(-1-i)\xi_5^2-\xi_5-2-i:(-1+i)\xi_5^3+i\xi_5+i:\xi_5-1\big]\in\Sigma_{16}^2,
$$
$$
\big[-2i\xi_5^3+(1-i)\xi_5^2-i\xi_5+1-2i:(1-i)\xi_5^3-(1+i)\xi_5^2+\xi_5-i:(1-i)\xi_5^3-i\xi_5-i: \xi_5-1\big]\in\Sigma_{16}^3,
$$
and
$$
\big[(-1+i)\xi_5^2+i\xi_5+i:(1+i)\xi_5^3+(1+i)\xi_5^2+i\xi_5+i:-1+(-1-i)\xi_5^3-\xi_5:\xi_5-1\big]\in\Sigma_{16}^4.
$$

Recall from~\S\ref{section:Heisenberg}, that there are three $\overline{G}_{80}$-orbits
$\Sigma_{20}$, $\Sigma_{20}^\prime$ and $\Sigma_{20}^{\prime\prime}$ of length $20$ in~$\mathbb{P}^3$
described by \eqref{equation:Sigma-20}, \eqref{equation:Sigma-20-prime} and \eqref{equation:Sigma-20-prime-prime},
respectively.

\begin{lemma}
\label{lemma:80-Sigma-20}
The orbits $\Sigma_{20}$, $\Sigma_{20}^\prime$ and $\Sigma_{20}^{\prime\prime}$ are the only $\overline{G}_{80}$-orbits of length $20$ in $\mathbb{P}^3$.
\end{lemma}

\begin{proof}
Let $\Sigma$ be a $\overline{G}_{80}$-orbit of length $20$,
let $P$ be a point in $\Sigma$, and let $\Gamma$ be the stabilizer in $\overline{G}_{80}$ of the point $P$.
Then $\Gamma\cong\mumu_2^2$ by Lemma~\ref{lemma:80-subgroups}.
Let $\gamma_1$, $\gamma_2$ and $\gamma_3$ be three non-trivial elements in $\Gamma$.
Then the locus of fixed points of each of them consists of two skew lines in $\mathbb{P}^3$ by Lemma~\ref{lemma:Heisenberg-subgroup}(i).
One of the lines from each pair must pass through the point $P$,
so that $P\in\Sigma_{20}\cup\Sigma_{20}^\prime\cup\Sigma_{20}^{\prime\prime}$ by Lemma~\ref{lemma:Heseinebrg-lines-points}.
\end{proof}

To describe $\overline{G}_{80}$-orbits in $\mathbb{P}^3$ of length $40$,
recall from Lemma~\ref{lemma:Heisenberg-subgroup}(i) that there are $30$ lines
in $\mathbb{P}^3$ such that each of them is pointwise fixed by some non-trivial element in $\overline{\mathrm{H}}$.
These lines have been described in \eqref{equation:ell-1-2-3-4-5}, \eqref{equation:ell-1-2-3-4-5-prime} and \eqref{equation:ell-1-2-3-4-5-prime-prime}.
Let
$$
\mathcal{L}_{10}=\sum_{i=1}^{5}\ell_i+\sum_{i=1}^{5}\check{\ell}_i.
$$
Similarly, let
$$
\mathcal{L}_{10}^\prime=\sum_{i=1}^{5}\ell_i^\prime+\sum_{i=1}^{5}\check{\ell}_i^\prime.
$$
Finally, we let
$$
\mathcal{L}_{10}^{\prime\prime}=\sum_{i=1}^{5}\ell_i^{\prime\prime}+\sum_{i=1}^{5}\check{\ell}_i^{\prime\prime}.
$$
Then $\mathcal{L}_{10}$,  $\mathcal{L}_{10}^\prime$, $\mathcal{L}_{10}^{\prime\prime}$ are $\overline{G}_{80}$-irreducible curves by construction.

\begin{lemma}
\label{lemma:80-L10}
The curve $\mathcal{L}_{10}$ is a disjoint union of ten lines,
$\mathcal{L}_{10}^\prime$ and $\mathcal{L}_{10}^{\prime\prime}$ are nodal unions of ten lines,
$\mathrm{Sing}(\mathcal{L}_{10}^\prime)=\Sigma_{20}^\prime\not\subset\mathcal{L}_{10}^{\prime\prime}$ and $\mathrm{Sing}(\mathcal{L}_{10}^{\prime\prime})=\Sigma_{20}^{\prime\prime}\not\subset\mathcal{L}_{10}^\prime$.
Moreover, one has
$$
\mathcal{L}_{10}\cap\mathcal{L}_{10}^\prime\cap\mathcal{L}_{10}^{\prime\prime}=
\mathcal{L}_{10}^\prime\cap\mathcal{L}_{10}^{\prime\prime}=\Sigma_{20},\quad
\mathcal{L}_{10}\cap\mathcal{L}_{10}^\prime=\Sigma_{20}\cup\Sigma_{20}^\prime,\quad
\mathcal{L}_{10}\cap\mathcal{L}_{10}^{\prime\prime}=\Sigma_{20}\cup\Sigma_{20}^{\prime\prime},
$$
and $\Sigma_{20}\cup\Sigma_{20}^\prime\cup\Sigma_{20}^{\prime\prime}\subset\mathcal{L}_{10}$.
Furthermore, the $\overline{G}_{80}$-orbit of every point in
$$
\Big(\mathcal{L}_{10}\cup \mathcal{L}_{10}^\prime\cup\mathcal{L}_{10}^{\prime\prime}\Big)\setminus\Big(\Sigma_{20}\cup\Sigma_{20}^\prime\cup\Sigma_{20}^{\prime\prime}\Big)
$$
consists of $40$ points, and every $\overline{G}_{80}$-orbit in $\mathbb{P}^3$ of length $40$ is contained in $\mathcal{L}_{10}\cup \mathcal{L}_{10}^\prime\cup \mathcal{L}_{10}^{\prime\prime}$.
\end{lemma}

\begin{proof}
It is straightforward to check using explicit equations of the lines of $\mathcal{L}_{10}$ that they are pairwise disjoint.
Similarly, we see that $\mathcal{L}_{10}$ contains $\Sigma_{20}$, $\Sigma_{20}^\prime$ and $\Sigma_{20}^{\prime\prime}$.
Likewise, the curve $\mathcal{L}_{10}^\prime$ contains $\Sigma_{20}$ and $\Sigma_{20}^\prime$, and it does not contain $\Sigma_{20}^{\prime\prime}$.
Moreover, for every point in~$\Sigma_{20}^\prime$,
there are exactly two among the ten lines of $\mathcal{L}_{10}^{\prime}$ that pass through this point,
and all these ten lines are pairwise disjoint away from $\Sigma_{20}^\prime$.
This means that the curve $\mathcal{L}_{10}^\prime$ is nodal and $\mathrm{Sing}(\mathcal{L}_{10}^\prime)=\Sigma_{20}^\prime$.
In fact, one can also check that
$$
\ell_i^\prime\cap\ell_j^\prime\ne\varnothing\iff\check{\ell}_i^\prime\cap\check{\ell}_j^\prime\ne\varnothing\iff\ell_i^\prime\cap\check{\ell}_j^\prime\ne\varnothing\iff j-i=\pm 1\ \mathrm{mod}\ 5.
$$
Similarly, we see that $\mathcal{L}_{10}^{\prime\prime}$ contains $\Sigma_{20}\cup\Sigma_{20}^{\prime\prime}$, and it does not contain $\Sigma_{20}^{\prime}$.
Furthermore, for every point in $\Sigma_{20}^{\prime\prime}$,
there are exactly two lines among the ten lines of $\mathcal{L}_{10}^{\prime\prime}$
that pass through this point, the lines of $\mathcal{L}_{10}^{\prime\prime}$
do not intersect each other in other points, and
$$
\ell_i^{\prime\prime}\cap\ell_j^{\prime\prime}\ne\varnothing\iff\check{\ell}_i^{\prime\prime}\cap\check{\ell}_j^{\prime\prime}\ne\varnothing\iff\ell_i^{\prime\prime}\cap\check{\ell}_j^{\prime\prime}\ne\varnothing\iff j-i=\pm 2\ \mathrm{mod}\ 5.
$$
In particular, the curve $\mathcal{L}_{10}^{\prime\prime}$ is nodal, and $\mathrm{Sing}(\mathcal{L}_{10}^{\prime\prime})=\Sigma_{20}^{\prime\prime}$.

Explicit computations show that
$$
\mathcal{L}_{10}^\prime\cap\mathcal{L}_{10}^{\prime\prime}=\Sigma_{20},
\quad
\mathcal{L}_{10}\cap\mathcal{L}_{10}^\prime=\Sigma_{20}\cup\Sigma_{20}^\prime,
\quad
\mathcal{L}_{10}\cap\mathcal{L}_{10}^{\prime\prime}=\Sigma_{20}\cup\Sigma_{20}^{\prime\prime}.
$$

By Lemma~\ref{lemma:Heisenberg-subgroup}(i), the curves $\mathcal{L}_{10}$,  $\mathcal{L}_{10}^\prime$ and $\mathcal{L}_{10}^{\prime\prime}$
contain all $\overline{G}_{80}$-orbits in $\mathbb{P}^3$ of length~$40$.
Vice versa, let $P$ be a point in
$\mathcal{L}_{10}\cup \mathcal{L}_{10}^\prime\cup\mathcal{L}_{10}^{\prime\prime}$ such that
$P\not\in\Sigma_{20}\cup\Sigma_{20}^\prime\cup\Sigma_{20}^{\prime\prime}$, and let $\Sigma$ be its $\overline{G}_{80}$-orbit.
Then $|\Sigma|\ne 80$ by construction.
Thus, we have $|\Sigma|=40$ by Lemmas~\ref{lemma:80-subgroups} and \ref{lemma:80-Sigma-20}.
\end{proof}

\begin{corollary}
\label{corollary:80-small-orbits}
Let $\Sigma$ be a $\overline{G}_{80}$-orbit in $\mathbb{P}^3$ such that $|\Sigma|<|\overline{G}_{80}|=80$.
Then one of the following possibilities holds:
\begin{itemize}
\item either $|\Sigma|=40$ and $\Sigma\subset \mathcal{L}_{10}\cup \mathcal{L}_{10}^\prime\cup \mathcal{L}_{10}^{\prime\prime}$,
\item or $|\Sigma|=20$ and $\Sigma$ is one of the $\overline{G}_{80}$-orbits $\Sigma_{20}$, $\Sigma_{20}^\prime$, $\Sigma_{20}^{\prime\prime}$,
\item or $|\Sigma|=16$ and $\Sigma$ is one of the $\overline{G}_{80}$-orbits $\Sigma_{16}^1$, $\Sigma_{16}^2$, $\Sigma_{16}^3$, $\Sigma_{16}^4$.
\end{itemize}
\end{corollary}

\begin{proof}
Let $\Gamma$ be a stabilizer of a point in $\Sigma$.
If $\Gamma\cong\mumu_2$, then $|\Sigma|=40$, and the assertion follows from Lemma~\ref{lemma:80-L10}.
If $\Gamma\cong\mumu_2^2$, then $|\Sigma|=20$, and the assertion follows from Lemma~\ref{lemma:80-Sigma-20}.
If $\Gamma\cong\mumu_5$, then $|\Sigma|=16$, and the assertion follows from Lemma~\ref{lemma:80-Sigma-16}.
Otherwise, the group $\Gamma$ contains a subgroup $\mumu_2^3\subset\overline{\mathrm{H}}$ by Lemma~\ref{lemma:80-subgroups},
which is impossible by Lemma~\ref{lemma:Heisenberg-subgroup}(iii).
\end{proof}

Now we are ready to prove

\begin{lemma}
\label{lemma:80-curves-reducible}
Let $C$ be a reducible $\overline{G}_{80}$-invariant  curve in $\mathbb{P}^3$ such that $\mathrm{deg}(C)\leqslant 15$.
Then~$C$ is one of the curves $\mathcal{L}_{10}$, $\mathcal{L}_{10}^\prime$ or $\mathcal{L}_{10}^{\prime\prime}$.
\end{lemma}

\begin{proof}
We may assume that the curve $C$ is $\overline{G}_{80}$-irreducible.
Write $C=C_1+\ldots+C_r$, where each $C_i$ is an irreducible curve. Then $r\geqslant 2$ by assumption.
Let $d$ be the degree of the curve $C_1$. Then
$$
15\geqslant \mathrm{deg}(C)=dr.
$$
Let $\Gamma$ be the stabilizer in $\overline{G}_{80}$ of the curve $C_1$.
Then there exists an exact sequence of groups
$$
1\to\Gamma_{C_1}\to\Gamma\to\mathrm{Aut}(C_1),
$$
where $\Gamma_{C_1}$ fixes the curve $C_1$ pointwise.
By Corollary~\ref{corollary:80-small-orbits}, there are only finitely many  $\overline{G}_{80}$-orbits in $\mathbb{P}^3$
of length at most $20$.
Thus, we see that $|\Gamma_{C_1}|<4$, so that either $\Gamma_{C_1}$ is trivial or $\Gamma_{C_1}\cong\mumu_2$.

By Lemma~\ref{lemma:80-subgroups}, the group $\Gamma$ is one of the groups $\mumu_2$, $\mumu_2^2$, $\mumu_2^3$, $\mumu_2^4$ or $\mumu_5$.
Since
$$
r|\Gamma|=|\overline{G}_{80}|=80
$$
and $15>dr$, we see that either $r=5$ or $10$.

Suppose that $r=5$. Then $\Gamma\cong\mumu_2^4$ and $\Gamma/\Gamma_{C_1}\supset\mumu_2^3$.
Hence, the curve $C_1$ is not rational, because the group $\mumu_2^3$ does not act faithfully on $\mathbb{P}^1$.
Since $d\leqslant 3$, the curve $C_1$ is a smooth plane cubic curve.
Hence there is a $\overline{G}_{80}$-invariant set of $5$ planes in $\mathbb{P}^3$, which implies  that there exists a $\overline{G}_{80}$-orbit in $\mathbb{P}^3$ of length $5$.
The latter is impossible by Corollary~\ref{corollary:80-small-orbits}.

Thus, we see that $r=10$, which implies that $d=1$, so that $C_1$ is a line.
In this case $\Gamma\cong\mumu_2^3$ and $\Gamma_{C_1}\cong\mumu_2$,
because $\mumu_2^3$ cannot act faithfully on $\mathbb{P}^1$.
Then $C_1$ is one of two lines that is pointwise fixed by $\Gamma_{C_1}$, cf. Lemma~\ref{lemma:Heisenberg-subgroup}(i).
Now it follows from Lemma~\ref{lemma:80-L10} that $C_1$ is an irreducible component of one of the curves
$\mathcal{L}_{10}$, $\mathcal{L}_{10}^\prime$, $\mathcal{L}_{10}^{\prime\prime}$.
\end{proof}

\begin{corollary}
\label{corollary:80-small-degree-curves}
There are no $\overline{G}_{80}$-invariant curves in $\mathbb{P}^3$ of degree less than $5$.
\end{corollary}

\begin{proof}
Let $C$ be a $\overline{G}_{80}$-invariant curve in $\mathbb{P}^3$ of degree less $5$.
Then $C$ is irreducible by Lemma~\ref{lemma:80-curves-reducible}.
Thus the genus of its normalization is at most $3$, which is impossible by Lemma~\ref{lemma:80-sporadic-genera}.
\end{proof}

Using Lemma~\ref{lemma:Heseinebrg-lines-quadrics}, we obtain the following two technical results.

\begin{lemma}
\label{lemma:80-L10-sextics}
Let $\mathcal{M}$ (respectively, $\mathcal{M}^\prime$, $\mathcal{M}^{\prime\prime}$)
be the linear system of surfaces in $\mathbb{P}^3$ of degree~$6$ (respectively, $4$, $4$)
that contains $\mathcal{L}_{10}$ (respectively, $\mathcal{L}_{10}^\prime$, $\mathcal{L}_{10}^{\prime\prime}$).
Then the base locus of~$\mathcal{M}$ (respectively, $\mathcal{M}^\prime$, $\mathcal{M}^{\prime\prime}$)
does not contain curves except for $\mathcal{L}_{10}$ (respectively,~$\mathcal{L}_{10}^\prime$,~$\mathcal{L}_{10}^{\prime\prime}$).
\end{lemma}

\begin{proof}
Recall from~\S\ref{section:Heisenberg}, that $\mathbb{P}^3$ contains ten $\overline{\mathrm{H}}$-invariant quadrics $\mathcal{Q}_i$, where~\mbox{$1\leqslant i\leqslant 10$}, which are defined by equations \eqref{equation:quadrics-1-2-3-4-5} and \eqref{equation:quadrics-6-7-8-9-10}.
For every $i<j<k$, define the reducible sextic surface
$$
T_{i,j,k}=\mathcal{Q}_i+\mathcal{Q}_j+\mathcal{Q}_k.
$$
By Lemma~\ref{lemma:Heseinebrg-lines-quadrics}, the curve $\mathcal{L}_{10}$ is contained in
$T_{1,2,4}$, $T_{3,5,6}$,  $T_{2,4,5}$, $T_{2,3,5}$, $T_{7,8,9}$, $T_{1,6,10}$ and~$T_{6,9,10}$,
so that these surfaces are contained in the linear system $\mathcal{M}$.
Using Lemma~\ref{lemma:Heseinebrg-lines-quadrics}, one can check that
$$
T_{1,2,4}\cap T_{3,5,6}=\mathcal{L}_{10}\cup\ell_1^\prime\cup\check{\ell}_1^\prime\cup\ell_3^\prime\cup\check{\ell}_3^\prime\cup\ell_2^{\prime\prime}\cup\check{\ell}_2^{\prime\prime}\cup\ell_3^{\prime\prime}\cup\check{\ell}_3^{\prime\prime}\cup\ell_4^{\prime\prime}\cup\check{\ell}_4^{\prime\prime}\cup\ell_5^{\prime\prime}\cup\check{\ell}_5^{\prime\prime}.
$$
On the other hand, the sextic  surface $T_{2,4,5}$ does not contain the lines $\ell_1^{\prime}$ and $\check{\ell}_1^{\prime}$,
and the sextic  surface $T_{2,3,5}$ does not contain the lines $\ell_3^{\prime}$ and $\check{\ell}_3^{\prime}$.
This also follows from Lemma~\ref{lemma:Heseinebrg-lines-quadrics}.
Similarly, the sextic  surface $T_{7,8,9}$ does not contain
the lines $\ell_4^{\prime\prime}$, $\check{\ell}_4^{\prime\prime}$, $\ell_5^{\prime\prime}$ and~$\check{\ell}_5^{\prime\prime}$,
the sextic  surface $T_{1,6,10}$ does not contain the lines $\ell_3^{\prime\prime}$ and $\check{\ell}_3^{\prime\prime}$,
and the sextic surface~$T_{6,9,10}$ does not contain the lines $\ell_2^{\prime\prime}$ and $\check{\ell}_2^{\prime\prime}$.
Thus, the intersection of the above seven sextic surfaces does not contain other curves except for $\mathcal{L}_{10}$,
so that the same holds for the base locus of $\mathcal{M}$.

For every $i<j$, define the reducible quartic surface
$$
R_{i,j}=\mathcal{Q}_i+\mathcal{Q}_j.
$$
By Lemma~\ref{lemma:Heseinebrg-lines-quadrics}, the curve $\mathcal{L}_{10}^\prime$ is contained in $R_{6,7}$, $R_{8,10}$ and  $R_{9,10}$,
so that these surfaces are contained in the linear system $\mathcal{M}^\prime$.
Using Lemma~\ref{lemma:Heseinebrg-lines-quadrics} one more time, we see that
$$
R_{6,7}\cap R_{8,10}=\mathcal{L}_{10}^\prime\cup\ell_5\cup\check{\ell}_5.
$$
However, the surface  $R_{9,10}$ does not contain the lines $\ell_5$ and $\check{\ell}_5$.
This shows that the base locus of $\mathcal{M}^\prime$ does not contain curves except for $\mathcal{L}_{10}^\prime$.

Similarly, the quartics $R_{1,2}$, $R_{3,4}$ and $R_{4,5}$ contain the curve $\mathcal{L}_{10}^{\prime\prime}$.
Using Lemma~\ref{lemma:Heseinebrg-lines-quadrics}, one can check that
$$
R_{1,2}\cap R_{3,4}=\mathcal{L}_{10}^{\prime\prime}\cup\ell_4\cup\check{\ell}_4.
$$
However, the surface  $R_{4,5}$ does not contain the lines $\ell_4$ and $\check{\ell}_5$,
which implies that the base locus of $\mathcal{M}^{\prime\prime}$ does not contain curves except for $\mathcal{L}_{10}^{\prime\prime}$.
\end{proof}

\begin{lemma}
\label{lemma:80-L10-mult}
Let $\mathcal{D}$ be a (non-empty) mobile $\overline{G}_{80}$-invariant linear system on $\mathbb{P}^3$,
and let~$n$ be a positive integer such that $\mathcal{D}\sim\mathcal{O}_{\mathbb{P}^3}(n)$.
Then
$$
\mathrm{max}\Big\{\mathrm{mult}_{\mathcal{L}_{10}}\big(\mathcal{D}\big), \mathrm{mult}_{\mathcal{L}_{10}^\prime}\big(\mathcal{D}\big), \mathrm{mult}_{\mathcal{L}_{10}^{\prime\prime}}\big(\mathcal{D}\big)\Big\}\leqslant\frac{n}{4}.
$$
\end{lemma}

\begin{proof}
Let us use the notation of the proof of Lemma~\ref{lemma:80-L10}.
Let $D$ be a general surface in~$\mathcal{D}$.
Since the lines $\ell_1$, $\ell_5$, $\check{\ell}_1$ and~$\check{\ell}_5$
are contained in the quadric surface $\mathcal{Q}_1$ by Lemma~\ref{lemma:Heseinebrg-lines-quadrics},
we have
$$
D\vert_{\mathcal{Q}_1}=m\big(\ell_1+\ell_5+\check{\ell}_1+\check{\ell}_5\big)+\Delta,
$$
where $m$ is a non-negative integer such that $m\geqslant\mathrm{mult}_{\mathcal{L}_{10}}(\mathcal{D})$,
and $\Delta$ is an effective divisor on $\mathcal{Q}_1$ whose support does not contain the lines $\ell_1$, $\ell_5$, $\check{\ell}_1$ and~$\check{\ell}_5$.
By Lemma~\ref{lemma:80-L10}, the lines  $\ell_1$, $\ell_5$, $\check{\ell}_1$ and~$\check{\ell}_5$ are disjoint.
Thus, we may assume that these lines are divisors of bi-degree~\mbox{$(1,0)$} on $\mathcal{Q}_1\cong\mathbb{P}^1\times\mathbb{P}^1$.
Let $\ell$ be a general line in $\mathcal{Q}_1$ that is a divisor of bi-degree~\mbox{$(0,1)$}. Then
$$
n=D\cdot\ell=m\big(\ell_1+\ell_5+\check{\ell}_1+\check{\ell}_5\big)\cdot\ell+\Delta\cdot\ell\geqslant m\big(\ell_1+\ell_5+\check{\ell}_1+\check{\ell}_5\big)\cdot\ell=4m,
$$
so that $\mathrm{mult}_{\mathcal{L}_{10}}(\mathcal{D})\leqslant m\leqslant\frac{n}{4}$ as required.
Similarly, it follows from Lemma~\ref{lemma:Heseinebrg-lines-quadrics} that the quadric $\mathcal{Q}_1$ also contains the lines $\ell_1^{\prime\prime}$, $\ell_2^{\prime\prime}$, $\ell_4^{\prime\prime}$,
$\check{\ell}_1^{\prime\prime}$, $\check{\ell}_2^{\prime\prime}$ and~$\check{\ell}_4^{\prime\prime}$.
Moreover, it follows from the proof of Lemma~\ref{lemma:80-L10} that
the four lines $\ell_1^{\prime\prime}$, $\ell_2^{\prime\prime}$, $\check{\ell}_1^{\prime\prime}$ and~$\check{\ell}_2^{\prime\prime}$ are disjoint.
Arguing as in the previous case, we see that $\mathrm{mult}_{\mathcal{L}_{10}^{\prime\prime}}(\mathcal{D})\leqslant\frac{n}{4}$.
Finally, observe that it also follows from Lemma~\ref{lemma:Heseinebrg-lines-quadrics} that
the quadric $\mathcal{Q}_6$ contains the lines $\ell_1^{\prime}$, $\ell_2^{\prime}$, $\ell_3^{\prime}$,
$\check{\ell}_1^{\prime}$, $\check{\ell}_2^{\prime}$ and~$\check{\ell}_3^{\prime}$.
Furthermore, in the proof of Lemma~\ref{lemma:80-L10} we established that
the four lines $\ell_1^{\prime}$, $\ell_3^{\prime}$, $\check{\ell}_1^{\prime}$ and~$\check{\ell}_3^{\prime}$ are disjoint,
which implies that~\mbox{$\mathrm{mult}_{\mathcal{L}_{10}^{\prime}}(\mathcal{D})\leqslant\frac{n}{4}$}.
\end{proof}

\begin{remark}
\label{remrak:Todd}
As we mentioned in the proof of Lemma~\ref{lemma:80-L10},
the lines $\ell_1^\prime$, $\ell_2^\prime$, $\ell_3^\prime$, $\ell_4^\prime$ and~$\ell_5^\prime$ form a ``wheel'' in $\mathbb{P}^3$,
the lines $\check{\ell}_1^\prime$, $\check{\ell}_2^\prime$, $\check{\ell}_3^\prime$, $\check{\ell}_4^\prime$ and~$\check{\ell}_5^\prime$ also form a ``wheel'',
and these two wheels intersect by ten points in $\Sigma_{20}^\prime$.
Using Lemma~\ref{lemma:80-L10-sextics}, one can check that there exists a $\overline{G}_{80}$-commutative diagram
$$
\xymatrix{
&&W\ar@{->}[lld]_\pi\ar@{->}[drr]^\alpha&&\\
\mathbb{P}^3\ar@{-->}[rrrr]_\rho &&&& X_4}
$$
Here $\pi$ is the blow up of the curve $\mathcal{L}_{10}^\prime$,
the morphism $\alpha$ is given by the linear system~\mbox{$|-K_W|$},
the map $\rho$ is given by be the linear system of quartic surfaces in $\mathbb{P}^3$
that contain $\mathcal{L}_{10}^\prime$, and $X_4$ is a quartic threefold in $\mathbb{P}^4$.
This configuration of ten lines in $\mathbb{P}^3$ has been studied by Todd in \cite{Todd35}
under generality assumption (and without group action).
He constructed the same commutative diagram with $\alpha$ being a small birational morphism.
However, the position of our $10$ lines is not general from this point of view.
Namely, it follows from the proof of Lemma~\ref{lemma:80-L10-mult}
that the map $\alpha$ contracts the proper transforms of the quadrics $\mathcal{Q}_6$,
$\mathcal{Q}_7$, $\mathcal{Q}_8$, $\mathcal{Q}_9$ and $\mathcal{Q}_{10}$
(so that our  $\overline{G}_{80}$-commutative diagram can be called a \emph{bad $\overline{G}_{80}$-Sarkisov link}
similarly to the terminology of \cite{CPR}).
One can consider this as a geometrical meaning of Lemma~\ref{lemma:80-L10-mult}.
Similarly, the lines of $\mathcal{L}_{10}^{\prime\prime}$ form the same configuration in $\mathbb{P}^3$,
which results in a similar $\overline{G}_{80}$-commutative diagram;
the only difference in this case is that $\alpha$ contracts
the proper transforms of the quadrics $\mathcal{Q}_1$,
$\mathcal{Q}_2$, $\mathcal{Q}_3$, $\mathcal{Q}_4$ and $\mathcal{Q}_{5}$.
\end{remark}

Denote by $S_0$, $S_1$, $S_2$, $S_3$ and $S_4$ the quartic surfaces in $\mathbb{P}^3$
that are given by the equations $q_0=0$, $q_1=0$, $q_2=0$, $q_3=0$ and $q_4=0$, respectively.
Here $q_0$, $q_1$, $q_2$, $q_3$ and $q_4$ are quartic polynomials in \eqref{equation:80-q0-q1-q2-q3-q4}.
Then $S_0$, $S_1$, $S_2$, $S_3$ and $S_4$ are the only $\overline{G}_{80}$-invariant quartic surfaces in $\mathbb{P}^3$ by Lemma~\ref{lemma:80-144-quadrics}.

\begin{lemma}
\label{lemma:80-quartic-surfaces}
The polynomials $q_0$, $q_1$, $q_2$, $q_3$ and $q_4$ are irreducible.
Moreover, the surface~$S_0$ is smooth,
and the surfaces $S_1$, $S_2$, $S_3$ and $S_4$ have isolated ordinary double points.
Furthermore, we have
$$
\mathrm{Sing}(S_1)=\Sigma_{16}^1,\quad
\mathrm{Sing}(S_2)=\Sigma_{16}^2,\quad
\mathrm{Sing}(S_3)=\Sigma_{16}^3,\quad
\mathrm{Sing}(S_4)=\Sigma_{16}^4.
$$
\end{lemma}

\begin{proof}
We can deduce all required assertions using the explicit formulas
of the polynomials $q_0$, $q_1$, $q_2$, $q_3$ and $q_4$.
However, most of these assertions follow easily from general facts.
Namely, the irreducibility of the polynomials $q_0$, $q_1$, $q_2$, $q_3$ and $q_4$ follows from Lemma~\ref{lemma:80-144-quadrics}.
Now taking the general plane section of the surface $S_i$,
we see that either it has isolated singularities, or its singular locus contains a $\overline{G}_{80}$-curve of degree less than $4$.
The latter is impossible by Corollary~\ref{corollary:80-small-degree-curves}.
Thus, the surfaces $S_0$, $S_1$, $S_2$, $S_3$ and $S_4$ have isolated singularities.

By \cite[Theorem~1]{Umezu81}, the surface $S_i$ cannot have more that two non-Du Val singular points.
Thus, the surface $S_i$ has at most Du Val singularities by Corollary~\ref{corollary:80-small-orbits}.
Then its minimal resolution of singularities is a smooth $K3$ surface.
Now using Corollary~\ref{corollary:80-small-orbits} and the fact that the rank of the Picard group of a smooth $K3$ surface is at most $20$,
we see that either $S_i$ is smooth, or $S_i$ has
isolated ordinary double points, and its singular locus is one of the $\overline{G}_{80}$-orbits $\Sigma_{16}^1$,
$\Sigma_{16}^2$, $\Sigma_{16}^3$ or~$\Sigma_{16}^4$.
Taking partial derivatives of the polynomial~$q_0$
at the points of the set $\Sigma_{16}^1\cup\Sigma_{16}^2\cup\Sigma_{16}^3\cup\Sigma_{16}^4$,
we see that $S_0$ is smooth.
Similarly, we see that $\mathrm{Sing}(S_i)=\Sigma_{16}^i$ for $1\leqslant i\leqslant 4$.
\end{proof}

\begin{corollary}
\label{corollary:80-cubics-16-points}
Let $\Sigma$ be one of the $\overline{G}_{80}$-orbits $\Sigma_{16}^1$, $\Sigma_{16}^2$, $\Sigma_{16}^3$, $\Sigma_{16}^4$,
and let $\mathcal{M}$ be the linear system that consists of cubic surfaces in $\mathbb{P}^3$ containing $\Sigma$.
Then $\Sigma$ is the base locus of the linear system $\mathcal{M}$.
\end{corollary}

\begin{proof}
See the proof of Corollary~\ref{corollary:144-cubics-12-points}.
\end{proof}

The $\overline{G}_{80}$-orbits $\Sigma_{20}$, $\Sigma_{20}^\prime$ and $\Sigma_{20}^{\prime\prime}$ are not contained in any of the quartic surfaces $S_0$, $S_1$, $S_2$, $S_3$ and $S_4$.
By Lemma~\ref{lemma:80-L10}  and Corollary~\ref{corollary:80-small-orbits}, this shows that each curve among $\mathcal{L}_{10}$, $\mathcal{L}_{10}^\prime$, $\mathcal{L}_{10}^{\prime\prime}$ intersects each surface
among $S_0$, $S_1$, $S_2$, $S_3$ and $S_4$ transversally in a $\overline{G}_{80}$-orbit of length~$40$.
Furthermore, by Lemma~\ref{lemma:80-L10},
the curves $\mathcal{L}_{10}$, $\mathcal{L}_{10}^\prime$ and $\mathcal{L}_{10}^{\prime\prime}$
contain all $\overline{G}_{80}$-orbits of length~$40$, and they are disjoint away from the $\overline{G}_{80}$-orbits of length~$20$.
Thus, we conclude
that the surfaces $S_0$, $S_1$, $S_2$, $S_3$ and $S_4$ contain exactly three $\overline{G}_{80}$-orbits of length~$40$.
Observe also that $S_0$ contains all $\overline{G}_{80}$-orbits $\Sigma_{16}^1$, $\Sigma_{16}^2$, $\Sigma_{16}^3$, $\Sigma_{16}^4$.
This gives the following

\begin{corollary}
\label{corollary:80-Mukai}
The group $\mathrm{Pic}(S_0)^{\overline{G}_{80}}$ is generated by the plane section of~$S_0$.
\end{corollary}

\begin{proof}
By \cite[Lemma~2.1]{Mukai}, the action of the group $\overline{G}_{80}$ on the $K3$ surface $S_0$ is symplectic.
Thus, the surface $S_0/\overline{G}_{80}$ is a singular $K3$ surface.
By construction, it has four Du Val singular points of type~$\mathbb{A}_4$ and three Du Val singular points of
type~$\mathbb{A}_1$.
The minimal resolution of singularities of this surface is a smooth $K3$ surface.
Since the rank of the Picard group of a smooth $K3$ surfaces cannot exceed $20$,
we see that the rank of the Picard group of the surface $S_0/\overline{G}_{80}$ must be~$1$.
This shows that the rank of the group~\mbox{$\mathrm{Pic}(S_0)^{\overline{G}_{80}}$} is also~$1$.
Since $\mathrm{Pic}(S_0)$ has no torsion and the intersection form of curves on the surface~$S_0$ is even,
we see that the group $\mathrm{Pic}(S_0)^{\overline{G}_{80}}$ is generated by the plane section of~$S_0$.
\end{proof}

Table~\ref{table:Sigma-S} shows which $\overline{G}_{80}$-orbit $\Sigma_{16}^1$, $\Sigma_{16}^2$, $\Sigma_{16}^3$
and~$\Sigma_{16}^4$
is contained in which surface  $S_1$, $S_2$, $S_3$ and~$S_4$.
\begin{table}\renewcommand\arraystretch{1.8}
\caption{Incidence between $S_i$ and $\Sigma_{16}^j$.\label{table:Sigma-S}}
\begin{center}\renewcommand\arraystretch{1.8}
\begin{tabular}{|c||c|c|c|c|}
  \hline
   $\qquad$& $S_1$ & $S_2$ & $S_3$ & $S_4$ \\
  \hline
  \hline
  $\Sigma_{16}^1$ & $\mathrm{Sing}(S_1)$ & $+$ & $-$ & $+$\\
  \hline
  $\Sigma_{16}^2$ & $-$ & $\mathrm{Sing}(S_2)$ & $+$ & $+$\\
  \hline
  $\Sigma_{16}^3$ & $+$ & $+$ & $\mathrm{Sing}(S_3)$ & $-$\\
  \hline
  $\Sigma_{16}^4$ & $+$ & $-$ & $+$ & $\mathrm{Sing}(S_4)$\\
  \hline
\end{tabular}
\end{center}
\end{table}
In particular, we see that $\Sigma_{16}^j\subset S_i$ if and only if $j\ne 2i\ \mathrm{mod}\ 5$.

Now we are ready to prove

\begin{theorem}
\label{theorem:80-curves}
Let $C$ be a  $\overline{G}_{80}$-irreducible curve in $\mathbb{P}^3$ such that $\mathrm{deg}(C)\leqslant 15$.
Then~\mbox{$C\not\subset S_0$} and one of the following cases holds:
\begin{itemize}
\item $\mathrm{deg}(C)=10$, and $C$ is one of the curves $\mathcal{L}_{10}$, $\mathcal{L}_{10}^\prime$, $\mathcal{L}_{10}^{\prime\prime}$,

\item $\mathrm{deg}(C)=12$, and $C$ is an irreducible smooth curve of genus $5$ that is contained in exactly one surface among $S_1$, $S_2$, $S_3$, $S_4$,

\item $\mathrm{deg}(C)=8$, and $C$ is an irreducible smooth curve of genus $5$ such that either $C=S_1\cap S_4$
and $C$ contains $\mathrm{Sing}(S_1)\cup\mathrm{Sing}(S_1)$,
or $C=S_2\cap S_3$ and $C$ contains $\mathrm{Sing}(S_2)\cup\mathrm{Sing}(S_3)$.
\end{itemize}
\end{theorem}

\begin{proof}
Recall from Corollary~\ref{corollary:80-Mukai} that the group $\mathrm{Pic}(S_0)^{\overline{G}_{80}}$ is generated by the plane section $H_{S_0}$ of the surface~$S_0$.
Since there is an exact sequence of $\overline{G}_{80}$-representations
$$
0\longrightarrow H^0\Big(\mathcal{O}_{\mathbb{P}^3}\big(n-4\big)\Big)
\longrightarrow H^0\Big(\mathcal{O}_{\mathbb{P}^3}\big(n\big)\Big)
\longrightarrow H^0\Big(\mathcal{O}_{S_0}\big(nH_{S_0}\big)\Big)\longrightarrow 0
$$
for every $n\geqslant 1$, we see that $S_0$ does not contain $C$ by Lemma~\ref{lemma:80-144-quadrics}.

By Lemma~\ref{lemma:80-curves-reducible}, we may assume that $C$ is irreducible.
Let $d=\mathrm{deg}(C)$.
Then
$$
4d=S_0\cdot C=16a+40b
$$
for some non-negative integers $a$ and $b$, because $\Sigma_{20}$, $\Sigma_{20}^\prime$ and $\Sigma_{20}^{\prime\prime}$ are not contained in $S_0$.
Keeping in mind that $d\geqslant 5$ by Corollary~\ref{corollary:80-small-degree-curves}, we have the following possibilities:
\begin{center}\renewcommand\arraystretch{1.1}
\begin{tabular}{|c|c|c|}
\hline
$\quad$ $d$ $\quad$& $\quad$ $a$ $\quad$ & $\quad$ $b$ $\quad$ \\
\hline\hline
$8$ & $2$ & $0$ \\
\hline
$10$ & $0$ & $1$ \\
\hline
$12$ & $3$ & $0$ \\
\hline
$14$ & $1$ & $1$ \\
\hline
\end{tabular}
\end{center}

Suppose that $d=14$.
Since $a=1$ and $S_0$ contains all $\overline{G}_{80}$-orbits of length $16$ in $\mathbb{P}^3$,
we see that the curve $C$ contains a unique $\overline{G}_{80}$-orbit of length $16$.
Without loss of generality, we may assume that $C$ contains $\Sigma_{16}^1$.
Then $C\not\subset S_3$, because $S_3$ does not contain $\Sigma_{16}^1$, see Table~\ref{table:Sigma-S}.
By~Corollary~\ref{corollary:80-small-orbits}, this implies that
$$
56=4d=S_3\cdot C
$$
is divisible by $40$, which is absurd.
This shows that~\mbox{$d\ne 14$}, so that $d\leqslant 12$.

Pick four general points $O_1$, $O_2$, $O_3$ and $O_4$ in the curve $C$.
Since $C$ is irreducible, their stabilizers in $\overline{\mathrm{H}}$ are trivial.
In particular, the $\overline{\mathrm{H}}$-orbits of these points consist of $16$ different points.
On the other hand, the vector space generated by the quartic polynomials
$$
q_0(x,y,z,w),\ q_1(x,y,z,w),\ q_2(x,y,z,w),\ q_3(x,y,z,w),\ q_4(x,y,z,w),
$$
contains all $\mathrm{H}$-invariant polynomials of degree $4$.
There exists $[\lambda_0:\lambda_1:\lambda_2:\lambda_3:\lambda_4]\in\mathbb{P}^4$ such that the polynomial
$$
\lambda_0q_0(x,y,z,w)+\lambda_1q_1(x,y,z,w)+\lambda_2q_2(x,y,z,w)+\lambda_3q_3(x,y,z,w)+\lambda_4q_4(x,y,z,w)
$$
vanishes at the four points $O_1$, $O_2$, $O_3$ and $O_4$.
The latter polynomial defines a (possibly reducible or non-reduced) $\overline{\mathrm{H}}$-invariant quartic surface $S$ in $\mathbb{P}^3$.
Then $S$ contains $O_1$, $O_2$, $O_3$ and $O_4$ together with their $\overline{\mathrm{H}}$-orbits.
This implies that
$$
4d=S\cdot C\geqslant 64
$$
provided that the curve $C$ is not contained in $S$.
Since $d\leqslant 12$, we see that $C\subset S$.
Thus, the vector space of all $\mathrm{H}$-invariant quartic polynomials in $\mathbb{C}[x,y,z,w]$ that vanish along~$C$ is at least one-dimensional.
Since $\mathrm{H}$ is a normal subgroup in $G_{80}$, the group
$$
G_{80}/\mathrm{H}\cong\mumu_5
$$
acts on this space.
Keeping in mind that $\mumu_5$ is abelian, we see that there exists at least one $G_{80}$-invariant quartic polynomial in $\mathbb{C}[x,y,z,w]$
that vanishes along the curve~$C$.
Moreover, if $d\leqslant 10$, then the rank-nullity theorem implies that
the vector space consisting of all \mbox{$\mathrm{H}$-invariant} quartic polynomials in $\mathbb{C}[x,y,z,w]$ that vanish along $C$ is at least two-dimensional,
because we can repeat the same arguments with three general points in $C$ instead of four.
Therefore, if $d\leqslant 10$, then there exists at least two linearly independent $G_{80}$-invariant quartic polynomials in $\mathbb{C}[x,y,z,w]$
that vanish at the curve $C$.
Now using Lemma~\ref{lemma:80-quartic-surfaces},
we see that $C$ is contained in one of the surfaces $S_1$, $S_2$, $S_3$ or $S_4$.
Moreover, if $d\leqslant 10$, then $C$ is contained in at least two surfaces among  $S_1$, $S_2$, $S_3$, $S_4$.

Let $t\in\{1,2,3,4\}$ be such that $C$ is contained in $S_t$.
Denote by $H$ the plane section of the surface $S_t$.
If $d\ne 12$, then there is $k\in\{1,2,3,4\}$ such that $C$ is contained in the surfaces $S_k$ and $k\ne t$.
In this case, we have
$$
S_k\vert_{S_t}=C+Z,
$$
where $Z$ is a $\overline{G}_{80}$-invariant effective divisor such that $H\cdot Z=16-d\leqslant 8$.
On the other hand, we already proved earlier that $\mathbb{P}^3$ does not contain $\overline{G}_{80}$-invariant curves of degree less than $8$.
This shows that $d\ne 10$, so that either $d=8$ or $d=12$.

Recall that $\Sigma_{16}^t$ is the singular locus of the surface $S_t$.
Let $f\colon\widetilde{S}_t\to S_t$ be the minimal resolution of singularities of the surface $S_t$.
Then the action of the group $\overline{G}_{80}$ lifts to the surface $\widetilde{S}_t$, because $f$ is the minimal resolution of singularities.
Denote by $E_1,\ldots,E_{16}$ the $f$-exceptional curves, and denote by $\widetilde{C}$ the proper transform of the curve $C$ on the surface $\widetilde{S}_t$.
Let $E=E_1+\ldots+E_{16}$ and $\widetilde{H}=f^*(H)$. Then
$$
\widetilde{C}\sim_{\mathbb{Q}} f^{*}(C)-\frac{m}{2}E
$$
for some non-negative integer number $m$.
By Hodge index theorem, one has
$$
0\geqslant\left|\begin{matrix} %
\widetilde{H}^2& \widetilde{H}\cdot(\widetilde{C}+\frac{m}{2}E)\cr%
\widetilde{H}\cdot(\widetilde{C}+\frac{m}{2}E)& (\widetilde{C}+\frac{m}{2}E)^2\cr%
\end{matrix}\right|=
\left|\begin{matrix} %
4 & d\cr%
d & (\widetilde{C}+\frac{m}{2}E)^2\cr%
\end{matrix}\right|=4(\widetilde{C}^2+8m^2)-d^2.
$$
This implies
$$
\widetilde{C}^2\leqslant\frac{d^2}{4}-8m^2.
$$
Let $g$ be the genus of the normalization of the curve $\widetilde{C}$,
and let $p_a(\widetilde{C})$ be its arithmetic genus.
Then
\begin{multline*}
g\leqslant p_a(\widetilde{C})-\big|\mathrm{Sing}(\widetilde{C})\big|=\frac{\widetilde{C}^2}{2}+1-\big|\mathrm{Sing}(\widetilde{C})\big|\leqslant\\
\leqslant\frac{d^2}{8}-4m^2+1-\big|\mathrm{Sing}(\widetilde{C})\big|\leqslant\\
\leqslant\frac{144}{8}-4m^2+1-\big|\mathrm{Sing}(\widetilde{C})\big|=19-4m^2-\big|\mathrm{Sing}(\widetilde{C})\big|.
\end{multline*}
On the other hand, it follows from  Lemma~\ref{lemma:80-sporadic-genera} that $g\geqslant 5$.
Thus, as $\mathbb{P}^3$ does not have $\overline{G}_{80}$-orbits of length less than $16$ by Corollary~\ref{corollary:80-small-orbits}, we see that the curve $\widetilde{C}$ is smooth, one has
$$
g=p_a(\widetilde{C})\leqslant 19-4m^2,
$$
and $m\leqslant 1$.
In particular, since $m\leqslant 1$, we see that the curve $C$ is also smooth.
Moreover, since $g\leqslant 19$, we have $g\in\{5,13,17\}$ by Lemma~\ref{lemma:80-sporadic-genera}.

Suppose that $g=17$.
Then $C$ contains three $\overline{G}_{80}$-orbits of length $16$ by Lemma~\ref{lemma:80-sporadic-genera}.
Thus, the curve $C$ contains every $\overline{G}_{80}$-orbit $\Sigma_{16}^i$ except for $\Sigma_{16}^j$ such
that~\mbox{$j=2t\ \mathrm{mod}\ 5$},
because these are the only $\overline{G}_{80}$-orbits of length $16$ contained in the surface $S_t$.
Since $C$ contains~\mbox{$\Sigma_{16}^t=\mathrm{Sing}(S_t)$}, we have $m\geqslant 1$, which implies that $m=1$, because we proved already that $m\leqslant 1$.
Since we also proved that $g\leqslant 19-4m^2$, we have
$$
17=g\leqslant 19-4m^2=15,
$$
which is absurd. This shows that $g\ne 17$.

Suppose that $g=13$.
Then $C$ contains one $\overline{G}_{80}$-orbit of length $16$ by Lemma~\ref{lemma:80-sporadic-genera}.
Recall that there exists $r\in\{1,2,3,4\}$ such that the surface $S_r$ that does not contain this $\overline{G}_{80}$-orbit.
Then $C\not\subset S_r$, so that $S_r\cdot C$ must be a multiple of $40$ by Lemma~\ref{lemma:80-sporadic-genera}. Since
$$
S_r\cdot C=4d,
$$
we see that $d$ is divisible by $10$.
This is impossible, because we already proved that either~\mbox{$d=8$} or $d=12$.
This shows that $g\ne 13$.

We have seen that $C$ is a smooth curve of degree $d\in\{8,12\}$ and genus $5$ that is contained in $S_t$ such that $t\ne 0$.
In particular, we have seen that there are no $\overline{G}_{80}$-invariant curves in~$\mathbb{P}^3$ of degree less than $8$.
Thus, if $d=12$, then $S_t$ is the only surface among $S_1$, $S_2$, $S_3$ and~$S_4$ that contains $C$.
Therefore, to complete the proof of the theorem, we may assume that~\mbox{$d=8$}.
We have to show that
either $C=S_1\cap S_4$ and $C$ contains $\mathrm{Sing}(S_1)\cup\mathrm{Sing}(S_1)$,
or $C=S_2\cap S_3$ and $C$ contains $\mathrm{Sing}(S_2)\cup\mathrm{Sing}(S_3)$.

Recall that $C$ is also contained in a quartic surface $S_k$ such that $k\ne t$. Then
\begin{equation}
\label{equation:80-C-and-Z}
4H\sim S_k\vert_{S_t}=C+Z,
\end{equation}
where $Z$ is a $\overline{G}_{80}$-invariant effective divisor such that $H\cdot Z=8$.
Since we already proved that $\mathbb{P}^3$ does not contain $\overline{G}_{80}$-invariant curves of degree less than $8$,
we see that $Z$ is an irreducible curve of degree $8$.
Moreover, we already proved that every $\overline{G}_{80}$-invariant curve in $\mathbb{P}^3$ of degree $8$ is a smooth curve of genus $5$.
In particular, we see that $Z$ is a smooth curve of genus $5$.
Note that, a priori, we may have the case when $Z=C$ (and it will turn out that this is exactly the case).

Suppose that both curves $C$ and $Z$ do not contain the singular locus of the surface $S_t$.
Then $C^2=Z^2=8$ by the adjunction formula, which gives
$$
32=4H\cdot C=C\cdot (C+Z)=C^2+C\cdot Z=8+C\cdot Z,
$$
which implies that $C\cdot Z=24$. This is impossible.
Indeed, if we have $C=Z$, then~\mbox{$C\cdot Z=C^2=8$}.
Similarly, if $C\ne Z$, then
$$
C\cdot Z=16\alpha+40\beta
$$
for some non-negative integers $\alpha$ and $\beta$, because $C$ only contains $\overline{G}_{80}$-orbits of length $16$, $40$ and $80$ by Lemma~\ref{lemma:80-sporadic-genera}.
Thus, we see that $C\cdot Z\ne 24$.
This shows that either $C$ or $Z$ contains $\Sigma_{16}^t=\mathrm{Sing}(S_t)$.

If $\Sigma_{16}^t\subset C$, then $\Sigma_{16}^t\subset Z$ by \eqref{equation:80-C-and-Z}, because both $C$ and $Z$ are smooth.
Similarly, if~\mbox{$\Sigma_{16}^t\subset Z$}, then $\Sigma_{16}^t\subset C$.
Thus, we see that $\Sigma_{16}^t\subset C\cap Z$.
Applying the same arguments to the surface $S_k$, we see that
$$
\mathrm{Sing}(S_k)=\Sigma_{16}^k\subset C\cap Z,
$$
which implies that $k+t=5$ (see Table~\ref{table:Sigma-S}).

Denote by $\widetilde{Z}$ the proper transform of the curve $Z$ on the surface $\widetilde{S}_t$.
Then
$$
\widetilde{C}+\widetilde{Z}\sim f^{*}(4H)-E.
$$
This follows from the fact that $C+Z$ is a Cartier divisor at every point of $\Sigma_{16}^t$,
because both $C$ and $Z$ are smooth and $\Sigma_{16}^t\subset C\cap Z$.
As above, the adjunction formula gives~\mbox{$\widetilde{C}^2=\widetilde{Z}^2=8$}.
Then
$$
8+\widetilde{C}\cdot\widetilde{Z}=\widetilde{C}^2+\widetilde{C}\cdot\widetilde{Z}=\widetilde{C}\cdot\big(\widetilde{C}+\widetilde{Z}\big)=\widetilde{C}\cdot\big(f^{*}(4H)-E\big)=32-\widetilde{C}\cdot E=16,
$$
so that $\widetilde{C}\cdot\widetilde{Z}=8$.
Keeping in mind that $\widetilde{C}$ only contains $\overline{G}_{80}$-orbits of length $16$, $40$,~$80$,
we see that $\widetilde{C}=\widetilde{Z}$.
Then $C=Z$, so that $S_k\cap S_t=C$ and $k+t=5$ as required. This completes the proof of Theorem~\ref{theorem:80-curves}.
\end{proof}

It should be pointed out that Theorem~\ref{theorem:80-curves} does not assert that
the projective space~$\mathbb{P}^3$ contains $\overline{G}_{80}$-irreducible curves of degree $8$ and $12$.

\begin{remark}
\label{remark:McKelvey}
The $\overline{G}_{80}$-invariant curves of degree~$8$ in $\mathbb{P}^3$ do exist.
In fact, there are exactly two of them.
Indeed, one can show that the ideal in $\mathbb{C}[x,y,z,w]$ generated by the polynomials
\begin{multline*}
q_1(x,y,z,w),\quad q_4(x,y,z,w),\quad \frac{\partial q_1}{\partial x}\frac{\partial q_4}{\partial y}-\frac{\partial q_4}{\partial x}\frac{\partial q_1}{\partial y},\quad \frac{\partial q_1}{\partial x}\frac{\partial q_4}{\partial z}-\frac{\partial q_4}{\partial x}\frac{\partial q_1}{\partial z},\\
\frac{\partial q_1}{\partial x}\frac{\partial q_4}{\partial w}-\frac{\partial q_4}{\partial x}\frac{\partial q_1}{\partial w},\quad \frac{\partial q_1}{\partial y}\frac{\partial q_4}{\partial z}-\frac{\partial q_4}{\partial y}\frac{\partial q_1}{\partial z},\quad \frac{\partial q_1}{\partial y}\frac{\partial q_4}{\partial w}-\frac{\partial q_4}{\partial y}\frac{\partial q_1}{\partial w},\quad \frac{\partial q_1}{\partial z}\frac{\partial q_4}{\partial w}-\frac{\partial q_4}{\partial z}\frac{\partial q_1}{\partial w}
\end{multline*}
defines a one-dimensional subscheme in $\mathbb{P}^3$.
This shows that the scheme-theoretic intersection $S_1\cdot S_4$ is not reduced.
By Theorem~\ref{theorem:80-curves}, this implies that
the set-theoretic intersection $S_1\cap S_4$ is a $\overline{G}_{80}$-invariant irreducible smooth curve of genus $5$.
Similarly, we see that $S_2\cap S_3$ is another $\overline{G}_{80}$-invariant irreducible smooth curve of genus $5$.
Note that the curves $S_1\cap S_4$ and $S_2\cap S_3$ are projectively equivalent, since
$$
\overline{R}(S_1)=S_2,\quad
\overline{R}(S_2)=S_4,\quad
\overline{R}(S_4)=S_3,\quad
\overline{R}(S_3)=S_1.
$$
This also implies that both of these curves are $\overline{G}_{160}$-invariant,
and they are swapped by the group $\overline{G}_{320}$.
In particular, there exists a smooth irreducible curve of genus $5$ with an action of the group
$\overline{G}_{160}\cong\mumu_2^4\rtimes\mathrm{D}_{10}$.
This curve is well-known: it is a complete intersection in $\mathbb{P}^4$ that is given by
\begin{equation}
\label{equation:80-Humbert-curve}
\left\{\aligned%
&x_0^2+x_1^2+x_3^2+x_4^2+x_5^2=0,\\
&x_0^2+\xi_5 x_1^2+\xi_5^2x_3^2+\xi_5^3 x_4^2+\xi_5^4 x_5^2=0,\\
&\xi_5^4 x_0^2+\xi_5^3 x_1^2+\xi_5^2 x_3^2+\xi_5 x_4^2+x_5^2=0,\\
\endaligned
\right.
\end{equation}
where $\xi_5$ is a primitive fifth root of unity.
This was proved by Joseph McKelvey in \cite{McKelvey}, who also proved that
$\mumu_2^4\rtimes\mathrm{D}_{10}$ is the full automorphism group of the curve~\eqref{equation:80-Humbert-curve},
see also~\mbox{\cite{Kuribayashi,Ruben}}.
The curve given by~\eqref{equation:80-Humbert-curve} is a special case of the so-called Humbert curves,
which have been discovered by Humbert in \cite{Humbert}, and have been studied by Edge in \cite{Edge}.
\end{remark}

\begin{remark}
\label{remark:80-L10-curve-degree-8}
Let $C$ be a $\overline{G}_{80}$-invariant curve of degree~$8$ in $\mathbb{P}^3$.
Then $C$ is a smooth curve of genus $5$ by Theorem~\ref{theorem:80-curves}.
Thus, it contains unique $\overline{G}_{80}$-orbit of length $40$ by Lemma~\ref{lemma:80-sporadic-genera}.
We claim that this orbit is the intersection $\mathcal{L}_{10}\cap C$.
To show this, let us describe the intersections $\ell_1\cap S_1$, $\ell_1\cap S_2$, $\ell_1\cap S_3$ and $\ell_1\cap S_4$,
where $\ell_1$ is the irreducible component of the curve $\mathcal{L}_{10}$ defined in \eqref{equation:ell-1-2-3-4-5}.
Note that $\ell_1$ is the line in $\mathbb{P}^3$ that passes through the points $[0:i:1:0]$ and $[1:0:0:-i]$,
so that is parameterized by
$$
\big[\mu:\lambda i:\lambda:-\mu i\big],
$$
where $[\lambda:\mu]\in\mathbb{P}^1$.
To find the intersections $\ell_1\cap S_1$, $\ell_1\cap S_2$, $\ell_1\cap S_3$ and $\ell_1\cap S_4$,
we may assume that $\lambda=1$, because $[0:i:1:0]$ is not contained in $S_1\cup S_2\cup S_3\cup S_4$.
Substituting the point $[\mu:i:1:-\mu i]$ into the polynomials $q_1$, $q_2$, $q_3$, $q_4$ defined in \eqref{equation:80-q0-q1-q2-q3-q4},
we see that both intersections $\ell_1\cap S_1$ and $\ell_1\cap S_4$ are given by
$$
\mu^4-(4\xi_5^3+4\xi_5^2+2)\mu^2+1=0,
$$
and both intersections $\ell_1\cap S_2$ and $\ell_1\cap S_3$ are given by
$$
\mu^4+(4\xi_5^3+4\xi_5^2+2)\mu^2+1=0.
$$
Thus, we see that $\ell_1\cap S_1=\ell_1\cap S_4$ and $\ell_1\cap S_2=\ell_1\cap S_3$.
Since $C$ is one of the curves $S_1\cap S_4$ or $S_2\cap S_3$ by Theorem~\ref{theorem:80-curves},
the intersection $\mathcal{L}_{10}\cap C$ is the unique $\overline{G}_{80}$-orbit in $C$ of length $40$.
\end{remark}

We will show later that $\mathbb{P}^3$ contains $\overline{G}_{80}$-invariant curves of degree $12$ (see Lemma~\ref{lemma:80-degree-12}).

\begin{remark}
\label{remark:80-degree-12}
Let $C$ be a $\overline{G}_{80}$-invariant curve in $\mathbb{P}^3$ of degree $12$.
By Theorem~\ref{theorem:80-curves}, the curve $C$ is smooth, its genus is $5$, and
$C\subset S_t$, where $S_t$ is one of the surfaces $S_1$, $S_2$, $S_3$ and~$S_4$.
By Remark~\ref{remark:McKelvey}, the surface $S_t$ contains an irreducible $\overline{G}_{80}$-invariant curve $Z$
such that $2Z=S_k\vert_{S_t}$, where $k+t=5$.
If $C$ does not contain $\mathrm{Sing}(S_t)$, then
$Z\cdot C=24$, which is impossible, because the only $\overline{G}_{80}$-orbits contained in $C$ are of lengths
$16$, $40$ and $80$ by Lemma~\ref{lemma:80-sporadic-genera}.
Thus, we see that $C$ contains the singular locus of the surface $S_t$.
\end{remark}

The first consequence of Theorem~\ref{theorem:80-curves} is

\begin{corollary}
\label{corollary:80-C12-mult}
Let $C$ be an irreducible $\overline{G}_{80}$-invariant curve in $\mathbb{P}^3$ of degree $12$,
let $\mathcal{D}$ be a mobile $\overline{G}_{80}$-invariant linear system on $\mathbb{P}^3$,
let $n$ be a positive integer such that $\mathcal{D}\sim\mathcal{O}_{\mathbb{P}^3}(n)$.
Then $\mathrm{mult}_{C}(\mathcal{D})\leqslant\frac{n}{4}$.
\end{corollary}

\begin{proof}
By Theorem~\ref{theorem:80-curves},
the curve $C$ is a smooth curve of genus $5$ that is contained in a surface among $S_1$, $S_2$, $S_3$, $S_4$.
Denote this surface by $S$, and denote by $H$ its general plane section.
Let $D$ be a general surface in $\mathcal{D}$.
Then
$$
D\big\vert_{S}=mC+\Delta\sim nH,
$$
where $m$ is a non-negative integer such that $m\geqslant\mathrm{mult}_{C}(\mathcal{D})$,
and $\Delta$ is an effective divisor on $S$ such that $C\not\subset\mathrm{Supp}(\Delta)$.

Recall from Lemma~\ref{lemma:80-quartic-surfaces}, that $\mathrm{Sing}(S)$ is one of the $\overline{G}_{80}$-orbits $\Sigma_{16}^1$, $\Sigma_{16}^2$, $\Sigma_{16}^3$, $\Sigma_{16}^4$,
and $S$ has isolated ordinary double points.
Let $f\colon\widetilde{S}\to S$ be the minimal resolution of singularities of the surface~$S$.
Then the action of the group $\overline{G}_{80}$ lifts to the surface $\widetilde{S}$.
Denote by $E_1,\ldots,E_{16}$ the $f$-exceptional curves,
denote by $\widetilde{C}$ the proper transform of the curve $C$ on the surface $\widetilde{S}$,
and denote by $\widetilde{\Delta}$ the proper transform of the divisor $\Delta$ on the surface $\widetilde{S}$.
Let $E=E_1+\ldots+E_{16}$. Then $E^2=-32$ and
$$
m\widetilde{C}+\widetilde{\Delta}\sim_{\mathbb{Q}} f^*(nH)-\epsilon E
$$
for some non-negative rational number $\epsilon$.
Moreover, the curve $\widetilde{C}$ is a smooth irreducible curve of genus $5$, so that $\widetilde{C}^2=8$ by the adjunction formula.

By Remark~\ref{remark:80-degree-12}, the curve $C$ contains $\mathrm{Sing}(S)$.
Then $\widetilde{C}\cdot E=16$. Let $T$ be the divisor~\mbox{$f^*(6H)-\widetilde{C}-E$}.
Then $T\cdot E=16$ and $T^2=8$, so that
\begin{multline*}
h^0\big(\mathcal{O}_{\widetilde{S}}(T)\big)=6+h^1\big(\mathcal{O}_{\widetilde{S}}(T)\big)-h^2\big(\mathcal{O}_{\widetilde{S}}(T)\big)=\\
=6+h^1\big(\mathcal{O}_{\widetilde{S}}(T)\big)+h^0\big(\mathcal{O}_{\widetilde{S}}(-T)\big)=6+h^1\big(\mathcal{O}_{\widetilde{S}}(T)\big)\geqslant 6
\end{multline*}
by the Riemann--Roch formula and Serre duality, since $T\cdot f^{*}(H)=12>0$.

We claim that the linear system $|T|$ does not have base curves.
Indeed, suppose it does. Then $T\sim F+M$,
where $F$ is the fixed part of the linear system $|T|$,
and $M$ is its mobile part.
Then $M\cdot f^*(H)\geqslant 3$, because $S$ is not uniruled.
This gives
$$
0\leqslant f(F)\cdot H=F\cdot f^*(H)=12-M\cdot f^*(H)\leqslant 9.
$$
Since $F$ is $\overline{G}_{80}$-invariant, it follows from Theorem~\ref{theorem:80-curves} that
either $F$ is contracted by $f$, or $f(F)$ is an irreducible smooth curve of degree $8$ and genus $5$.
In the former case we have $F=rE$ for some integer $r\geqslant 1$, so that
$$
M^2=\Big(T-rE\Big)^2=6-2rT\cdot E-32r^2\leqslant 6-32r^2<0,
$$
which is absurd. Thus, we see that $f(F)$ is an irreducible smooth curve of degree $8$ and genus $5$.
Denote its proper transform on $\widetilde{S}$ by $\widetilde{C}_8$.
Then $\widetilde{C}_8\subset\mathrm{Supp}(F)$ and $\widetilde{C}_8^2=8$.
Hence, it follows from the Riemann--Roch formula and Serre duality that
\begin{multline*}
h^0\big(\mathcal{O}_{\widetilde{S}}(F)\big)\geqslant h^0\big(\mathcal{O}_{\widetilde{S}}(\widetilde{C}_8)\big)=6+h^1\big(\mathcal{O}_{\widetilde{S}}(\widetilde{C}_8)\big)-h^2\big(\mathcal{O}_{\widetilde{S}}(\widetilde{C}_8)\big)=\\
=6+h^1\big(\mathcal{O}_{\widetilde{S}}(\widetilde{C}_8)\big)+h^0\big(\mathcal{O}_{\widetilde{S}}(-\widetilde{C}_8)\big)=6+h^1\big(\mathcal{O}_{\widetilde{S}}(\widetilde{C}_8)\big)\geqslant 6,
\end{multline*}
which is also absurd, since $F$ is the fixed part of the linear system $|T|$.
Thus, we see that the linear system $|T|$ does not have base curves.

Let $Z$ be a general curve in $|T|$. Then $Z\cdot E=16$ and
$$
Z\cdot\widetilde{C}=\widetilde{C}\cdot\big(6f^*(H)-\widetilde{C}-E\big)=48.
$$
On the other hand, the divisor $Z$ is nef, so that $Z\cdot\widetilde{\Delta}\geqslant 0$.
Therefore, we have
$$
0\leqslant Z\cdot\widetilde{\Delta}=Z\cdot\Big(f^*(nH)-m\widetilde{C}-\epsilon E\Big)=12n-mZ\cdot\widetilde{C}-\epsilon Z\cdot E\leqslant 12n-48m,
$$
which implies that $\mathrm{mult}_{C}(\mathcal{D})\leqslant m\leqslant\frac{n}{4}$.
\end{proof}

\begin{corollary}
\label{corollary:80-C12-sextics}
Let $C$ be a $\overline{G}_{80}$-invariant curve of degree $12$ in $\mathbb{P}^3$,
and let $\mathcal{M}$ be the linear system of surfaces in $\mathbb{P}^3$ of degree $6$ that contains $C$.
Then the base locus of $\mathcal{M}$ does not have curves that are different from $C$.
\end{corollary}

\begin{proof}
The required assertion easily follows from the proof of Corollary~\ref{corollary:80-C12-mult}.
Namely, let us use notation of the proof of Corollary~\ref{corollary:80-C12-mult}.
Then the linear system $|f^*(6H)-\widetilde{C}-E|$ is free from base curves.
This was shown in the proof of Corollary~\ref{corollary:80-C12-mult}.
In particular, the linear system $|6H-C|$ does not have base points.
Since the quartic surface $S$ is projectively normal, we immediately obtain the required assertions.
\end{proof}

Now let us prove that $\overline{G}_{80}$-invariant curves of degree $12$ in $\mathbb{P}^3$ do exist.

\begin{lemma}
\label{lemma:80-degree-12-exists}
Each surface among $S_1$, $S_2$, $S_3$, $S_4$ contains
at least one $\overline{G}_{80}$-invariant curve of degree $12$.
\end{lemma}

\begin{proof}
Let $C$ be a smooth curve of genus $5$ that admits a faithful action of the group $\overline{G}_{80}$.
Then it follows from \cite[Proposition~2.2]{Dolgachev} that there exists an exact sequence of groups
\begin{equation}
\label{equation:Dolgachev-sequence}
0\longrightarrow\mathrm{Hom}\big(\overline{G}_{80},\mathbb{C}^\ast\big)\longrightarrow\mathrm{Pic}\big(\overline{G}_{80},C\big)\longrightarrow
\mathrm{Pic}\big(C\big)^{\overline{G}_{80}}\longrightarrow H^2\big(\overline{G}_{80},\mathbb{C}^\ast\big)\longrightarrow 0,
\end{equation}
where $\mathrm{Pic}(\overline{G}_{80},C)$ stands for the group of $\overline{G}_{80}$-equivariant line bundles on $C$.
One can compute that $\mathrm{Hom}(\overline{G}_{80},\mathbb{C}^\ast)\cong\mumu_5$ and $H^2(\overline{G}_{80},\mathbb{C}^\ast)\cong\mumu_2^2$.
Moreover, one has $C/\overline{G}_{80}\cong\mathbb{P}^1$, so that it follows from \cite[(2.2)]{Dolgachev} that
$$
\mathrm{Pic}\big(\overline{G}_{80},C\big)\cong\mathbb{Z}\oplus\mumu_5.
$$
Furthermore, let $\kappa$ be a positive generator of the free part of $\mathrm{Pic}(\overline{G}_{80},C)$.
Then it follows from \cite[(2.1)]{Dolgachev} that $\mathrm{deg}(\kappa)=8$, so that we may assume that $\kappa=K_C$.
Therefore, there is an exact sequence
$$
0\longrightarrow\mathbb{Z}\longrightarrow\mathrm{Pic}(C)^{\overline{G}_{80}}\longrightarrow\mumu_2^2\longrightarrow 0.
$$
Thus, either $\mathrm{Pic}(C)^{\overline{G}_{80}}\cong\mathbb{Z}\oplus\mumu_2$ or $\mathrm{Pic}(C)^{\overline{G}_{80}}\cong\mathbb{Z}\oplus\mumu_2^2$.
In both cases, there exists a $\overline{G}_{80}$-invariant $2$-torsion line bundle on $C$.
In the former case, such a $2$-torsion is unique, and there exists a $\overline{G}_{80}$-invariant line bundle   $\theta\in\mathrm{Pic}(C)$ such that $2\theta\sim K_C$,
i.e. the divisor $\theta$ is a $\overline{G}_{80}$-invariant theta-characteristic.
In the latter case, there are exactly three non-trivial $\overline{G}_{80}$-invariant $2$-torsion line bundles on $C$,
and there are no $\overline{G}_{80}$-invariant theta-characteristics.

Let $D$ be a $2$-torsion divisor on $C$ such that its class in $\mathrm{Pic}(C)$ is $\overline{G}_{80}$-invariant.
Then
$$
h^0\big(\mathcal{O}_C(K_C+D)\big)=h^0\big(\mathcal{O}_C(-D)\big)+4=4
$$
by the Riemann--Roch formula.
Moreover, the linear system $|K_C+D|$ does not have base points,
because the degree of the divisor $K_C+D$ is $8$,
and $C$ does not contain $\overline{G}_{80}$-orbits of length less than $16$
by Lemma~\ref{lemma:80-sporadic-genera}.
So, the linear system $|K_C+D|$ gives a $\overline{G}_{80}$-equivariant morphism
$$
\sigma\colon C\to\mathbb{P}\Big(H^0(\mathcal{O}_C(K_C+D))^\vee\Big).
$$
By Corollary~\ref{corollary:G80}, we can identify the action of the group $\overline{G}_{80}$ on $\mathbb{P}(H^0(\mathcal{O}_C(K_C+D))^\vee)$
with the action of the group $\overline{G}_{80}$ on $\mathbb{P}^3$ introduced in~\S\ref{section:Heisenberg}.
Note that $\sigma(C)$ is a $\overline{G}_{80}$-invariant curve of degree at most $\deg(K_C+D)=8$.
Thus, the degree of the curve $\sigma(C)$ is $8$ by Theorem~\ref{theorem:80-curves}.
In particular, this gives another proof that $\mathbb{P}^3$ contains a $\overline{G}_{80}$-invariant curve of degree $8$ (see Remark~\ref{remark:McKelvey}).
Note that Theorem~\ref{theorem:80-curves} also implies that the curve $\sigma(C)$ is smooth,
and either $\sigma(C)=S_1\cap S_4$ or~\mbox{$\sigma(C)=S_2\cap S_3$}.
Hence the morphism $\sigma$ induces an isomorphism $C\cong\sigma(C)$.
Since the curves $S_1\cap S_4$ and~\mbox{$S_2\cap S_3$} are projectively equivalent by Remark~\ref{remark:McKelvey},
we see that the class of the divisor $D$ in~$\mathrm{Pic}(C)$ is uniquely determined.

We see that there exists a unique $\overline{G}_{80}$-invariant $2$-torsion line bundle on $C$,
which implies that there exists a $\overline{G}_{80}$-invariant theta-characteristic $\theta\in\mathrm{Pic}(C)$.
In fact, there are exactly two such theta-characteristics: one is $\theta$, and another one is $\theta+D$.
Arguing as above, we see that both linear systems $|K_C+\theta|$ and $|K_C+D+\theta|$ are free from base points and define
$\overline{G}_{80}$-equivariant morphisms $C\to\mathbb{P}^7$,
where $\mathbb{P}^7$ are the projectivizations of~\mbox{$H^0(\mathcal{O}_C(K_C+\theta))^\vee$} and
$H^0(\mathcal{O}_C(K_C+D+\theta))^\vee$,
which are eight-dimensional representations of central extensions of the group $\overline{G}_{80}$ given by the images
of~\mbox{$K_C+\theta$} and~\mbox{$K_C+D+\theta$} in
$$
H^2(\overline{G}_{80},\mathbb{C}^\ast)\cong\mumu_2^2
$$
in~\eqref{equation:Dolgachev-sequence},
respectively.
One can show that these eight-dimensional representations split as sums of two four-dimensional representations,
so that there exists a $\overline{G}_{80}$-equivariant map~$\varsigma$ from $C$ to a three-dimensional projective space
that is given by some $\overline{G}_{80}$-invariant three-dimensional linear subsystem $\mathcal{P}$ in~\mbox{$|K_C+\theta|$}
or~\mbox{$|K_C+D+\theta|$}.
By Corollary~\ref{corollary:G80}, we can identify the action of the group $\overline{G}_{80}$ on the target space of $\varsigma$
with the action of the group $\overline{G}_{80}$ on $\mathbb{P}^3$ introduced in~\S\ref{section:Heisenberg}.
Since the degree of both divisors $K_C+\theta$ and $K_C+D+\theta$ is $12$, the linear system $\mathcal{P}$ does not have base points.
Thus, the curve $\varsigma(C)$ is a $\overline{G}_{80}$-invariant curve whose degree divides
$$
\deg(K_C+\theta)=\deg(K_C+D+\theta)=12.
$$
Hence, $\varsigma(C)$ is a smooth curve of degree $12$ and genus $5$ by Theorem~\ref{theorem:80-curves}.

By Theorem~\ref{theorem:80-curves}, the curve $\varsigma(C)$ is contained in a unique quartic surface among $S_1$, $S_2$, $S_3$
and~$S_4$.
Let $\overline{R}$ be the element in $\mathrm{PGL}_4(\mathbb{C})$ defined in~\S\ref{section:Heisenberg}, cf. Remark~\ref{remark:McKelvey}.
Then
$$
\overline{R}(S_1)=S_2,\quad
\overline{R}(S_2)=S_4,\quad
\overline{R}(S_4)=S_3,\quad
\overline{R}(S_3)=S_1,
$$
so that the cyclic group generated by $\overline{R}$ acts transitively on the set $\{S_1,S_2,S_3,S_4\}$.
Since this group also acts transitively on the set
$$
\Big\{\varsigma(C),\overline{R}(\varsigma(C)),\overline{R}^2(\varsigma(C)),\overline{R}^3(\varsigma(C))\Big\},
$$
we conclude that each surface among $S_1$, $S_2$, $S_3$, $S_4$ contains a unique
curve
among $\varsigma(C)$, $\overline{R}(\varsigma(C))$, $\overline{R}^2(\varsigma(C))$ and $\overline{R}^3(\varsigma(C))$.
\end{proof}

Now we are going to describe all $\overline{G}_{80}$-invariant curves of degree $12$ in $\mathbb{P}^3$.

\begin{lemma}
\label{lemma:80-degree-12}
There are exactly eight $\overline{G}_{80}$-invariant curves of degree $12$ in $\mathbb{P}^3$.
Moreover, each quartic surface among $S_1$, $S_2$, $S_3$, $S_4$ contains exactly two such curves.
Furthermore, the union of these two curves is cut out on the quartic surface by a $\overline{G}_{80}$-invariant sextic surface in $\mathbb{P}^3$.
\end{lemma}

\begin{proof}
By Theorem~\ref{theorem:80-curves}, it is enough to show that
$S_1$ contains exactly two  $\overline{G}_{80}$-invariant curves of degree $12$,
and the union of these two curves is cut out by a $\overline{G}_{80}$-invariant sextic surface in $\mathbb{P}^3$.
In the remaining cases, the proof is the same.

By Lemma~\ref{lemma:80-degree-12-exists}, the surface $S_1$ contains a $\overline{G}_{80}$-invariant curve of degree $12$.
Denote this curve by $C_{12}$, and consider the exact sequence of $G_{80}$-representations
\begin{multline*}
0\longrightarrow H^0\big(\mathcal{O}_{\mathbb{P}^3}(6)\otimes\mathcal{I}_{C_{12}}\big)\longrightarrow H^0\big(\mathcal{O}_{\mathbb{P}^3}(6)\big)\longrightarrow \\
\longrightarrow H^0\big(\mathcal{O}_{C_{12}}(9K_C)\big)\longrightarrow H^1\big(\mathcal{O}_{\mathbb{P}^3}(6)\otimes\mathcal{I}_{C_{12}}\big)\longrightarrow 0,
\end{multline*}
where $\mathcal{I}_{C_{12}}$ is the ideal sheaf of the curve $C_{12}$.
By Lemma~\ref{lemma:80-sextic-surfaces}, the vector space~\mbox{$H^0(\mathcal{O}_{\mathbb{P}^3}(6))$} contains exactly four one-dimensional subrepresentations of $G_{80}$,
and $H^0(\mathcal{O}_{C_{12}}(9K_C))$ contains exactly three one-dimensional subrepresentations of the group $G_{80}$,
which correspond to three $\overline{G}_{80}$-invariant effective divisors on $C$ of degree $72$.
Hence, there is a  $\overline{G}_{80}$-invariant sextic surface $R$ in $\mathbb{P}^3$ that contains the curve $C_{12}$.
By Lemma~\ref{lemma:80-sextic-surfaces}, the surface $R$ is irreducible and reduced, so that Theorem~\ref{theorem:80-curves} implies that
$$
R\big\vert_{S_1}=C_{12}+Z_{12},
$$
where $Z_{12}$ is an irreducible smooth $\overline{G}_{80}$-invariant curve of degree $12$ and genus $5$.

By Remark~\ref{remark:80-degree-12}, both curves $C_{12}$ and $Z_{12}$ contain the singular locus of the surface $S_1$.
Let $f\colon\widetilde{S}_1\to S_1$ be the minimal resolution of singularities of the surface $S_1$.
Then the action of the group $\overline{G}_{80}$ lifts to the surface $\widetilde{S}_1$.
Denote by $E_1,\ldots,E_{16}$ the exceptional curves of the birational morphism $f$.
Denote by $\widetilde{C}_{12}$ and $\widetilde{Z}_{12}$
the proper transforms of the curves $C_{12}$ and $Z_{12}$ on the surface $\widetilde{S}_1$, respectively.
Let $E=E_1+\ldots+E_{16}$ and let $H$ be a plane section of the surface~$S_1$.
Then
$$
\widetilde{C}_{12}+\widetilde{Z}_{12}\sim f^{*}(6H)-E,
$$
because both curves $C_{12}$ and $Z_{12}$ are smooth. Then
$\widetilde{C}_{12}\cdot\widetilde{Z}_{12}=48$,
which implies, in particular, that $\widetilde{C}_{12}\ne\widetilde{Z}_{12}$, because
$\widetilde{C}_{12}\cdot\widetilde{C}_{12}=\widetilde{Z}_{12}\cdot\widetilde{Z}_{12}=8$ by the adjunction formula.

To complete the proof, we have to show that $S_1$ does not contain $\overline{G}_{80}$-invariant curves of degree~$12$
that are distinct from $C_{12}$ and $Z_{12}$.
Suppose that this is not the case, i.e. the surface $S_1$ contains a $\overline{G}_{80}$-invariant curve $Z$ of degree $12$
such that $Z$ is distinct from the curves $C_{12}$ and $Z_{12}$.
Let us seek for a contradiction.

Let $\widetilde{Z}$ be the proper transform of the curve $Z$ on $\widetilde{S}_1$ via $f$.
Then $\widetilde{Z}\cdot(\widetilde{C}_{12}+\widetilde{Z}_{12})=56$.
Since the only $\overline{G}_{80}$-orbits in $\widetilde{Z}$ are of lengths $16$, $40$ and $80$ by Lemma~\ref{lemma:80-sporadic-genera},
we see that either
$\widetilde{Z}\cdot\widetilde{C}_{12}=16$ and $\widetilde{Z}\cdot\widetilde{Z}_{12}=40$,
or
$\widetilde{Z}\cdot\widetilde{C}_{12}=40$ and $\widetilde{Z}\cdot\widetilde{Z}_{12}=16$.
Without loss of generality, we may assume that
$\widetilde{Z}\cdot\widetilde{C}_{12}=16$ and $\widetilde{Z}\cdot\widetilde{Z}_{12}=40$.
Then
$$
\Big(\widetilde{C}_{12}+f^*(2H)-\widetilde{Z}\Big)^2=0.
$$
Thus, it follows from the Riemann--Roch formula and Serre duality that
\begin{multline*}
h^0\Big(\mathcal{O}_{\widetilde{S}}\big(\widetilde{C}_{12}+f^*(2H)-\widetilde{Z}\big)\Big)=\\
=2+h^1\Big(\mathcal{O}_{\widetilde{S}}\big(\widetilde{C}_{12}+f^*(2H)-\widetilde{Z}\big)\Big)-h^2\Big(\mathcal{O}_{\widetilde{S}}\big(\widetilde{C}_{12}+f^*(2H)-\widetilde{Z}\big)\Big)=\\
=2+h^1\Big(\mathcal{O}_{\widetilde{S}}\big(\widetilde{C}_{12}+f^*(2H)-\widetilde{Z}\big)\Big)-h^0\Big(\mathcal{O}_{\widetilde{S}}\big(\widetilde{Z}-\widetilde{C}_{12}-f^*(2H)\big)\Big)=\\
=2+h^1\Big(\mathcal{O}_{\widetilde{S}}\big(\widetilde{C}_{12}+f^*(2H)-\widetilde{Z}\big)\Big)\geqslant 2.
\end{multline*}
Hence, the linear system $|\widetilde{C}_{12}+f^*(2H)-\widetilde{Z}|$ is at least a pencil.
Moreover, it does not contain base curves.
Indeed, suppose it does. Then
$$
\widetilde{C}_{12}+f^*(2H)-\widetilde{Z}\sim F+M,
$$
where $F$ is the fixed part of the linear system $|\widetilde{C}_{12}+f^*(2H)-\widetilde{Z}|$,
and $M$ is its mobile part.
Then
$$
0\leqslant f(F)\cdot H=F\cdot f^*(H)=8-M\cdot f^*(H)<8.
$$
By Theorem~\ref{theorem:80-curves}, this implies that $F$ is contracted by $f$, so that $F=mE$ for some~\mbox{$m\geqslant 1$}. Then
$$
M^2=\Big(\widetilde{C}_{12}+f^*(2H)-\widetilde{Z}-mE\Big)^2=-32m^2,
$$
which is absurd. This shows that the linear system $|\widetilde{C}_{12}+f^*(2H)-\widetilde{Z}|$ is free from base curves.
Since
$$\big(\widetilde{C}_{12}+f^*(2H)-\widetilde{Z}\big)^2=0,
$$
we see that it has no base points,
so that it is composed of a base point free pencil.
This pencil gives a $\overline{G}_{80}$-equivariant morphism $\upsilon\colon\widetilde{S}\to\mathbb{P}^1$,
whose general fibers are smooth elliptic curves (by the adjunction formula).
This is impossible, because  $\overline{\mathrm{H}}$ is the only non-trivial proper normal
subgroup of the group $\overline{G}_{80}$ by Lemma~\ref{lemma:80-subgroups},
and the group~$\overline{\mathrm{H}}$ does not act faithfully on rational and elliptic curves by Lemma~\ref{lemma:Heisenberg-elliptic}.
This completes the proof of Lemma~\ref{lemma:80-degree-12}.
\end{proof}

Recall from Remark~\ref{remark:McKelvey} that there are two $\overline{G}_{80}$-invariant curves of degree $8$ in $\mathbb{P}^3$.
These curves are just the intersections $S_1\cap S_4$ or $S_2\cap S_3$ taken with the reduced structure.
Furthermore, these curves are smooth curves of genus $5$ by Theorem~\ref{theorem:80-curves}.

\begin{proposition}[{cf. \cite{BlancLamy,Cutrone,CutroneMarshburn}}]
\label{proposition:link-symmetric}
Let $C$ be a $\overline{G}_{80}$-invariant curve of degree $8$ in $\mathbb{P}^3$,
and let $\pi\colon X\to\mathbb{P}^3$ be a blow up of the curve $C$.
Then there exists a $\overline{G}_{80}$-commutative diagram
\begin{equation}
\label{equation:80-Sarkisov-link-symmetric}
\xymatrix{
X\ar@{->}[dd]_\pi\ar@{->}[dr]^\alpha\ar@{-->}[rrr]^\iota &&& X\ar@{->}[dd]^\pi\ar@{->}[ld]_\alpha \\
& V_8\ar@{->}[r]^\upsilon & V_8 &\\
\mathbb{P}^3\ar@{-->}[rrr]^\tau &&& \mathbb{P}^3}
\end{equation}
where $V_8$ is the Fano threefold with terminal Gorenstein singularities such that~\mbox{$-K_{V_8}^3=8$},
the map $\alpha$ is a flopping contraction,
the map $\tau$ is a birational map, $\upsilon$ is an isomorphism,
and $\iota$ is a composition of flops that acts on $\mathrm{Pic}(X)$ in the following way:
$$
\left\{\aligned%
&\iota^*\big(E\big)\sim 24\pi^*(H)-7E,\\
&\iota^*\big(\pi^*(H)\big)\sim 7\pi^*(H)-2E,
\endaligned
\right.
$$
where $H$ is a plane in $\mathbb{P}^3$, and $E$ is the $\pi$-exceptional divisor.
Moreover, the diagram~\eqref{equation:80-Sarkisov-link-symmetric} is also $\overline{G}_{160}$-commutative.
\end{proposition}

\begin{proof}
Without loss of generality, we may assume that~\mbox{$C=S_1\cap S_4$}.
Denote by $\widetilde{S}_1$ and~$\widetilde{S}_4$ the proper transforms of the surfaces $S_1$ and $S_4$ on the threefold~$X$, respectively.
Recall from Theorem~\ref{theorem:80-curves} that $C$ contains singular loci of the surfaces $S_1$ and $S_4$.
This implies that $\widetilde{S}_1$ and~$\widetilde{S}_4$ are smooth,
because the singularities of the surfaces $S_1$ and $S_4$ are isolated ordinary double points by Lemma~\ref{lemma:80-quartic-surfaces}. 
Furthermore, we have
$$
\widetilde{S}_1\sim\widetilde{S}_4\sim -K_X\sim\pi^*(4H)-E.
$$
Since $-K_{X}^3=8$, we see that the intersection $\widetilde{S}_1\cap\widetilde{S}_4$ is a $\overline{G}_{80}$-invariant smooth curve isomorphic to $C$.
Let us denote this curve by $\widetilde{C}$.
Note that $\widetilde{C}$ is the scheme-theoretic intersection of the surfaces~$\widetilde{S}_1$ and~$\widetilde{S}_4$,
and
$$
\widetilde{S}_1\cdot\widetilde{C}=\widetilde{S}_4\cdot\widetilde{C}=-K_{X}^3=8,
$$
so that $-K_X$ is nef and big.
By the Base Point Free Theorem (see \cite[Theorem~1.3.6]{IsPr99}),
the linear system $|-nK_X|$ gives a birational morphism $\alpha\colon X\to V_8$ for some $n\gg 0$,
where~$V_8$ is the Fano threefold with canonical Gorenstein singularities such that~\mbox{$-K_{V_8}^3=8$}.

Let us show that the birational morphism $\alpha$ is small (cf. \cite[Theorems~4.9~and~4.11]{JPR2005}).
Suppose that $\alpha$ is not small.
Then there exists a prime divisor $G$ that is contracted by~$\alpha$. One has $(-K_X)^2\cdot G=0$.
On the other hand, we have $G\sim\pi^*(aH)-bE$ for some positive integers $a$ and $b$.
Then $a=3b$, since
$$
0=(-K_X)^2\cdot G=\Big(\pi^*(4H)-E\Big)^2\cdot\Big(\pi^*(aH)-bE\Big)=8(a-3b).
$$
If $\alpha(G)$ is a point, then
$$
0=-K_X\cdot G^2=4(a^2-4ab+8b^2)=20b^2,
$$
which is absurd. Thus, we see that $\alpha(G)$ is a curve, so that the intersection $G\cap \widetilde{S}_1$ contains a curve.
Then
$$
0=(-K_X)^2\cdot G=G\vert_{\widetilde{S}_1}\cdot\widetilde{S}_2\vert_{\widetilde{S}_1}=G\vert_{\widetilde{S}_1}\cdot\widetilde{C},
$$
which implies that $\widetilde{C}$ is not an ample divisor on the surface $\widetilde{S}_1$.
On the other hand, arguing as in the proof of Corollary~\ref{corollary:80-C12-mult},
one can show that the linear system $|2\widetilde{C}-\pi^*(H)\vert_{\widetilde{S}_1}|$ is free from base points,
which implies the ampleness of the divisor
$$
\pi^*(H)\vert_{\widetilde{S}_1}+\Big(2\widetilde{C}-\pi^*(H)\vert_{\widetilde{S}_1}\Big)\sim 2\widetilde{C}.
$$
The obtained contradiction implies that the birational morphism $\alpha$ is small.

Since $\alpha$ is a small birational morphism,
the singularities of the threefold $V_8$ are terminal.
Moreover, there exists a $\overline{G}_{80}$-commutative  diagram
\begin{equation}
\label{equation:Sarkisov-link}
\xymatrix{
&&X\ar@{->}[lld]_\pi\ar@{->}[drr]^\alpha\ar@{-->}[rrrr]^\chi &&&& Y\ar@{->}[rrd]^\phi\ar@{->}[dll]_\beta &&\\
\mathbb{P}^3 &&&& V_8 &&&& Z}
\end{equation}
such that $\chi$ is a composition of flops,
$\beta$ is a small birational morphism, and $\phi$ is an extremal contraction such that $-K_Y$ is $\phi$-ample.
Observe that the threefold $Y$ is smooth, because the threefold $X$ is smooth and $\chi$ is a composition of flops.
Now arguing as in the proof of \cite[Proposition~4.4]{Cutrone},
we see that the morphism~$\phi$ is birational, one has $Z\cong\mathbb{P}^3$,
and $\phi$ is a blow up of a smooth curve of degree $8$ and genus~$5$.
Moreover, this proof also implies that
\begin{equation}
\label{equation:exceptional-divisor}
\widetilde{D}\sim 24\pi^{*}(H)-7E,
\end{equation}
where $\widetilde{D}$ is the proper transform on $X$ of the $\phi$-exceptional surface.

By Corollary~\ref{corollary:G80},
we may assume that the diagram \eqref{equation:Sarkisov-link} is $\overline{G}_{80}$-commutative.
Then~$\phi$ is a blow up of one of the two $\overline{G}_{80}$-invariant smooth curves of degree $8$ and genus~$5$.
By Remark~\ref{remark:McKelvey}, the group $\overline{G}_{320}$ swaps these curves.
Thus, composing the map $\phi$ with an appropriate element in~$\overline{G}_{320}$, we may assume that $\phi$ is a blow up of the curve $C$.
Therefore, we can identify $X$ with $Y$ and $\phi$ with $\pi$.
This gives the $\overline{G}_{80}$-commutative diagram~\eqref{equation:80-Sarkisov-link-symmetric}
with
$$
\iota=\chi,\quad \tau=\pi\circ\iota\circ\pi^{-1},\quad
\upsilon=\alpha\circ\iota\circ\alpha^{-1}.
$$
Since $V_8$ is an anticanonical model of the threefold $X$, the map $\upsilon$ must be biregular.
The action of $\iota$ on~\mbox{$\mathrm{Pic}(X)$} follows from~\eqref{equation:exceptional-divisor}, since $\iota^*(K_X)\sim K_X$.
The diagram~\eqref{equation:80-Sarkisov-link-symmetric} also $\overline{G}_{160}$-commutative, because $\mathrm{Aut}(C)\cong\overline{G}_{160}$ by Remark~\ref{remark:McKelvey}.
\end{proof}

Using Remark~\ref{remark:80-L10-curve-degree-8}, one can show that
the morphism $\alpha$ in Proposition~\ref{proposition:link-symmetric} contracts the proper transforms of the curve $\mathcal{L}_{10}$.
Using this, one can prove that $V_8$ in Proposition~\ref{proposition:link-symmetric} is a complete intersection of three quadrics in~$\mathbb{P}^6$ that has $10$ isolated ordinary double points.

\section{Birational rigidity}
\label{section:rigidity}

Let $G$ be a finite subgroup in $\mathrm{PGL}_4(\mathbb{C})$.
In this section we prove Theorems~\ref{theorem:main} and \ref{theorem:super-main}, which are the main results of our paper.
To do this, we will use the following well-known result, which goes back to the classical works of Noether, Fano, and Iskovskikh.

\begin{theorem}[{\cite[Theorem~3.3.1]{CheltsovShramov}}]
\label{theorem:NFI}
Let $\Gamma$ be a subset in $\mathrm{Bir}^{G}(\mathbb{P}^3)$.
Suppose that for every non-empty $G$-invariant mobile linear system $\mathcal{D}$ on $\mathbb{P}^3$,
there exists $\gamma\in\Gamma$ such that the log pair $(\mathbb{P}^3,\frac{4}{n}\gamma(\mathcal{D}))$ has canonical singularities,
where $n$ be a positive integer that is defined by $\gamma(\mathcal{D})\sim\mathcal{O}_{\mathbb{P}^3}(n)$.
Then $\mathbb{P}^3$ is $G$-birationally rigid,
and the group $\mathrm{Bir}^{G}(\mathbb{P}^3)$ is generated by $\Gamma$ and $\mathrm{Aut}^{G}(\mathbb{P}^3)$.
\end{theorem}

\begin{corollary}
\label{corollary:NFI}
Suppose that for every non-empty $G$-invariant mobile linear system $\mathcal{D}$ on~$\mathbb{P}^3$,
the log pair $(\mathbb{P}^3,\frac{4}{n}\mathcal{D})$ has canonical singularities,
where $n$ be a positive integer that is defined by $\mathcal{D}\sim\mathcal{O}_{\mathbb{P}^3}(n)$.
Then $\mathbb{P}^3$ is $G$-birationally super-rigid.
\end{corollary}

The following simple fact is useful for the proof of $G$-birational rigidity of $\mathbb{P}^3$.

\begin{lemma}
\label{lemma:root-deg}
Let $\mathcal{D}$ be a non-empty $G$-invariant mobile linear system on $\mathbb{P}^3$,
and let $Z$ be a $G$-irreducible curve in $\mathbb{P}^3$. Then
$$
\mathrm{mult}_{Z}\big(\mathcal{D}\big)\leqslant \frac{n}{\sqrt{\mathrm{deg}\big(Z\big)}},
$$
where $n$ is a positive integer such that $\mathcal{D}\sim\mathcal{O}_{\mathbb{P}^3}(n)$.
\end{lemma}

\begin{proof}
Let $D_1$ and $D_2$ be two general surfaces in $\mathcal{D}$.
Then
$$
D_1\cdot D_2=mZ+\Delta,
$$
where $m$ is a positive integer such that $m\geqslant\mathrm{mult}_{Z}^2(\mathcal{D})$, and $\Delta$ is an effective one-cycle on~$\mathbb{P}^3$ such that $\mathrm{Supp}(\Delta)$ does not contain irreducible components of the curve $Z$.
Let $H$ be a general plane in $\mathbb{P}^3$. Then
$$
n^2=H\cdot D_1\cdot D_2=H\cdot\Big(mZ+\Delta\Big)\geqslant m\cdot\mathrm{deg}(Z)\geqslant\mathrm{mult}_{Z}^2(\mathcal{D})\cdot\mathrm{deg}(Z),
$$
which implies the required inequality.
\end{proof}

To illustrate Lemma~\ref{lemma:root-deg} and Corollary~\ref{corollary:NFI}, let us prove

\begin{lemma}
\label{lemma:rigidity-A6}
Suppose that $G$ contains a proper subgroup isomorphic to $\mathfrak{A}_6$.
Then $\mathbb{P}^3$ is $G$-birationally super-rigid.
\end{lemma}

\begin{proof}
Let $\mathcal{D}$ be a non-empty $G$-invariant mobile linear system on $\mathbb{P}^3$.
Then~\mbox{$\mathcal{D}\sim\mathcal{O}_{\mathbb{P}^3}(n)$} for some positive integer $n$.
By Corollary~\ref{corollary:NFI}, to complete the proof,
we have to show that the singularities of the log pair $(\mathbb{P}^3,\frac{4}{n}\mathcal{D})$ are canonical.
By assumption, the group $G$ contains a subgroup isomorphic to $\mathfrak{A}_6$. Denote it by $\Theta$.
It follows from \cite[\S4]{ChSh09b} that there are two $\Theta$-irreducible curves $\mathcal{L}_6$ and $\mathcal{L}_6^\prime$ in $\mathbb{P}^3$
such that $\mathcal{L}_6\cap\mathcal{L}_6^\prime=\varnothing$,
and each of them is a disjoint union of $6$ lines.
Moreover, it follows from \cite[Theorem~4.3]{ChSh09b} that the log pair~\mbox{$(\mathbb{P}^3,\frac{4}{n}\mathcal{D})$}
has canonical singularities provided that
$$
\mathrm{max}\Big\{\mathrm{mult}_{\mathcal{L}_{6}}\big(\mathcal{D}\big),\mathrm{mult}_{\mathcal{L}_{6}^\prime}\big(\mathcal{D}\big)\Big\}\leqslant\frac{n}{4}.
$$
On the other hand, it follows from the proof of \cite[Theorem~4.9]{ChSh09b} that
$$
\mathrm{min}\Big\{\mathrm{mult}_{\mathcal{L}_{6}}\big(\mathcal{D}\big),\mathrm{mult}_{\mathcal{L}_{6}^\prime}\big(\mathcal{D}\big)\Big\}\leqslant\frac{n}{4}.
$$

Note that $\Theta$ is a primitive subgroup in $\mathrm{PGL}_4(\mathbb{C})$.
In particular, it does not have orbits of length at most $4$, and it does not leave invariant any line or a pair of lines in $\mathbb{P}^3$.
This implies that $G$ also enjoys these properties, so that it is primitive as well.
Thus, it follows from Blichfeldt's classification \cite[Chapter~VII]{Blichfeldt1917} that $G$ is one of the following groups: $\mathfrak{S}_6$,  $\mathfrak{A}_7$, or a simple group of order $25920$.
If $G\cong\mathfrak{S}_6$, then the log pair~\mbox{$(\mathbb{P}^3,\frac{4}{n}\mathcal{D})$}
has canonical singularities, because $G$ permutes the curves $\mathcal{L}_6$ and $\mathcal{L}_6^\prime$ in this case,
so that
$$
\mathrm{mult}_{\mathcal{L}_{6}}(\mathcal{D})=\mathrm{mult}_{\mathcal{L}_{6}^\prime}(\mathcal{D}).
$$
Similarly, if $G$ is a simple group of order $25920$, then
the log pair $(\mathbb{P}^3,\frac{4}{n}\mathcal{D})$
also has canonical singularities, because this group contains $\mathfrak{S}_6$ as a subgroup (see \cite[p.~26]{Atlas}).

To complete the proof, we may assume that $G\cong\mathfrak{A}_7$.
Let $Z$ be a $G$-orbit of the curve~$\mathcal{L}_{6}$.
Then~\mbox{$\mathrm{deg}(Z)=42$}, so that $\mathrm{mult}_{\mathcal{L}_{6}}(\mathcal{D})\leqslant\frac{n}{4}$  by Lemma~\ref{lemma:root-deg}.
Similarly, we see that \mbox{$\mathrm{mult}_{\mathcal{L}_{6}^\prime}(\mathcal{D})\leqslant\frac{n}{4}$}.
Thus, it follows from \cite[Theorem~4.3]{ChSh09b} that the log pair $(\mathbb{P}^3,\frac{4}{n}\mathcal{D})$ has canonical singularities.
\end{proof}

\begin{remark}
\label{remark:subgroup-rigidity}
Let $\Theta$ be a subgroup of the group $G$.
If $\mathbb{P}^3$ is $\Theta$-birationally super-rigid,
then it is also $G$-birationally super-rigid. This follows from Corollary~\ref{corollary:NFI} and equivariant Minimal Model Program (see \cite[Corollary~1.4.3]{Ch05}).
However, if $\mathbb{P}^3$ is $\Theta$-birationally rigid,
then we cannot immediately conclude that $\mathbb{P}^3$ is $G$-birationally rigid.
A priori, the projective space $\mathbb{P}^3$ can be $G$-birational to a Fano threefold~$X$ with terminal singularities
such that the $G$-invariant class group of the threefold $X$ is of rank~$1$,
but the $\Theta$-invariant class group is not of rank~$1$.
We refer the reader to~\cite{Kollar-rigidity} for a discussion of a similar problem for the action of Galois groups.
\end{remark}

Now we will state two technical propositions and use them to prove Theorems~\ref{theorem:main} and~\ref{theorem:super-main}.
After this, we will present the proofs of these propositions.
Let $\overline{G}_{144}$ be the primitive subgroup in $\mathrm{PGL}_4(\mathbb{C})$ constructed in~\S\ref{section:Heisenberg} and dealt with in~\S\ref{section:144}.

\begin{proposition}
\label{proposition:main-144}
Let $\mathcal{D}$ be a non-empty $\overline{G}_{144}$-invariant mobile linear system on $\mathbb{P}^3$,
and let $n$ be a positive integer such that $\mathcal{D}\sim\mathcal{O}_{\mathbb{P}^3}(n)$.
Then the log pair $(\mathbb{P}^3,\frac{4}{n}\mathcal{D})$ has canonical singularities.
\end{proposition}

Proposition~\ref{proposition:main-144} and Corollary~\ref{corollary:NFI} imply the following:

\begin{corollary}
\label{corollary:144}
Suppose that $G$ contains $\overline{G}_{144}$.
Then $\mathbb{P}^3$ is $G$-birationally super-rigid.
\end{corollary}

Similarly, let $\overline{G}_{80}$ and $\overline{G}_{160}$ be primitive subgroups in $\mathrm{PGL}_4(\mathbb{C})$ constructed in~\S\ref{section:Heisenberg} and dealt with in~\S\ref{section:80-160-320}.
Let $\mathcal{C}_1$ and $\mathcal{C}_2$ be two distinct $\overline{G}_{80}$-invariant irreducible smooth curves of degree $8$ and genus $5$ in $\mathbb{P}^3$ constructed in~\S\ref{section:80-160-320}.

\begin{proposition}
\label{proposition:main-80}
Let $\mathcal{D}$ be a non-empty $\overline{G}_{80}$-invariant mobile linear system on $\mathbb{P}^3$,
and let $n$ be a positive integer such that $\mathcal{D}\sim\mathcal{O}_{\mathbb{P}^3}(n)$.
Suppose that
$$
\mathrm{max}\Big\{\mathrm{mult}_{\mathcal{C}_1}\big(\mathcal{D}\big),\mathrm{mult}_{\mathcal{C}_2}\big(\mathcal{D}\big)\Big\}\leqslant\frac{n}{4}.
$$
Then the log pair $(\mathbb{P}^3,\frac{4}{n}\mathcal{D})$ has canonical singularities.
\end{proposition}

Proposition~\ref{proposition:main-80} and Theorem~\ref{theorem:NFI} imply the following:

\begin{corollary}
\label{corollary:80-160}
Suppose that $G=\overline{G}_{80}$ or $G=\overline{G}_{160}$.
Then $\mathbb{P}^3$ is $G$-birationally rigid.
\end{corollary}

\begin{proof}
Let $\Gamma$ be a subgroup in $\mathrm{Bir}^{G}(\mathbb{P}^3)$ that is generated by the birational maps constructed in Proposition~\ref{proposition:link-symmetric}.
Let $\mathcal{D}$ be a non-empty $G$-invariant mobile linear system on~$\mathbb{P}^3$.
For every $\sigma\in\Gamma$, let $n_{\sigma}$ be a positive integer such that
$\sigma(\mathcal{D})\sim\mathcal{O}_{\mathbb{P}^3}(n_\sigma)$.
Then there is $\gamma\in\Gamma$ such that
$$
n_\gamma=\mathrm{min}\Big\{n_\sigma\ \Big\vert\ \sigma\in\Gamma\Big\}.
$$
Let $n=n_\gamma$ and $\mathcal{M}=\gamma(\mathcal{D})$.
By Theorem~\ref{theorem:NFI}, to prove the required assertion,
it is enough to show that the log pair $(\mathbb{P}^3,\frac{4}{n}\mathcal{M})$ has canonical singularities.
Suppose that this is not true.
By Proposition~\ref{proposition:main-80}, either
$\mathrm{mult}_{\mathcal{C}_1}(\mathcal{M})>\frac{n}{4}$ or $\mathrm{mult}_{\mathcal{C}_2}(\mathcal{M})>\frac{n}{4}$.
In the former case, we let $C=\mathcal{C}_1$.
In the latter case, we let $C=\mathcal{C}_2$.
Let $\tau$ be the birational map constructed in Proposition~\ref{proposition:link-symmetric} starting from the curve $C$.
Then $\tau\in\Gamma$ and
$$
\tau\big(\mathcal{M}\big)\sim\mathcal{O}_{\mathbb{P}^3}\Big(7n-24\mathrm{mult}_{C}\big(\mathcal{M}\big)\Big)
$$
by Proposition~\ref{proposition:link-symmetric}. This contradicts the minimality of $n$, since $7n-24\mathrm{mult}_{C}(\mathcal{D})<n$.
\end{proof}

To prove Theorems~\ref{theorem:main} and \ref{theorem:super-main}, we need the last auxiliary result.

\begin{lemma}
\label{lemma:rigidity-80}
Suppose that the group $G$ contains $\overline{G}_{80}$, and that $G$ does not coincide with~$\overline{G}_{80}$
and~$\overline{G}_{160}$.
Then $\mathbb{P}^3$ is $G$-birationally super-rigid.
\end{lemma}

\begin{proof}
Let $\mathcal{D}$ be a non-empty $G$-invariant mobile linear system on $\mathbb{P}^3$.
Then $\mathcal{D}\sim\mathcal{O}_{\mathbb{P}^3}(n)$ for some positive integer $n$.
By Corollary~\ref{corollary:NFI}, to complete the proof,
we have to show that the singularities of the log pair $(\mathbb{P}^3,\frac{4}{n}\mathcal{D})$ are canonical.
This follows from Proposition~\ref{proposition:main-80} provided that
$$
\mathrm{max}\Big\{\mathrm{mult}_{\mathcal{C}_1}\big(\mathcal{D}\big),\mathrm{mult}_{\mathcal{C}_2}\big(\mathcal{D}\big)\Big\}\leqslant\frac{n}{4},
$$
because $\mathcal{D}$ is also $\overline{G}_{80}$-invariant. Let us prove the latter inequality.

Let $Z$ be a $G$-orbit of the curve $\mathcal{C}_1$.
Recall from~\S\ref{section:80-160-320} that $\mathcal{C}_1$ is $\overline{G}_{160}$-invariant.
Moreover, it follows from Remark~\ref{remark:McKelvey} that $Z\ne\mathcal{C}_1$.
Thus, we have $\mathrm{deg}(Z)\geqslant 2\mathrm{deg}(\mathcal{C}_1)=16$,
so that $\mathrm{mult}_{\mathcal{C}_1}(\mathcal{D})\leqslant\frac{n}{4}$ by Lemma~\ref{lemma:root-deg}.
Similarly, we see that $\mathrm{mult}_{\mathcal{C}_2}(\mathcal{D})\leqslant\frac{n}{4}$.
Therefore, the log pair $(\mathbb{P}^3,\frac{4}{n}\mathcal{D})$ has canonical singularities by Proposition~\ref{proposition:main-80}.
\end{proof}

Now we are ready to prove the main results of our paper.

\begin{proof}[Proof of Theorems~\ref{theorem:main} and \ref{theorem:super-main}]
By Corollary~\ref{corollary:non-primitive}, we may assume that $G$ is a primitive subgroup in $\mathrm{PGL}_4(\mathbb{C})$.
By Corollary~\ref{corollary:A5-S5}, we may assume that $G$ is not isomorphic to $\mathfrak{A}_5$ and $\mathfrak{S}_5$.
For the remaining groups, we use the classification provided by Blichfeldt in~\mbox{\cite[Chapter~VII]{Blichfeldt1917}},
see Appendix~\ref{section:diagram}.

If $G\cong\mathrm{PSL}_2(\mathbf{F}_7)$,
then $\mathbb{P}^3$ is $G$-birationally rigid and is not $G$-birationally super-rigid by \cite[Theorem~1.9]{ChSh10a}.
Similarly, if $G\cong\mathfrak{A}_6$, then $\mathbb{P}^3$ is $G$-birationally rigid and is not $G$-birationally super-rigid by \cite[Theorem~1.24]{ChSh09b}.
Moreover, if $G$ contains $\mathfrak{A}_6$ as a proper subgroup, then $\mathbb{P}^3$ is $G$-birationally super-rigid by Lemma~\ref{lemma:rigidity-A6}.

Recall from \cite[Chapter~VII]{Blichfeldt1917} that up to conjugation all remaining finite primitive subgroups of  $\mathrm{PGL}_4(\mathbb{C})$
belongs to one of the following two classes.
The first class consists of primitive groups that leave a smooth quadric surface in $\mathbb{P}^3$ invariant and contain~$\overline{G}_{144}$ (see \cite[\S121]{Blichfeldt1917}).
The second class consists of primitive groups containing an imprimitive normal subgroup $\mumu_2^4$ (see \cite[\S124]{Blichfeldt1917}).
These classes overlap.
For instance, the group~\mbox{$\overline{G}_{144}\cong\mumu_2^4\rtimes(\mumu_3\times\mumu_3)$} itself is contained in both of them (see \cite[\S125]{Blichfeldt1917}).
However, if $G$ is in the second class and is not in the first class, then $G$ must contain the subgroup~$\overline{G}_{80}$ (see \cite[\S124]{Blichfeldt1917}).
Therefore, to complete the proof we may assume that $G$ contains either $\overline{G}_{144}$ or $\overline{G}_{80}$.

If the group $G$ contains $\overline{G}_{144}$ as a subgroup, then the projective space $\mathbb{P}^3$ is $G$-birationally super-rigid by Corollary~\ref{corollary:144}.
If $G=\overline{G}_{80}$ or $G=\overline{G}_{160}$,
then $\mathbb{P}^3$ is $G$-birationally rigid by Corollary~\ref{corollary:80-160},
and it is not $G$-birationally super-rigid by Proposition~\ref{proposition:link-symmetric}.
Finally, if $G$ contains $\overline{G}_{80}$ and neither $G=\overline{G}_{80}$ nor~\mbox{$G=\overline{G}_{160}$},
then $\mathbb{P}^3$ is $G$-birationally super-rigid by Lemma~\ref{lemma:rigidity-80}.
\end{proof}

Thus, to complete the proofs of our main results, it is enough to prove Propositions~\ref{proposition:main-144} and \ref{proposition:main-80}.
This is what we are going to do in the remaining part of this section.
Namely, suppose that either $G=\overline{G}_{144}$ or $G=\overline{G}_{80}$.
Let $\mathcal{D}$ be a $G$-invariant mobile linear system on $\mathbb{P}^3$.
Then $\mathcal{D}\sim\mathcal{O}_{\mathbb{P}^3}(n)$ for a positive integer $n$.
If $G=\overline{G}_{80}$, we also suppose that
\begin{equation}
\label{equation:C8-mult}
\mathrm{max}\Big\{\mathrm{mult}_{\mathcal{C}_1}\big(\mathcal{D}\big),\mathrm{mult}_{\mathcal{C}_2}\big(\mathcal{D}\big)\Big\}\leqslant\frac{n}{4}.
\end{equation}
To prove Propositions~\ref{proposition:main-144} and \ref{proposition:main-80}, we have to show that the singularities of the log pair~\mbox{$(\mathbb{P}^3,\frac{4}{n}\mathcal{D})$} are canonical.
Suppose that this is not true.
Let us show that this assumption leads to a contradiction.

\begin{lemma}
\label{lemma:curves-exlusion}
Let $C$ be a $G$-irreducible curve in $\mathbb{P}^3$. Then $\mathrm{mult}_{C}(\mathcal{D})\leqslant\frac{n}{4}$.
\end{lemma}

\begin{proof}
By Lemma~\ref{lemma:root-deg}, we may assume that $\mathrm{deg}(Z)\leqslant 15$.
If $G=\overline{G}_{144}$, the assertion follows from Theorem~\ref{theorem:144-curves}
together with Corollaries~\ref{corollary:144-L4-L6-L12-mult} and \ref{corollary:144-C12-mult}.
If $G=\overline{G}_{80}$, the assertion follows from the classification of $\overline{G}_{80}$-invariant curves of small degree
done in Theorem~\ref{theorem:80-curves} together with Lemma~\ref{lemma:80-L10-mult}, Corollary~\ref{corollary:80-C12-mult},
and our assumption \eqref{equation:C8-mult} on curves of degree $8$.
\end{proof}

Thus, the log pair $(\mathbb{P}^3,\frac{4}{n}\mathcal{D})$ is canonical away from finitely many points in $\mathbb{P}^3$ (see, for instance, \cite[Proposition~5.14]{CoKoSm03}).
On the other hand, this log pair is not canonical at some point $P\in\mathbb{P}^3$ by assumption.
Denote by $\Sigma$ the $G$-orbit of the point $P$.
Let $D_1$ and $D_2$ be two sufficiently general surfaces in $\mathcal{D}$. Then
\begin{equation}
\label{equation:4-n-2-inequality}
\mathrm{mult}_{O}\Big(D_1\cdot D_2\Big)>\frac{n^2}{4}
\end{equation}
for every point $O\in\Sigma$ by \cite[Corollary 3.4]{Co00}.

For every $r\geqslant 2$, denote by $\mathcal{M}_r$ the linear system of surfaces of degree $r$ in $\mathbb{P}^3$ that contain $\Sigma$.
If $\mathcal{M}_r$ is not empty, let $M_r$ be a general surface in $\mathcal{M}_r$.
In this case \eqref{equation:4-n-2-inequality} implies that
$$
rn^2=M_r\cdot D_1\cdot D_2\geqslant\sum_{O\in\Sigma}\mathrm{mult}_{O}\Big(D_1\cdot D_2\Big)>\frac{n^2}{4}|\Sigma|
$$
provided that $M_r$ does not contain any irreducible component of the one-cycle $D_1\cdot D_2$.
Thus, if the base locus of the linear system $\mathcal{M}_r$ does not contain curves, then~\eqref{equation:4-n-2-inequality}
implies the inequality
\begin{equation}
\label{equation:bound-for-orbit}
|\Sigma|<4r.
\end{equation}
We expect that the base locus of the linear system $\mathcal{M}_r$ contains no curves for $r=\lfloor\frac{|\Sigma|}{4}\rfloor$.
Unfortunately, we failed to prove this, so instead we use an alternative approach
that utilizes results obtained in~\S\ref{section:144} and~\S\ref{section:80-160-320}.
Let us start with

\begin{corollary}
\label{corollary:80-16-20}
If $G=\overline{G}_{144}$, then $|\Sigma|\geqslant 16$.
Similarly, if $G=\overline{G}_{80}$, then $|\Sigma|\geqslant 20$.
\end{corollary}

\begin{proof}
If $G=\overline{G}_{144}$ and $|\Sigma|<16$, then $|\Sigma|=12$ by Lemma~\ref{lemma:144-orbits}.
If $G=\overline{G}_{80}$ and $|\Sigma|<20$, then $|\Sigma|=16$ by Lemma~\ref{corollary:80-small-orbits}.
Thus, the base locus of the linear system $\mathcal{M}_3$ contains no curves by Corollaries~\ref{corollary:144-cubics-12-points}, \ref{corollary:80-small-orbits} and \ref{corollary:80-cubics-16-points}.
This gives a contradiction with~\eqref{equation:bound-for-orbit}.
\end{proof}

Recall from Lemma~\ref{lemma:80-144-quadrics} that there is unique $\overline{G}_{144}$-invariant quadric surface in $\mathbb{P}^3$.

\begin{lemma}
\label{lemma:144-Q-points-exlusion}
Suppose that $G=\overline{G}_{144}$.
Let $\mathcal{Q}$ be the unique $G$-invariant quadric surface in $\mathbb{P}^3$.
Then $\Sigma$ is not contained in $\mathcal{Q}$.
\end{lemma}

\begin{proof}
Suppose that $\Sigma$ is contained in $\mathcal{Q}$.
Then the log pair $(\mathbb{P}^3,\mathcal{Q}+\frac{4}{n}\mathcal{D})$ is not log canonical at every point of $\Sigma$.
Thus, the log pair $(\mathcal{Q}, \frac{4}{n}\mathcal{D}\vert_{\mathcal{Q}})$
is not log canonical at every point of $\Sigma$ by Inversion of Adjunction (see \cite[Theorem~5.50]{KoMo98}).
Hence, there is $\mu<\frac{4}{n}$ such that the log pair $(\mathcal{Q},\mu\mathcal{D}\vert_{\mathcal{Q}})$
is not Kawamata log terminal.
By Corollary~\ref{corollary:144-Q-L4-L6-L12-mult}, this log pair is Kawamata log terminal away from finitely many points in $\mathcal{Q}$.

Let $\mathcal{I}$ be the multiplier ideal sheaf of the log pair~$(\mathcal{Q},\mu\mathcal{D}\vert_{\mathcal{Q}})$.
Then the ideal $\mathcal{I}$ defines a (non-empty) zero-dimensional subscheme $\mathcal{L}$ of the quadric $\mathcal{Q}$
such that the support of the subscheme $\mathcal{L}$ contains $\Sigma$.
Moreover, this subscheme is $G$-invariant, because $\mathcal{D}$ is $G$-invariant.
Let $H$ be a plane section of the quadric $\mathcal{Q}$.
Then
$$
h^1\Big(\mathcal{I}\otimes\mathcal{O}_{\mathcal{Q}}\big(2H\big)\Big)=0
$$
by Nadel vanishing theorem (see \cite[Theorem~9.4.8]{La04}).
Now using the exact sequence of sheaves
$$
0\longrightarrow\mathcal{I}\otimes\mathcal{O}_{\mathcal{Q}}\big(2H\big)\longrightarrow
\mathcal{O}_{\mathcal{Q}}\big(2H\big)\longrightarrow\mathcal{O}_{\mathcal{L}}\otimes\mathcal{O}_{\mathcal{Q}}\big(2H\big)\longrightarrow 0,
$$
we see that
$$
\big|\Sigma\big|\leqslant h^0\big(\mathcal{O}_{\mathcal{L}}\big)=h^0\Big(\mathcal{O}_{\mathcal{L}}\otimes\mathcal{O}_{\mathcal{Q}}\big(2H\big)\Big)\leqslant h^0\Big(\mathcal{O}_{\mathcal{Q}}\big(2H\big)\Big)=9,
$$
which is impossible by Lemma~\ref{lemma:144-Q-orbits}.
\end{proof}

\begin{lemma}
\label{lemma:curve-points-exlusion}
Let $C$ be a $G$-irreducible curve in $\mathbb{P}^3$ such that $\mathrm{deg}(C)\leqslant 15$.
Then $\Sigma\not\subset C$.
\end{lemma}

\begin{proof}
Suppose that $\Sigma\subset C$. Let us seek for a contradiction. We have
$$
D_1\cdot D_2=mC+\Delta,
$$
where $m$ is a non-negative integer, and $\Delta$ is an effective one-cycle on $\mathbb{P}^3$ such that $\mathrm{Supp}(\Delta)$ does not contain irreducible components of the curve $C$.
Hence, it follows from \eqref{equation:4-n-2-inequality} that
\begin{equation}
\label{equation:4-n-2-inequality-refined}
\mathrm{mult}_{O}\big(\Delta\big)>\frac{n^2}{4}-m\cdot\mathrm{mult}_O\big(C\big)
\end{equation}
for every point $O\in\Sigma$.
Let $H$ be a plane in $\mathbb{P}^3$. Thus
\begin{equation}
\label{equation:degree-mult}
n^2=H\cdot D_1\cdot D_2=H\cdot\Big(mC+\Delta\Big)=m\cdot\mathrm{deg}(C)+H\cdot\Delta\geqslant m\cdot\mathrm{deg}(C).
\end{equation}

For every $r\geqslant 2$, denote by $\mathcal{B}_r$ the linear system that consists of surfaces of degree $r$ in~$\mathbb{P}^3$ containing $C$.
If the base locus of $\mathcal{B}_r$ does not contain $G$-irreducible curves different from $C$,
then it follows from \eqref{equation:4-n-2-inequality-refined} that
\begin{multline}
\label{equation:m-lower-bound}
r\Big(n^2-m\cdot\mathrm{deg}(C)\Big)=B_r\cdot\Big(D_1\cdot D_2-mC\Big)=\\
=B_r\cdot\Delta\geqslant\sum_{O\in\Sigma}\mathrm{mult}_{O}\big(\Delta\big)>|\Sigma|\Bigg(\frac{n^2}{4}-m\cdot\mathrm{mult}_O\big(C\big)\Bigg),
\end{multline}
where $B_r$ is a sufficiently general surface in the linear system $\mathcal{B}_r$.

Suppose that $G=\overline{G}_{144}$.
By Lemma~\ref{lemma:144-Q-points-exlusion} and Theorem~\ref{theorem:144-curves},
the curve $C$ is a smooth irreducible curve of degree $12$ and genus $13$.
In particular, we have $|\Sigma|\geqslant 24$ by Lemma~\ref{lemma:144-sporadic-genera}.
Moreover, it follows from Corollary~\ref{corollary:144-C12-sextics} that
the base locus of the linear system $\mathcal{B}_6$ does not contain curves different from $C$.
Thus, it follows from \eqref{equation:m-lower-bound} that
$$
6(n^2-12m)>24\Big(\frac{n^2}{4}-m\Big),
$$
so that $0>48m$, which is absurd.

We see that $G=\overline{G}_{80}$.
Let us use the notation of~\S\ref{section:80-160-320}.
By Theorem~\ref{theorem:80-curves}, one of the following cases holds:
\begin{itemize}
\item $\mathrm{deg}(C)=8$, and $C$ is an irreducible smooth curve of genus $5$;

\item $\mathrm{deg}(C)=12$, and $C$ is an irreducible smooth curve of genus $5$;

\item $\mathrm{deg}(C)=10$, and $C$ is one of the curves $\mathcal{L}_{10}$, $\mathcal{L}_{10}^\prime$, $\mathcal{L}_{10}^{\prime\prime}$.
\end{itemize}
Note that in the former two cases one has $|\Sigma|\geqslant 40$ by
Lemma~\ref{lemma:80-sporadic-genera} and Corollary~\ref{corollary:80-16-20}. Also, we know from~\eqref{equation:degree-mult}
that~\mbox{$n^2\geqslant m\cdot\deg(C)>4m$}.
Let us consider the above possibilities case by case.

If $\mathrm{deg}(C)=8$, then the base locus of $\mathcal{B}_4$ does not contain $G$-irreducible curves different from $C$ by Theorem~\ref{theorem:80-curves}.
Hence~\eqref{equation:m-lower-bound} gives
$$
4n^2-32m=4\Big(n^2-8m\Big)>|\Sigma|\Bigg(\frac{n^2}{4}-m\Bigg)\geqslant 40\Bigg(\frac{n^2}{4}-m\Bigg)=10n^2-40m.
$$
This implies that $6n^2<8m$. On the other hand, one has $8m\leqslant n^2$ by~\eqref{equation:degree-mult}, which is absurd.

If $\mathrm{deg}(C)=12$, then the base locus of $\mathcal{B}_6$ does not contain curves different from $C$ by Corollary~\ref{corollary:80-C12-sextics}.
Hence \eqref{equation:m-lower-bound} gives
$$
6n^2-72m=6\Big(n^2-12m\Big)>|\Sigma|\Bigg(\frac{n^2}{4}-m\Bigg)\geqslant 40\Bigg(\frac{n^2}{4}-m\Bigg)=10n^2-40m,
$$
which is absurd.

Thus, we see that  $C$ is one of the curves $\mathcal{L}_{10}$, $\mathcal{L}_{10}^\prime$ or $\mathcal{L}_{10}^{\prime\prime}$.
In particular, by~\eqref{equation:degree-mult}
we have~\mbox{$n^2\geqslant 10m>8m$}.

If $C=\mathcal{L}_{10}^\prime$ or $C=\mathcal{L}_{10}^{\prime\prime}$, then
the base locus of $\mathcal{B}_4$ does not contain $G$-irreducible curves different from $C$ by Lemma~\ref{lemma:80-L10-sextics},
so that Lemma~\ref{lemma:80-L10} and \eqref{equation:m-lower-bound} give
$$
4n^2-40m=4\Big(n^2-10m\Big)>|\Sigma|\Bigg(\frac{n^2}{4}-2m\Bigg)\geqslant 20\Bigg(\frac{n^2}{4}-2m\Bigg)=5n^2-40m,
$$
which is absurd. Thus, neither $C=\mathcal{L}_{10}^\prime$ nor $C=\mathcal{L}_{10}^{\prime\prime}$
(that is, $\Sigma$ is not contained in~$\mathcal{L}_{10}^\prime$ and in~$\mathcal{L}_{10}^{\prime\prime}$).
We see that $C=\mathcal{L}_{10}$.

By Lemma~\ref{lemma:80-L10-sextics}, the base locus of $\mathcal{B}_6$ does not contain $G$-irreducible curves different from $C$.
If $|\Sigma|\ne 20$, then Lemma~\ref{lemma:80-L10} implies that $|\Sigma|\geqslant 40$, so that \eqref{equation:m-lower-bound} gives
$$
6n^2-60m=6\Big(n^2-10m\Big)>|\Sigma|\Bigg(\frac{n^2}{4}-m\Bigg)\geqslant 40\Bigg(\frac{n^2}{4}-m\Bigg)=10n^2-40m,
$$
which is absurd. We see that $C=\mathcal{L}_{10}$ and $|\Sigma|=20$.
Then $\Sigma\subset\mathcal{L}_{10}^\prime\cup\mathcal{L}_{10}^{\prime\prime}$ by Lemma~\ref{lemma:80-L10}.
However, we just proved that $\Sigma$ is not contained in~$\mathcal{L}_{10}^\prime$ and~$\mathcal{L}_{10}^{\prime\prime}$.
\end{proof}

\begin{corollary}
\label{corollary:long-orbit}
One has $|\Sigma|>35$.
\end{corollary}

\begin{proof}
Assume that $|\Sigma|\leqslant 35$. Let us seek for a contradiction.

Suppose that $G=\overline{G}_{144}$. Then $\Sigma$ is not contained in $\mathcal{Q}$ by Lemma~\ref{lemma:144-Q-points-exlusion}.
Thus, $\Sigma$ is one of two $\overline{G}_{144}$-orbits of length $12$ by Lemma~\ref{lemma:144-orbits}.
This contradicts Corollary~\ref{corollary:80-16-20}.

Thus, we see that $G=\overline{G}_{80}$.
Then $|\Sigma|\geqslant 20$ by Corollary~\ref{corollary:80-16-20}, so that $|\Sigma|=20$.
Hence, $\Sigma$ is contained in one of the curves
$\mathcal{L}_{10}$, $\mathcal{L}_{10}^\prime$ or~$\mathcal{L}_{10}^{\prime\prime}$ by
Lemmas~\ref{lemma:80-Sigma-20} and \ref{lemma:80-L10}, which is impossible by Lemma~\ref{lemma:curve-points-exlusion}.
\end{proof}

Now we are going to use the \emph{multiplication by two trick} that was introduced in \cite{ChSh10a,ChSh09b}.
Recall that $P\in\Sigma$ is a center of non-canonical singularities of the log pair $(\mathbb{P}^3,\frac{4}{n}\mathcal{D})$,
so that $\mathrm{mult}_{P}(\mathcal{D})>\frac{n}{4}$ (see \cite[Proposition~5.14]{CoKoSm03}).
This immediately implies that $P$ is a center of non-log canonical singularities of the log pair $(\mathbb{P}^3,\frac{8}{n}\mathcal{D})$.
In particular, the log pair $(\mathbb{P}^3,\frac{8}{n}\mathcal{D})$ is not log canonical at $P$.
Thus, there is a positive rational number~\mbox{$\mu<\frac{8}{n}$}
such that $(\mathbb{P}^3,\mu\mathcal{D})$ is log canonical at the point $P$,
and it is not Kawamata log terminal at the point $P$.
Of course, the same holds for every point in~$\Sigma$, because $\Sigma$ is the $G$-orbit of the point $P$,
and the linear system $\mathcal{D}$ is $G$-invariant.

\begin{lemma}
\label{lemma:isolation}
The log pair $(\mathbb{P}^3,\mu\mathcal{D})$ is Kawamata log terminal in a punctured neighborhood of every point in $\Sigma$.
\end{lemma}

\begin{proof}
Suppose that the required assertion is not true.
Then there exists a $G$-irreducible curve $C$ in $\mathbb{P}^3$ such that $\Sigma\subset C$ and
$(\mathbb{P}^3,\mu\mathcal{D})$ is not Kawamata log terminal in every point of the curve $C$.
Let $D_1$ and $D_2$ be two general surfaces in $\mathcal{D}$.
Then
$$
D_1\cdot D_1=mC+\Delta,
$$
where $m$ is a positive integer, and $\Delta$ is an effective one-cycle on $\mathbb{P}^3$ such that $\mathrm{Supp}(\Delta)$ does not contain irreducible components of the curve $C$.
By \cite[Theorem~3.1]{Co00}, one has
$$
m\geqslant\frac{4}{\mu^2}>\frac{n^2}{16}.
$$
Let $H$ be a general plane in $\mathbb{P}^3$. Then
$$
n^2=H\cdot D_1\cdot D_2=H\cdot\Big(mC+\Delta\Big)=m\cdot\mathrm{deg}(C)+H\cdot\Delta\geqslant m\cdot\mathrm{deg}(C)>\frac{n^2}{16}\cdot\mathrm{deg}(C),
$$
so that $\mathrm{deg}(C)\leqslant 15$, which is impossible by Lemma~\ref{lemma:curve-points-exlusion}.
\end{proof}

Let $\mathcal{I}$ be the multiplier ideal sheaf of the log pair~$(\mathbb{P}^3,\mu\mathcal{D})$.
Then the ideal $\mathcal{I}$ is not trivial, since $(\mathbb{P}^3,\mu\mathcal{D})$ is not Kawamata log terminal.
Thus, it defines a (non-empty) subscheme $\mathcal{L}$ of $\mathbb{P}^3$ such that the support of the subscheme $\mathcal{L}$ contains $\Sigma$.
Since $\mathcal{D}$ is mobile, the subscheme $\mathcal{L}$ has no two-dimensional components.
Moreover, the subscheme~$\mathcal{L}$ is reduced in a neighborhood of every point in $\Sigma$,
because the log pair $(\mathbb{P}^3,\mu\mathcal{D})$ is log canonical at every point of $\Sigma$.
Furthermore, every point of $\Sigma$ is an isolated irreducible component of the subscheme $\mathcal{L}$ by Lemma~\ref{lemma:isolation}.
On the other hand, it follows from Nadel vanishing theorem (see \cite[Theorem~9.4.8]{La04}) that
$$
h^1\Big(\mathcal{I}\otimes\mathcal{O}_{\mathbb{P}^3}\big(4\big)\Big)=0.
$$
Now using the exact sequence of sheaves
$$
0\longrightarrow\mathcal{I}\otimes\mathcal{O}_{\mathbb{P}^3}\big(4\big)\longrightarrow\mathcal{O}_{\mathbb{P}^3}\big(4\big)\longrightarrow\mathcal{O}_{\mathcal{L}}\otimes\mathcal{O}_{\mathbb{P}^3}\big(4\big)\longrightarrow 0,
$$
we see that
$$
|\Sigma|\leqslant h^0\Big(\mathcal{O}_{\mathcal{L}}\otimes\mathcal{O}_{\mathbb{P}^3}\big(4\big)\Big)\leqslant h^0\Big(\mathcal{O}_{\mathbb{P}^3}\big(4\big)\Big)=35,
$$
which is impossible by Corollary~\ref{corollary:long-orbit}.
The obtained contradiction completes the proof of Propositions~\ref{proposition:main-144} and \ref{proposition:main-80}.
This in turns completes the proof of Theorems~\ref{theorem:main} and \ref{theorem:super-main}.

\appendix

\section{Finite primitive groups}
\label{section:diagram}

In diagrams~\eqref{figure:Blichfeldt1},
\eqref{figure:Blichfeldt2}, and~\eqref{figure:Blichfeldt3} below,
we present all the finite primitive subgroups
of $\mathrm{PGL}_4(\mathbb{C})$ together with
some inclusions between them, including those  that
are necessary for the proof of Theorem~\ref{theorem:main}.
We emphasize that we will not need the whole
classification for the proof, but only the information about
the few smallest groups.

By $\mumu_n$ we denote the cyclic group of order~$n$.
The inclusions are depicted by arrows going from a smaller group
to a larger one.
We point out that there are indeed several additional
inclusions (apart from those that are obtained merely by
composing those inclusions that we list). For instance,
the group $\A_4\times\A_4$ from diagram~\eqref{figure:Blichfeldt1}
and the group $\A_6$ from diagram~\eqref{figure:Blichfeldt3}
are both subgroups of $\mumu_2^4\rtimes\A_6$
from diagram~\eqref{figure:Blichfeldt2}.

In diagram~\eqref{figure:Blichfeldt1}, we list all finite
primitive subgroups
of $\mathrm{PGL}_4(\mathbb{C})$ that leave invariant a quadric surface.
The larger ``connected component'' of the diagram contains
the groups~\mbox{$1^\circ$--$12^\circ$} described in
\cite[\S\S121,122]{Blichfeldt1917}; the smaller ``connected component''
contains
the group $(B)$ described in
\cite[\S102]{Blichfeldt1917} and the group $(H)$ described
in \cite[\S119]{Blichfeldt1917}.
We denote by~\mbox{$\widehat{G\times G}$} the group generated by $G\times G$
and an involution that exchanges the factors of~$G\times G$.
This group is isomorphic to~\mbox{$(G\times G)\rtimes\mumu_2$}.
However, we reserve
the notation~\mbox{$(\A_4\times\A_4)\rtimes\mumu_2$} not for
a group of the latter type, but for a different subgroup
of~\mbox{$\mathrm{PGL}_4(\mathbb{C})$} generated by~$\A_4$
and an element of order
$2$ that \emph{does not} exchange the factors; this element can be thought of
as a product of two elementary transpositions in the factors
of~\mbox{$\SS_4\times\SS_4\supset\A_4\times\A_4$}. We refer the reader to
\cite[\S121]{Blichfeldt1917} for details. The above group
is contained as a subgroup of index two in two other
primitive subgroups of $\mathrm{PGL}_4(\mathbb{C})$ preserving a quadric
surface so that in both cases
an additional element of order $2$ exchanges
the factors of~\mbox{$\A_4\times\A_4\subset (\A_4\times\A_4)\rtimes\mumu_2$},
see \cite[\S122]{Blichfeldt1917} for details.
We denote these two groups
by~\mbox{$\widehat{(\A_4\times\A_4)\rtimes\mumu_2^{(1)}}$}
and~\mbox{$\widehat{(\A_4\times\A_4)\rtimes\mumu_2^{(2)}}$}.

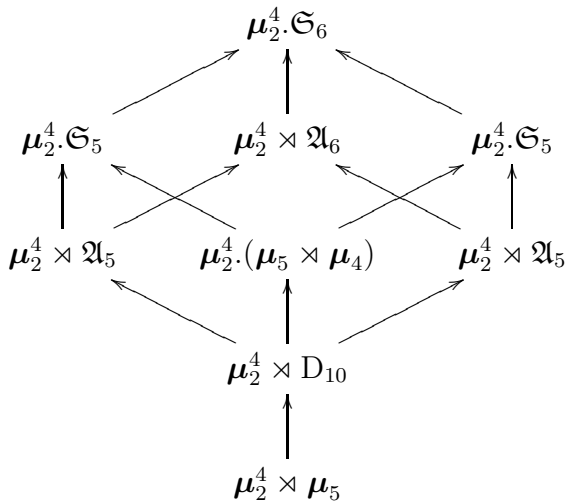
\begin{equation}
\label{figure:Blichfeldt1}
\xymatrix{
\widehat{\SS_4\times\SS_4} &  \widehat{\A_5\times\A_5}  &  \\
& \SS_4\times\SS_4\ar@{->}[lu] &
\SS_4\times\A_5 & \A_5\times\A_5\ar@{->}[llu] &  \\
\widehat{(\A_4\times\A_4)\rtimes\mumu_2^{(1)}} \ar@{->}[uu] &
\widehat{(\A_4\times\A_4)\rtimes\mumu_2^{(2)}} \ar@{->}[luu]
&  \A_4\times\SS_4\ar@{->}[ul]\ar@{->}[u]  &
\A_4\times\A_5\ar@{->}[lu]\ar@{->}[u]  & \SS_5\\
(\A_4\times\A_4)\rtimes\mumu_2\ar@{->}[ruu]\ar@{->}[ru]\ar@{->}[u] & &
\widehat{\A_4\times\A_4}\ar@/^/@{->}[llu]\ar@{->}[lu]\ar@{->}[luuu]
& & \A_5\ar@{->}[u] \\
& & & \A_4\times\A_4\ar@{->}[lllu]\ar@{->}[uu]\ar@{->}[luu]\ar@{->}[lu] &
}
\end{equation}

In diagram~\eqref{figure:Blichfeldt2}, we list all finite
primitive subgroups
of $\mathrm{PGL}_4(\mathbb{C})$ that
contain a subgroup isomorphic to $\mumu_2^4$ except for those that
leave invariant a quadric surface.
These are the groups~\mbox{$13^\circ$--$21^\circ$} described in
\cite[\S124]{Blichfeldt1917}.
We denote by $\mathrm{D}_{10}$ the dihedral group of order~$10$.
By~$F.G$ we mean a non-split extension of a group $G$ by a group~$F$.

\begin{equation}
\label{figure:Blichfeldt2}
\xymatrix{
& \mumu_2^4.\SS_6 & \\
\mumu_2^4.\SS_5\ar@{->}[ru] &
\mumu_2^4\rtimes\A_6\ar@{->}[u] &
\mumu_2^4.\SS_5\ar@{->}[lu]\\
\mumu_2^4\rtimes\A_5\ar@{->}[u]\ar@{->}[ru] &
\mumu_2^4.(\mumu_5\rtimes\mumu_4)\ar@{->}[lu]\ar@{->}[ru] &
\mumu_2^4\rtimes \A_5\ar@{->}[u]\ar@{->}[lu]\\
& \mumu_2^4\rtimes\mathrm{D}_{10}\ar@{->}[lu]\ar@{->}[ru]\ar@{->}[u] & \\
& \mumu_2^4\rtimes\mumu_5\ar@{->}[u] &
}
\end{equation}

Finally, in diagram~\eqref{figure:Blichfeldt3}, we list the remaining finite
primitive subgroups
of~$\mathrm{PGL}_4(\mathbb{C})$.
These are the groups~$(A)$, $(C)$, $(D)$, $(E)$, and~$(F)$ described in
\cite[\S102]{Blichfeldt1917},
and the groups~$(G)$ and~$(K)$ described in
\cite[\S119]{Blichfeldt1917}.

\begin{equation}
\label{figure:Blichfeldt3}
\xymatrix{
& \mathrm{PSp}_{4}(\mathbf{F}_3) & & & \\
& \SS_6\ar@{->}[u] & & \A_7 & \\
\SS_5\ar@{->}[ru] & & \A_6\ar@{->}[lu]\ar@{->}[ru] & &
\mathrm{PSL}_2(\mathbf{F}_7)\ar@{->}[lu]\\
&\A_5\ar@{->}[lu]\ar@{->}[ru] & & &
}
\end{equation}

\end{document}